\documentclass[11pt,a4paper]{amsart}
\pdfoutput=1
\usepackage[margin=2.5cm]{geometry} % - Control margins
\usepackage{amssymb}
\usepackage{amsmath}
\usepackage{mathtools}
\usepackage{amsthm}
\usepackage{amsfonts}
\usepackage{kpfonts} % - Nice overall font style
\usepackage{bbm} % - allows for \mathbbm fonts
\usepackage{dutchcal} % - mathcal font style
\usepackage[mathscr]{euscript} % - mathscr font style
\usepackage{tikz-cd} % - For commutative diagrams
\usepackage{todonotes} % - For nice informal todo notes
\usepackage{graphicx} % - Required for inserting images
\usepackage[toc,page]{appendix} % - used for making appendices
\usepackage{xcolor} % - for colors
\usepackage{float} % - force a figure to stay put
\usepackage{enumitem} % - more detailed options for enumerate
\usepackage{comment} % - for long comments
\usepackage[normalem]{ulem} %for strikethrough e.g. \sout{blah}

\usepackage{url}
\usepackage[style = alphabetic,
            backref = true,
            backrefstyle = two]{biblatex}
\renewcommand\mathfrak[1]{\mbox{\usefont{U}{euf}{m}{n}#1}}

\definecolor{winered}{rgb}{0.5,0,0}
\usepackage[citecolor = winered, 
            urlcolor = winered, 
            linkcolor = winered,
            menucolor = winered, 
            colorlinks = true]{hyperref}

\DefineBibliographyStrings{english}{%
  backrefpage = {p.}, % - originally "cit. on p.", changed to "p."
  backrefpages = {pp.}, % - originally "cit. on pp.", changed to "pp."
}

\newtheoremstyle{theoremdd}
{\topsep}{\topsep}{\upshape}{0pt}{\bfseries}{.}{ }{\thmname{#1}\thmnumber{ #2}\thmnote{ (#3)}}

\theoremstyle{definition}
\newtheorem{Th}{Theorem}[section]
\newtheorem{Lemma}[Th]{Lemma}
\newtheorem{Cor}[Th]{Corollary}
\newtheorem{Prop}[Th]{Proposition}
\newtheorem{Def}[Th]{Definition}

\newtheorem{Rem}[Th]{Remark}
\newtheorem{Ex}[Th]{Example}

\newcommand{\Hom}{\text{Hom}}
\newcommand{\C}{\mathbb{C}}
\newcommand{\R}{\mathbb{R}}

\newcommand{\ncat}{\mathbf} % - named categories
\newcommand{\cat}{\mathcal} % - unnamed categories
\newcommand{\cons}{\text} % - constructions
\renewcommand{\u}{\underline}
\newcommand{\colim}{\text{colim}}
\newcommand{\ncolim}[1]{\underset{#1}{\colim} \,}

\newcommand{\op}{\text{op}} 

\newcommand{\Pre}{\ncat{Pre}}

\newcommand{\Sh}{\ncat{Sh}}

\newcommand{\sat}[1]{\text{sat}(#1)}

\newcommand{\Map}{\text{Map}}

\newcommand{\Gro}{\text{Gro}}

\newcommand{\nd}{\text{nd}}

\newcommand{\sPre}{\ncat{sPre}}

\newcommand{\B}{\mathbf{B}}

\renewcommand{\H}{\mathbb{H}}

\newcommand{\hocolim}{\text{ho}\colim}
\newcommand{\nhocolim}[1]{\underset{#1}{\hocolim} \,}
\newcommand{\closed}{\mathscr{C}}
\newcommand{\cosk}{\text{cosk}}

\makeatletter
\newtheorem*{rep@theorem}{\rep@title}
\newcommand{\newreptheorem}[2]{%
\newenvironment{rep#1}[1]{%
 \def\rep@title{#2 \ref{##1}}%
 \begin{rep@theorem}}%
 {\end{rep@theorem}}}
\makeatother

\newreptheorem{theorem}{Theorem}

\newreptheorem{lemma}{Lemma}

\newreptheorem{cor}{Corollary}
\newreptheorem{prop}{Proposition}

\usetikzlibrary{calc}
\usetikzlibrary{decorations.pathmorphing}

\tikzset{curve/.style={settings={#1},to path={(\tikztostart)
    .. controls ($(\tikztostart)!\pv{pos}!(\tikztotarget)!\pv{height}!270:(\tikztotarget)$)
    and ($(\tikztostart)!1-\pv{pos}!(\tikztotarget)!\pv{height}!270:(\tikztotarget)$)
    .. (\tikztotarget)\tikztonodes}},
    settings/.code={\tikzset{quiver/.cd,#1}
        \def\pv##1{\pgfkeysvalueof{/tikz/quiver/##1}}},
    quiver/.cd,pos/.initial=0.35,height/.initial=0}

\tikzset{tail reversed/.code={\pgfsetarrowsstart{tikzcd to}}}
\tikzset{2tail/.code={\pgfsetarrowsstart{Implies[reversed]}}}
\tikzset{2tail reversed/.code={\pgfsetarrowsstart{Implies}}}

\definecolor{emilioeditcolor}{rgb}{0.94, 0.97, 1.0}
\newcommand{\emilio}[1]{\todo[inline,color=emilioeditcolor, size=\small]{{\bf E:} #1}}

\definecolor{cheyneeditcolor}{rgb}{1.0, 0.71, 0.76}

\title{Hypercovers in Differential Geometry}

\usepackage[foot]{amsaddr}
\author{Cheyne Glass$^1$}
\address{$^1$SUNY New Paltz}

\author{Emilio Minichiello$^2$}
\address{$^2$CUNY CityTech}
\email{glassc(at)newpaltz.edu, eminichiello67(at)gmail.com}

\date{Last compilation: \today}
\begin{document}

\begin{abstract}
In this paper we provide a simple proof that for several sites of interest in differential geometry, the local projective model structure and the \v{C}ech projective model structure are equal. In particular, this applies to the site of smooth manifolds with open covers and the site of cartesian spaces with good open covers. As an application, we show that for a presheaf of sets on these sites, applying the plus construction once is enough to sheafify.
\end{abstract}

\maketitle

\setcounter{tocdepth}{1}
\tableofcontents

\section{Introduction}

Higher sheaf theory is the study of sheaves of homotopy types, rather than sheaves of sets or abelian groups. While many applications of higher sheaf theory are found in algebraic geometry and homotopy theory, there is a growing body of work not limited to the following examples \cite{Moerdijk1997, Larusson2002, Carchedi2011, Fiorenza2011, Bunke2013, Nikolaus2014a, Hopkins2015, Amabel2021, Bunk2021, Clough2021, Bunk2022, Glass2022, Grady2022, Pavlov2022, Minichiello2024, Minichiello2024Obstruction} that applies higher sheaf theory to differential geometry. This burgeoning field is a subset of \textbf{higher differential geometry}, a relatively new subject which uses higher categorical structures to study differential geometric objects, with motivation mainly coming from physics.

Much of this theory can be studied using $\infty$-categories, and it benefits immensely from the elegance and ease which comes from using this formalism. However, in order to obtain cocycle data from higher sheaves on manifolds, one must use an underlying model category which can ``transmit'' information between the differential geometric world and the homotopical world.

For the moment, let us fix the site $(\ncat{Man}, j_{\text{open}})$ of finite dimensional smooth manifolds with the coverage of open covers, and let $\ncat{sPre}(\ncat{Man})$ denote the category of simplicial presheaves on $\ncat{Man}$, namely functors $ X: \ncat{Man}^\op \to \ncat{sSet}$, with morphisms natural transformations. In this paper, we are concerned with various model structures on $\ncat{sPre}(\ncat{Man})$. Let $\H(\ncat{Man})$ denote\footnote{This notation was inspired by Urs Schreiber's notation for $\infty$-toposes \cite{Schreiber2013}.} the category of simplicial presheaves equipped with the \textbf{projective model structure}, where weak equivalences and fibrations are defined objectwise. The \textbf{local projective model structure} on simplicial presheaves, denoted $\hat{\H}(\ncat{Man}, j_{\text{open}})$, has as its weak equivalences the $j_{\text{open}}$-\textbf{local weak equivalences}, which blend the homotopical constraints of a weak equivalence with locality coming from the site structure.

This idea of local weak equivalence (which can be defined for simplicial presheaves over arbitrary sites) has a long history. Perhaps the first reference is Joyal's letter to Grothendieck \cite{Joyal1984}, where he develops this notion for simplicial sheaves. The motivation for this notion was to develop a homotopy theory of $\infty$-stacks---a generalization of the stacks of groupoids used in algebraic geometry---the philosophy of which is laid out in Grothendieck's Pursuing Stacks \cite{Grothendieck1983}. It was believed\footnote{See the bottom of page 2 in \cite{Dugger2004}.} that the fibrant objects in the local projective model structure were those objectwise fibrant simplicial presheaves that satisfied the following property, called \textbf{\v{C}ech descent}: if $X$ is an objectwise fibrant simplicial presheaf, and $\mathcal{U} = \{U_i \subseteq M \}_{i \in I}$ is an open cover of $M$, then the canonical map
\begin{equation*}
X(M) \to \text{holim}_{\mathbf{\Delta}} \left( \prod_i X(U_i) \, \substack{\leftarrow\\[-1em] \leftarrow} \, \prod_{ij} X(U_{ij}) \,  \substack{\leftarrow\\[-1em] \leftarrow \\[-1em] \leftarrow} \, \prod_{ijk} X(U_{ijk}) \, \dots \right)
\end{equation*}
is a weak equivalence. This is equivalent to the map of simplicial sets
\begin{equation} \label{eq equiv def of cech descent}
    X(M) \cong \u{\sPre}(\ncat{Man})(y(M), X) \to \u{\sPre}(\ncat{Man})(\check{C}(\mathcal{U}), X)
\end{equation}
induced by the canonical map $\check{C}(\mathcal{U}) \to y(M)$ of simplicial presheaves being a weak equivalence, where $\u{\ncat{sPre}}(\cat{C})(X,Y)$ denotes the simplicially enriched Hom in the category of simplicial presheaves over a category $\cat{C}$ and $\check{C}(\mathcal{U})$ is the \v{C}ech nerve of the cover $\mathcal{U}$.

If $X$ is a discrete simplicial presheaf, i.e. a presheaf of sets, then $X$ satisfies \v{C}ech descent if and only if $X$ is a sheaf of sets, and a similar statement holds for stacks of groupoids \cite[Section 5]{Minichiello2024}. Hence a simplicial presheaf satisfying \v{C}ech descent really deserves to be called an $\infty$-stack.

However, Dugger-Hollander-Isaksen proved \cite[Theorem 1.2]{Dugger2004} that the fibrant objects in the local projective model structure are \textit{not} those simplicial presheaves satisfying \v{C}ech descent. Rather, they are those (objectwise fibrant) simplicial presheaves that satisfy a stronger property called \textbf{hyperdescent}. Rather than requiring that $X$ be local (in the sense of (\ref{eq equiv def of cech descent}) being a weak equivalence) with respect to maps $\check{C}(\mathcal{U}) \to y(M)$, hyperdescent requires that $X$ be local with respect to all maps $H \to y(M)$ which are \textbf{hypercovers}. Intuitively, one can think of a hypercover as consisting of an open cover $\mathcal{U} = \{U_i \subseteq M \}$ of $M$, and for every nonempty double intersection $U_{ij}$ an open cover $\mathcal{V}^{ij} = \{V^{ij}_\alpha \subseteq U_{ij} \}_{\alpha \in I^{ij}}$, and for every triple intersection $V^{ij}_\alpha \cap V^{jk}_\beta \cap V^{ik}_\gamma$ a cover, etc. Defining hypercovers rigorously takes some work, and they can be daunting objects to deal with.

To obtain the more intuitive \v{C}ech descent, a model structure is constructed by Dugger-Hollander-Isaksen in \cite[Appendix A]{Dugger2004}, called the \textbf{\v{C}ech projective model structure}, denoted $\check{\H}(\ncat{Man}, j_{\text{open}})$. It is the left Bousfield localization of $\H(\ncat{Man})$ at the projection maps $\check{C}(\mathcal{U}) \to y(M)$, where $\mathcal{U}$ is an open cover of a manifold $M$. The fibrant objects of this model structure are precisely those objectwise fibrant simplicial presheaves that satisfy \v{C}ech descent. But now the weak equivalences of this model structure are mysterious, compared to the intuitive and useful local weak equivalences of the local projective model structure $\hat{\H}(\ncat{Man}, j_{\text{open}})$. 

In general, the \v{C}ech and local projective model structures are not Quillen equivalent \cite[Example A.10]{Dugger2004}. If $(\cat{C}, j)$ is a site where If $\hat{\H}(\cat{C}, j)$ and $\check{\H}(\cat{C}, j)$ are Quillen equivalent, then we say that $(\cat{C}, j)$ is \textbf{hypercomplete}. If the two model structures actually coincide $\hat{\H}(\cat{C}, j) = \check{\H}(\cat{C}, j)$, then we say that $(\cat{C}, j)$ is \textbf{strictly hypercomplete}. In cases of strict hypercompleteness, one has the best of both model categories: the fibrant objects are those objectwise fibrant simplicial presheaves satisfying \v{C}ech descent, and the weak equivalences are just the local weak equivalences\footnote{Furthermore the cofibrations are relatively easy to understand, since they agree with those of the projective model structure, which is cofibrantly generated.}. This makes working with simplicial presheaves over strictly hypercomplete sites easier.

In an unfinished set of notes, Dugger claims \cite[Proposition 3.4.8]{Dugger1998} that $(\ncat{Man}, j_{\text{open}})$ is strictly hypercomplete. He even refers to this as an ``important result,'' which is ``not an easy proposition'' \cite[Page 31]{Dugger1998}. However, no proof is provided, and the notes remain unfinished two decades later. 

In this paper, we prove Dugger's claim, and extend it to a variety of sites relevant to differential geometry:

\begin{reptheorem}{th many sites are strictly hypercomplete}
Let $X$ be a paracompact Hausdorff topological space with finite covering dimension. Then all of the sites from Example \ref{ex sites} are strictly hypercomplete:
\begin{equation*}
    (\ncat{TopMan}, j_{\text{open}}), \quad (\mathcal{O}(X), j_X), \quad (\ncat{Man}, j_{\text{open}}), \quad (\ncat{Cart}, j_{\text{open}/\text{good}}), \quad (\C\ncat{Man}, j_{\text{open}}), \quad (\ncat{Stein}, j_{\text{open}}).
\end{equation*}
\end{reptheorem}

Our proof relies on a result of Jacob Lurie \cite{Lurie2009}, which we refer to as \textbf{Lurie's Lemma}, that shows how one can refine each level of a hypercover by an actual open cover. It is a folklore result that $(\ncat{Cart}, j_{\text{open}})$ and $(\ncat{Man}, j_{\text{open}})$ are hypercomplete \cite{Hoy2013}. However, the proofs given by Hoyois and Carchedi are explicitly $\infty$-categorical, and furthermore rely on somewhat advanced $\infty$-categorical machinery (postnikov towers and homotopy dimension of $\infty$-toposes). Our proof is short and simple, using only the barest of model categorical machinery. Furthermore, since we are working directly with the model categories, our proof provides strict hypercompleteness.

Next we turn to the issue of how one computes fibrant replacement in $\check{\H}(\ncat{Man}, j_{\text{open}})$. Such a construction is analogous to $\infty$-sheafification, and is very important to the study of higher differential geometry. Many higher sheaves in higher differential geometry are constructed by defining a structure that is locally trivial, and then taking some sort of ``stackification'' to obtain a nontrivial global structure. Perhaps the best example of this is the construction of abelian bundle gerbes by Nikolaus and Schweigert in \cite[Section 4]{Nikolaus2011}, also see \cite[Remark 3.2.2]{Fiorenza2011}. It turns out that a fibrant replacement functor for the local projective model structure has already been constructed by Zhen Lin Low in \cite{Low2015}. Given a locally fibrant simplicial presheaf $A$ on a site $(\cat{C}, j)$ with $U \in \cat{C}$, if we let $\hat{A}$ denote a fibrant replacement of $A$ in $\hat{\H}(\cat{C}, j)$, then there is a weak equivalence of simplicial sets
\begin{equation} \label{eq low's formula intro}
   \hat{A}(U) \simeq \nhocolim{H \in \ncat{Hyp}^\op_U} \u{\ncat{sPre}}(\cat{C})(H, A),
\end{equation}
where $\ncat{Hyp}_U$ denotes the category of hypercovers over $y(U)$. By combining the above construction with Theorem \ref{th many sites are strictly hypercomplete}, we obtain a fibrant replacement functor for $\check{\H}(\ncat{Man}, j_{\text{open}})$. Note that Low's formula (\ref{eq low's formula intro}) is a strengthening of the well-known Verdier Hypercovering Theorem, which in modern form states that given locally fibrant simplicial presheaves $X$ and $A$, there is an isomorphism
\begin{equation} \label{eq verdier hypercovering theorem intro}
  \ncolim{H\in [\ncat{Triv}_X^\op]} \pi_0 \u{\ncat{sPre}}(\cat{C})(H,A) \cong [X,A],
\end{equation}
where $[X,A]$ denotes the set of morphisms from $X$ to $A$ in the homotopy category $\text{ho} \hat{\H}(\cat{C}, j)$, and $[\ncat{Triv}_X]$ denotes the category whose objects are local trivial fibrations over $X$ and whose morphisms are homotopy classes of maps between them.

Given a presheaf of sets $A$ on a site $(\cat{C}, j)$, let $A^+$ denote the plus construction. Typically one needs to apply the plus construction twice to a presheaf over a site in order to obtain a sheaf, see \cite[Section V.3]{MacLane2012}. However, the plus construction can be seen as a special case of applying fibrant replacement to a discrete simplicial presheaf, which gives the following result.

\begin{reptheorem}{th plus construction is sheafification}
Given a presheaf of sets $A$ over any of the sites mentioned in Theorem \ref{th many sites are strictly hypercomplete}, the presheaf $A^+$ is a sheaf.
\end{reptheorem}

In Section \ref{section sites and sheaves} we review the basic notions of topos theory: sites and sheaves, and give the examples of sites we are interested in for this paper. In Section \ref{section simplicial presheaves} we review the machinery of simplicial presheaves, their various model structures and describe different notions of hypercovers. In Section \ref{section strictly hypercomplete}, we prove Theorem \ref{th many sites are strictly hypercomplete}. In Section \ref{section plus construction} we prove Theorem \ref{th plus construction is sheafification}. In Section \ref{section applications} we discuss some applications of our main results. We show how strict hypercompleteness is useful for deriving cocycle constructions in higher differential geometry, for obtaining isomorphisms between \v{C}ech and sheaf cohomology, and for proving Verdier's Hypercovering theorem.

\subsection*{Note}
For the preprint version of this paper, we have decided to include significantly more expository material than is strictly necessary for the main results Theorem \ref{th many sites are strictly hypercomplete} and Theorem \ref{th plus construction is sheafification}. This is because we found the literature on this material to be a bit scattered and confusing. Since internet ink is free, we figured we would add this material to this preprint version.

This includes an extended section (Section \ref{section examples of sites}) with many examples of sites and their properties, and many appendices which provide more detail on results used in the paper, along with an appendix on the winding history of hypercovers. For similar reasoning, we also provide in this version some extra partial results which are somewhat tangential to the main results of this paper: Theorem \ref{th truncated result} and Theorem \ref{prop can refine hypercovers by Verdier hypercovers on matching Verdier site}.

In Appendix \ref{section augmented and truncated simplicial presheaves} we review the theory of augmented and truncated simplicial presheaves. In Appendix \ref{section verdier sites} we give some background on Verdier sites and prove Proposition \ref{prop can refine DHI-hypercover by verdier hypercover on Verdier site}. In Appendix \ref{section matching verdier sites} we discuss some partial results we obtained in an attempt to loosen the requirements on Verdier sites. In Appendix \ref{section luries lemma} we give a detailed proof of Lurie's Lemma \ref{lurie's lemma}. In Appendix \ref{section homotopical categories} we review some of the theory of homotopical categories, especially categories of fibrant objects, that is needed in Section \ref{section plus construction}. Finally in Appendix \ref{section history} we review some of the history of hypercovers.

The version of this paper we submit for publication will consist of only the strictly necessary material to present the new theorems and will be significantly shorter.

\subsection*{Acknowledgements}

The authors would like to thank Mahmoud Zeinalian, Matthew Cushman, Tim Hosgood, Severin Bunk and Dmitri Pavlov for various helpful discussions. The second author was supported by PSC-CUNY grant TRADB 56-13.

\section{Sites and Sheaves} \label{section sites and sheaves}

In this section we quickly overview the basics of the theory of sites and sheaves. For more details, see \cite[Section 2]{Minichiello2025}.

\begin{Def} \label{def family of morphisms}
Let $\cat{C}$ be a category, and $U \in \cat{C}$. A \textbf{family of morphisms} over $U$ is a set of morphisms $r = \{r_i : U_i \to U \}_{i \in I}$ in $\cat{C}$ with codomain $U$. A \textbf{refinement} of a family of morphisms $t = \{t_j : V_j \to U \}_{j \in J}$ over $U$ consists of a family of morphisms $r = \{r_i : U_i \to U \}_{i \in I}$, a function $\alpha: I \to J$ and for each $i \in I$ a map $f_i: U_i \to V_{\alpha(i)}$, which we call the $i$ \textbf{component} of $f$, making the following diagram commute:
\begin{equation*}
    \begin{tikzcd}
	{U_i} && {V_{\alpha(i)}} \\
	& U
	\arrow["{f_i}", from=1-1, to=1-3]
	\arrow["{r_i}"', from=1-1, to=2-2]
	\arrow["{t_{\alpha(i)}}", from=1-3, to=2-2]
\end{tikzcd}
\end{equation*}
If $r$ is a refinement of $t$, with maps $f_i: U_i \to V_{\alpha(i)}$, then we write $f : r \to t$. We say that $r$ refines $t$, and write $r \leq t$ if there exists a refinement $f : r \to t$.
\end{Def}

\begin{Def} \label{def collection of families}
A \textbf{collection of families} $j$ on a category $\cat{C}$ consists of a set $j(U)$ for each $U \in \cat{C}$, whose elements $r \in j(U)$ are families of morphisms over $U$. We write $r \in j$ to mean that there exists some $U \in \cat{C}$ such that $r \in j(U)$. 

If $j$ and $j'$ are collections of families on a category $\cat{C}$, then we say that $j'$ \textbf{contains} $j$ or $j \subseteq j'$ if for every $U \in \cat{C}$, $j(U) \subseteq j'(U)$. We say that $j$ \textbf{refines} $j'$ and write $j \leq j'$ if for every $r' \in j'$ there exists a $r \in j$ and a refinement $r \leq r'$.
\end{Def}

\begin{Def} \label{def coverage}
We say that a collection of families $j$ on a small category $\cat{C}$ is a \textbf{coverage} if
\begin{enumerate}
    \item for every isomorphism\footnote{This differs from \cite[Definition 2.9]{Minichiello2025}, but by \cite[Lemma 2.26]{Minichiello2025}, in practice this makes no difference.} $f : V \to U$, $\{f : V \to U \} \in j(U)$, and
    \item for every $U \in \cat{C}$, $r \in j(U)$, and $g : V \to U$, there exists a $t \in j(V)$ such that for every $t_j : V_j \to V$ in $t$ there exists a map $r_i : U_i \to U$ in $r$ and a map $s_j: V_j \to U_i$ in $\cat{C}$ making the following diagram commute:
\begin{equation*} \label{eqn coverage def}
    \begin{tikzcd} 
	{V_j} & {U_i} \\
	V & U
	\arrow["{t_j}"', from=1-1, to=2-1]
	\arrow["{s_j}", from=1-1, to=1-2]
	\arrow["{r_i}", from=1-2, to=2-2]
	\arrow["g"', from=2-1, to=2-2]
\end{tikzcd}
\end{equation*}
\end{enumerate}
If $j$ only satisfies (2), then we call it a \textbf{precoverage}. If $j$ is a coverage on $\cat{C}$, then we call families $r \in j(U)$ \textbf{covering families} over $U$. If a map $r_i: U_i \to U$ belongs to a covering family $r \in j(U)$, then we say that $r_i$ is a \textbf{basal map}\footnote{This terminology comes from \cite[Definition 9.1]{Dugger2004}. Many references simply call these covering maps.}. If $\cat{C}$ is a small category, and $j$ is a coverage on $\cat{C}$, then we call the pair $(\cat{C}, j)$ a \textbf{site}. We also think of $j(U)$ as a poset, where $r \leq r'$ in $j(U)$ if both $r$ and $r'$ are covering families over $U$ and there is a refinement $f : r \to r'$.

Given coverages $j, j'$ on $\cat{C}$, we say that $j$ \textbf{refines} $j'$, and write $j \leq j'$ if for every $U \in \cat{C}$ and every $r' \in j'(U)$, there exists a $r \in j(U)$ such that $r \leq r'$. We write $j \subseteq j'$ if $j(U) \subseteq j'(U)$ for every $U \in \cat{C}$.
\end{Def}

\begin{Def} \label{def presheaf, section, matching family, amalgamation}
A \textbf{presheaf} on a category $\cat{C}$ is a functor $X: \cat{C}^{\op} \to \ncat{Set}$. An element $x \in X(U)$ for an object $U \in \cat{C}$ is called a \textbf{section} over $U$. If $f: U \to V$ is a map in $\cat{C}$, and $x \in X(V)$, then we sometimes denote $X(f)(x)$ by $x|_U$. Given a family $r = \{r_i : U_i \to U \}_{i \in I}$ over $U$, an \textbf{$X$-matching family} over $r$ is a collection $\{x_i \in X(U_i) \}_{i \in I}$ such that given any commutative diagram in $\cat{C}$ of the form
\begin{equation*}
\begin{tikzcd}
	U_{ij} & {U_j} \\
	{U_i} & U
	\arrow["{r_i}"', from=2-1, to=2-2]
	\arrow["{r_j}", from=1-2, to=2-2]
	\arrow["u_i"', from=1-1, to=2-1]
	\arrow["u_j", from=1-1, to=1-2]
\end{tikzcd}
\end{equation*}
then $X(u_i)(x_i) = X(u_j)(x_j)$ for all $i,j \in I$. Note here that $U_{ij}$ is an arbitrary object in $\cat{C}$ and $u_i$, $u_j$ are arbitrary maps. If the presheaf is clear from context we may say that $\{ x_i \}$ is a matching family for $r$.

If $X$ is a presheaf on $\cat{C}$, and $r$ is a family of morphisms on $U$, then let $\text{Match}(r,X)$ denote the set of $X$-matching families over $r$. Given a $X$-matching family $\{ x_i \}$ over $r$, an \textbf{amalgamation} for $\{ x_i \}$ is a section $x \in X(U)$ such that $X(r_i)(x) = x_i$ for all $i$.
\end{Def}

\begin{Lemma}[{\cite[Lemma 2.27]{Minichiello2025}}] \label{lem coverage if pullbacks exist}
Suppose that $j$ is a coverage on a small category $\cat{C}$ such that pullbacks along covering maps exist. If $r = \{ r_i: U_i \to U \}_{i \in I}$ is a covering family of $U$, then a collection $\{s_i \in X(U_i) \}$ of sections of a presheaf $X$ on $\cat{C}$ is a matching family for $r$ if and only if for every pullback square of the form
\begin{equation*}
\begin{tikzcd}
	{U_i \times_U U_j} & {U_j} \\
	{U_i} & U
	\arrow["{r_i}"', from=2-1, to=2-2]
	\arrow["{r_j}", from=1-2, to=2-2]
	\arrow["{\pi_i}"', from=1-1, to=2-1]
	\arrow["{\pi_j}", from=1-1, to=1-2]
\end{tikzcd}
\end{equation*}
it follows that $X(\pi_i)(s_i) = X(\pi_j)(s_j)$.
\end{Lemma}

\begin{Lemma}[{\cite[Lemma 2.20]{Minichiello2025}}] \label{lem pullback of matching family by refinement is a matching family}
Given a presheaf $X$ on a category $\cat{C}$, suppose that $r = \{r_i: U_i \to U \}$ and $t = \{t_j: V_j \to U \}$ are families over $U$ and $f : r \to t$ is a refinement. If $\{ x_j \}$ is a matching family for $X$ over $t$, then $\{ X(f_i)(x_{\alpha(i)}) \}$ is a matching family for $X$ over $r$.
\end{Lemma}

\begin{Def} \label{def separated and sheaf}
Given any presheaf $X$ and family of morphisms $r$, there is a canonical map
\begin{equation} \label{eq restriction map of presheaf to matching families}
    \text{res}_{r,X}: X(U) \to \text{Match}(r,X)
\end{equation}
which is defined for an element $x \in X(U)$ to be the matching family $\text{res}_{r,X}(x) = \{ X(r_i)(x) \}$ of $X$ over $r$. We say that $X$ is \textbf{separated} on $r$ if $\text{res}_{r,X}$ is injective. We say that $X$ is a \textbf{sheaf} on $r$ if $\text{res}_{r,X}$ is bijective. Given a collection of families $j$ on $\cat{C}$, we say that $X$ is a (separated presheaf) sheaf on $j$ if $X$ is a (separated presheaf) sheaf on $r$ for every $r \in j$.

Given a site $(\cat{C}, j)$, we call $X$ a $j$-\textbf{sheaf} or just a sheaf if it is a sheaf on every covering family $r$ of $j$. Let $\ncat{Sh}(\cat{C}, j)$ denote the full subcategory of $\ncat{Pre}(\cat{C})$ whose objects are sheaves.
\end{Def}

\begin{Lemma}[{\cite[Lemma 2.30]{Minichiello2025}}] \label{lem coverage comparison results}
Given a small category $\cat{C}$ and coverages $j, j'$ on $\cat{C}$, 
\begin{enumerate}
    \item if $j \leq j'$, then $\Sh(\cat{C}, j) \subseteq \Sh(\cat{C}, j')$.
    \item if $j \subseteq j'$, then $\Sh(\cat{C}, j') \subseteq \Sh(\cat{C}, j)$,
\end{enumerate}
\end{Lemma}

\begin{Def} \label{def equiv coverages}
Given coverages $j$ and $j'$ on a small category $\cat{C}$, we write $j \simeq j'$ if $j \leq j'$ and $j' \leq j$, and we say that $j$ and $j'$ are \textbf{refinement equivalent}. By Lemma \ref{lem coverage comparison results}, if two coverages $j,j'$ on the same category $\cat{C}$ are refinement equivalent, then
\begin{equation*}
    \ncat{Sh}(\cat{C}, j) = \ncat{Sh}(\cat{C}, j').
\end{equation*}
\end{Def}

\begin{Def} \label{def saturated coverage}
We say that a coverage $j$ on a small category $\cat{C}$ is \textbf{refinement closed} if for every $r \in j$ and $f : r \to t$ is a refinement, then $t \in j$. We say that $j$ is \textbf{composition closed} if whenever there is a covering family $r  = \{r_i : U_i \to U \}_{i \in I} \in j(U)$ and for each $i \in I$ there is a covering family $t^i = \{ t^i_j : U^i_j \to U_i \}_{j \in J_i} \in j(U_i)$ their composite
\begin{equation*}
    (r \circ t) \coloneqq \{ r_i t^i_j : U^i_j \to U \}_{j \in J_i}
\end{equation*}
is a covering family. We say that $j$ is \textbf{saturated} if it is both refinement closed and composition closed. 
\end{Def}

\begin{Prop}[{\cite[Proposition 6.11]{Minichiello2025}}]
Given a coverage $j$ on a small category $\cat{C}$, there exists a smallest saturated coverage $\cons{sat}(j)$ (which we call the \textbf{saturated closure} of $j$) containing $j$ such that
\begin{equation*}
    \Sh(\cat{C}, j) = \Sh(\cat{C}, \cons{sat}(j)).
\end{equation*}
\end{Prop}

\begin{Rem}
The saturated closure is obtained by first closing $j$ under composition and then closing it under refinement of covering families, see \cite[Section 6]{Minichiello2025} for more details.
\end{Rem}

\begin{Def}
Given a category $\cat{C}$ with $U \in \cat{C}$, a \textbf{sieve} $R$ over $U$ is a family of morphisms over $U$ that is closed under precomposition, namely if $r_i : U_i \to U \in R$, then for any morphism $g: V \to U_i$, the composite $r_i g : V \to U$ is also in $R$. Let $\cons{Sieve}(U)$ denote the full subcategory of $\cons{Fam}(U)$ on the sieves over $U$. Note that if $R$ and $T$ are sieves, then $R \leq T$ if and only if $R \subseteq T$. Let $\cons{sieve}(U)$ denote the poset with objects sieves on $U$ and partial order $\leq$.
\end{Def}

\begin{Def} \label{def sifted coverage}
We say a collection of families $j$ on a small category $\cat{C}$ is \textbf{sifted} if each family $R \in j$ is a sieve. By a \textbf{sifted coverage} we mean a sifted collection of families $J$ such that each $R \in J(U)$ is a sieve, and $y(U) \in J(U)$ for every $U \in \cat{C}$. 
\end{Def}

\begin{Lemma}[{\cite[Lemma 3.2]{Minichiello2025}}] \label{lem sieves and subfunctors iso}
Given a category $\cat{C}$ with $U \in \cat{C}$, there is an isomorphism of posets
\begin{equation}
    \cons{sieve}(U) \cong \cons{Sub}(y(U)),
\end{equation}
where $\cons{Sub}(y(U))$ denotes the poset of subobjects $R \hookrightarrow y(U)$ in $\Pre(\cat{C})$, and where $y(U) = \left( V \mapsto \cat{C}(V,U) \right)$ denotes the Yoneda embedding on $U$.
\end{Lemma}

\begin{Def} \label{def sieve generated by a family of morphisms}
Given a category $\cat{C}$, with $U \in \cat{C}$, and a family of morphisms $r = \{ r_i : U_i \to U \}$ over $U$, let $\overline{r}$ denote the set of morphisms $f : V \to U$ such that $f$ factors through some $r_i \in r$:
\begin{equation*}
    \begin{tikzcd}
	V && U \\
	& {U_i}
	\arrow["f", from=1-1, to=1-3]
	\arrow["{s_f}"', from=1-1, to=2-2]
	\arrow["{r_i}"', from=2-2, to=1-3]
\end{tikzcd}
\end{equation*}
We say that $R$ is \textbf{generated} by $r$ if $R = \overline{r}$, and that $r$ is a \textbf{generating family} for $R$. 
\end{Def}

\begin{Def} \label{def sifted closure}
If $j$ is a coverage, then let $\overline{j}$ denote the sifted collection of families where $R \in \overline{j}(U)$ if $R = \overline{r}$ for some $r \in j(U)$. We call $\overline{j}$ the \textbf{sifted closure} of $j$.   
\end{Def}

\begin{Prop}[{\cite[Corollary 3.23]{Minichiello2025}}] \label{cor sheaves on sifted closure}
A presheaf $X$ on a category $\cat{C}$ is a sheaf for a coverage $j$ if and only if it is a sheaf for $\overline{j}$. In other words, $\ncat{Sh}(\cat{C}, j) = \ncat{Sh}(\cat{C}, \overline{j})$. 
\end{Prop}

\begin{Def} \label{def Grothendieck coverage}
A \textbf{Grothendieck coverage}\footnote{More commonly known as a \textbf{Grothendieck topology}} on a small category $\cat{C}$ is a sifted precoverage $J$ on $\cat{C}$ such that
\begin{itemize}
	\item[(1)] $y(U) \in J(U)$ for every $U \in \cat{C}$,
	\item[(2)] for any sieve $R \in J(U)$ and any morphism $g: V \to U$, $g^*(R) \in J(V)$, and
	\item[(3)] if $R \in J(U)$, $R'$ is a sieve on $U$, and $g^*(R') \in J(V)$ for every $g \in R(V)$, then $R' \in J(U)$.
\end{itemize}
We refer to a category $\cat{C}$ equipped with a Grothendieck coverage a \textbf{Grothendieck site}.
\end{Def}

\begin{Def} \label{def interior coverage}
Suppose that $J$ is a Grothendieck coverage on $\cat{C}$, let $J^\circ$ denote the collection of families on $\cat{C}$ where a family $r$ on $U$ belongs to $J^\circ(U)$ if $\overline{r} \in J(U)$. We call this the \textbf{interior coverage} of $J$.
\end{Def}

\begin{Def}
Given a category $\cat{C}$, let $\ncat{SatCvg}(\cat{C})$ denote the (large) poset of saturated coverages equipped with the $\subseteq$ relation. Similarly let $\ncat{GroCvg}(\cat{C})$ denote the (large) poset of Grothendieck coverages. The constructions $\overline{(-)}$ and $(-)^\circ$ defined above can easily be seen to define maps of posets.
\end{Def}

\begin{Prop}[{\cite[Proposition 6.35]{Minichiello2025}}]
The maps of (large) posets
\begin{equation}
\begin{tikzcd}
	{\ncat{SatCvg}} && {\ncat{GroCvg}}
	\arrow[""{name=0, anchor=center, inner sep=0}, "{{\overline{(-)}}}", curve={height=-18pt}, from=1-1, to=1-3]
	\arrow[""{name=1, anchor=center, inner sep=0}, "{{(-)^\circ}}", curve={height=-18pt}, from=1-3, to=1-1]
	\arrow["\dashv"{anchor=center, rotate=-90}, draw=none, from=0, to=1]
\end{tikzcd}
\end{equation}
form an adjoint isomorphism.
\end{Prop}

\begin{Def} \label{def Grothendieck closure of a coverage}
Given a coverage $j$ on a category $\cat{C}$, let 
$$\Gro(j) = \overline{\sat{j}}.$$
We call this the \textbf{Grothendieck closure} of $j$. 
\end{Def}

\begin{Lemma}[{\cite[Lemma 6.37]{Minichiello2025}}] \label{lem sifted closure has same Grothendieck closure}
Given a coverage $j$ on a small category $\cat{C}$, we have
\begin{equation*}
    \Gro(j) = \Gro(\overline{j}).
\end{equation*}
Furthermore, $\Gro(j)$ is the smallest Grothendieck coverage containing $\overline{j}$ and 
$$\Sh(\cat{C},j) = \Sh(\cat{C}, \Gro(j)).$$
\end{Lemma}

\begin{Def} \label{def pullback-stable}
We say that a coverage $j$ on a small category $\cat{C}$ is \textbf{pullback stable} if the pullback of any basal map $r_i : U_i \to U$ along an arbitrary map $f : V \to U$ exists and if $r = \{r_i : U_i \to U \}_{i \in I}$ is a covering family on $U$, then $f^*(r) = \{V \times_U U_i \to V \}_{i \in I}$ is a covering family on $V$.
\end{Def}

\begin{Def} \label{def grothendieck pretopology}
Given a small category $\cat{C}$, we say that a coverage $j$ is a \textbf{Grothendieck pretopology} or just pretopology, if it is pullback stable and composition closed.
\end{Def}

\begin{Prop}[{\cite[Section 7.4]{Minichiello2025}}]
Given a site $(\cat{C}, j)$, the inclusion functor of $j$-sheaves into presheaves has a left adjoint called \textbf{sheafification}
\begin{equation*}
   \begin{tikzcd}
	{\ncat{Sh}(\cat{C},j)} && {\ncat{Pre}(\cat{C})}
	\arrow[""{name=0, anchor=center, inner sep=0}, shift right=2, hook, from=1-1, to=1-3]
	\arrow[""{name=1, anchor=center, inner sep=0}, "a"', shift right=2, from=1-3, to=1-1]
	\arrow["\dashv"{anchor=center, rotate=-90}, draw=none, from=1, to=0]
\end{tikzcd} 
\end{equation*}
\end{Prop}

\begin{Def} \label{def local epi}
Given a site $(\cat{C}, j)$ and a map $f : X \to Y$ of presheaves, we say that $f$ is a \textbf{$j$-local epimorphism} (or just local epimorphism) if for every map $s : y(U) \to Y$ from a representable, there exists a $j$-covering family $r = \{r_i : U_i \to U \}$ and maps $s_i : y(U_i) \to X$ making the following diagram commute
\begin{equation*}
    \begin{tikzcd}
	{y(U_i)} & X \\
	{y(U)} & Y
	\arrow["{s_i}", from=1-1, to=1-2]
	\arrow["{r_i}"', from=1-1, to=2-1]
	\arrow["f", from=1-2, to=2-2]
	\arrow["s"', from=2-1, to=2-2]
\end{tikzcd}
\end{equation*}
\end{Def}

\begin{Rem} \label{rem local epis are problematic}
The notion of local epimorphism that we use above is not the same as in \cite[Definition 5.1]{Minichiello2025}, instead it is what we call a strong local epimorphism in \cite[Remark 5.3]{Minichiello2025}. In order for this confusion to not permeate further, let us call a morphism $f : X \to Y$ of presheaves satisfying \cite[Definition 5.1]{Minichiello2025} a \textbf{$j$-tree epimorphism}. When the site $(\cat{C}, j)$ is composition closed, a map of presheaves is a local epimorphism if and only if it is a $j$-tree epimorphism.

Let us explain the difference of these notions and why they it is an important distinction. If $f : X \to Y$ is a map of presheaves on a site $(\cat{C}, j)$ (with no further assumptions) and we let $a : \ncat{Pre}(\cat{C}) \to \ncat{Sh}(\cat{C}, j)$ denote the sheafification functor \cite[Definition 7.19]{Minichiello2025}, then $af$ is an epimorphism if and only if $f$ is a $j$-tree epimorphism. We show in \cite[Remark 7.39]{Minichiello2025} that there exist maps $f : X \to Y$ of presheaves such that $af$ is an epimorphism but $f$ is not a local epimorphism.

However, for the purposes of this paper, requiring the more general notion of $j$-tree epimorphism is unnecessary. We will just consider this stronger notion of local epimorphism. This is especially important for the site $(\ncat{Cart}, j_{\text{good}})$ from Example \ref{ex sites} below, as it is not composition closed, and it is the main site one wishes to work with in higher differential geometry. Hence, while this may be a confusing point, it is worth distinguishing it so that we may obtain strict hypercompleteness of $(\ncat{Cart}, j_{\text{good}})$ in Theorem \ref{th many sites are strictly hypercomplete}.
\end{Rem}

\subsection{Examples of Sites} \label{section examples of sites}

\begin{Rem}
In what follows, we always assume (topological, smooth, complex) manifolds are Hausdorff, second countable and have the same dimension in each connected component. In particular, a manifold can only have countably many connected components. Furthermore we do not consider the empty set a manifold for technical reasons\footnote{In the case that we do include the empty set as a manifold, then not all of the coverages on $\ncat{Man}$ given in Example \ref{ex coverages on Man} will be refinement equivalent. For example, we cannot refine the empty covering family on the empty set by a surjective submersion, so by allowing empty manifolds it follows that $j_{\text{open}} \simeq j_{\text{sub}}$.
Waldorf has also encountered this issue and dealt with it a different way, mainly by changing the notion of refinement and the definition of sheaves, see \cite[Warning 5.1.4]{waldorf2024internalgeometry}.}.

A priori, the categories of (topological, smooth, complex) manifolds are large categories, but using the Whitney embedding theorem \cite[Example A.27]{Minichiello2025} for smooth manifolds or a direct set-theoretic argument \cite[Example A.28]{Minichiello2025} for each of these cases, one can show that each of these categories is essentially small. Hence we will silently replace each of these categories by one of their small equivalent categories.
\end{Rem}

\begin{Rem}
For many of the following examples we provide a table summarizing the properties of the corresponding coverages, here is a guide for the shorthand we use for the headings of these tables:
\begin{itemize}
    \item cvg - coverage (Definition \ref{def coverage}),
    \item ref. cl - refinement closed (Definition \ref{def saturated coverage}),
    \item comp. cl - composition closed (Definition \ref{def saturated coverage}),
    \item sat - saturated (Definition \ref{def saturated coverage}),
    \item cofib - cofibrant (Definition \ref{def cofibrant site}),
    \item plbk - pullback-stable (Definition \ref{def pullback-stable}),
    \item Verdier - Verdier pretopology (Definition \ref{def verdier site})
\end{itemize}
\end{Rem}

\begin{Ex} \label{ex site of top space}
Given a topological space $X$, let $\mathcal{O}(X)$ denote the poset of open subsets of $X$, ordered by inclusion $U \subseteq V$. Let $j_X$ denote the coverage given by open covers.
\begin{center}    
\begin{tabular}{|c|c|c|c|c|c|c|} 
 \hline
 cvg & ref. cl & comp. cl & sat & cofib & plbk & Verdier \\ [0.5ex] 
 \hline\hline
 $j_X$ & $\checkmark$ & $\checkmark$ & $\checkmark$ & $\checkmark$ & $\checkmark$ & $\checkmark$ \\ 
 \hline
\end{tabular}
\end{center}
\end{Ex}

\begin{Ex} \label{ex coverages on Man}
Let $\ncat{Man}$ denote the category of nonempty finite dimensional smooth manifolds and smooth maps between them. Given a finite dimensional smooth manifold $M$, we call a family of morphisms $r = \{r_i : U_i \to M \}$ an
\begin{enumerate}
    \item \textbf{open covering family} if each map $r_i : U_i \to M$ is an open embedding\footnote{A map $f : M \to N$ of manifolds is an open embedding if it is an immersion and the underlying map $\tau(f) : \tau(M) \to \tau(N)$ of topological spaces is an open embedding, which means that $\tau(f)$ is a homeomorphism onto its image with the subspace topology, which implies that $f$ has underlying injective set function.} and $\bigcup_i r_i(U_i) = M$. Let $j_{\text{open}}$ denote the collection of open embedding families.
    \item \textbf{countable open covering family} if it is an open covering family with countable cardinality. Let $j_{\omega}$ denote the collection of countable open covering families. 
    \item \textbf{surjective submersion} if $r$ is a singleton set $r = \{f : N \to M\}$ such that $f$ is a surjective submersion of smooth manifolds. Let $j_{\text{sub}}$ denote the collection of surjective submersions.
    \item \textbf{surjective local diffeomorphism} if $r$ is a singleton set $r = \{f : N \to M\}$ such that $f$ is a surjective local diffeomorphism. Let $j_{\text{sld}}$ denote the collection of surjective local diffeomorphisms.
\end{enumerate}
It is easy to verify that each of these collections of families is a coverage, and each is furthermore composition closed. 

\begin{Rem}
Let us note that sometimes in the literature $j_{\text{open}}$ is taken to consist of covering families $\{r_i : U_i \to M\}$ where each $r_i$ is actually \textit{equal} to the inclusion of an open subset and $\cup U_i = M$. This collection of families forms a precoverage (Definition \ref{def coverage}) but not a coverage as we have defined it. This was a conscious choice we have made, as we wish for all the coverages we consider to be \textbf{isomorphism closed} in the sense that if $f : V \to U \in j(U)$ and there is an isomorphism $g : V' \to V$, then $fg \in j(U)$. This is an important property to have for example in the proof of Proposition \ref{prop can refine hypercovers by Verdier hypercovers on matching Verdier site} where we start to manipulate objects that are only defined up to isomorphism. However, it is not totally necessary. By making careful choices throughout all of the proofs in this paper, one could allow $j_{\text{open}}$ to be open covers.

Nothing is lost in working with open embeddings rather than open covers. Indeed, if $r_i : U_i \hookrightarrow M$ is an open embedding, then the manifold given by taking the image set $r_i(U_i)$ equipped with the induced smooth structure from $M$ is diffeomorphic to $U_i$, since the map $U \to r_i(U_i)$ is a bijective immersion. Hence using the inverse map, we have a commutative diagram
\begin{equation*}
\begin{tikzcd}
	{r_i(U_i)} && {U_i} \\
	& M
	\arrow["\cong", from=1-1, to=1-3]
	\arrow[hook, from=1-1, to=2-2]
	\arrow["{r_i}", hook', from=1-3, to=2-2]
\end{tikzcd}
\end{equation*} 
where the lefthand map is an inclusion of an open subset of $M$.

In practice this means that we can always identify an open covering family with an open cover and vice versa. Hence going forward, we will not distinguish between an open cover $\mathcal{U} = \{U_i \subseteq M \}$ of a topological space $M$ and a family $\{r_i : U_i \hookrightarrow M \}$ of open embeddings such that $\cup_i r_i(U_i) = M$.
\end{Rem}

Now the category $\ncat{Man}$ does not have all pullbacks, but it does have pullbacks of transverse maps, and hence the pullback of any smooth map along a submersion exists. Indeed, given two smooth maps $f : N \to M$ and $g : P \to M$ that are transverse, we can consider the diagonal $\Delta : M \to M \times M$ as an embedded submanifold. Since $f$ and $g$ are transverse to each other, the map $f \times g$ is transverse to $\Delta$ and hence by \cite[Theorem 6.30]{lee2003smooth}, the preimage $(f \times g)^{-1}(\Delta(M))$ is a smooth manifold, and this is isomorphic to the pullback $N \times_M P$. All maps are transverse to submersions, so the pullback of a smooth map along a submersion always exists in $\ncat{Man}$.

Given an open embedding $r_i : U_i \hookrightarrow M$ and a smooth map $g : N \to M$, then $g^{-1}(r_i(U_i))$ is an open subset of $N$, and hence inherits a smooth manifold structure such that it is an embedded submanifold \cite[Section 5.1]{lee2003smooth}. But $g^{-1}(r_i(U_i))$ is isomorphic to the pullback $N \times_M U_i$. Hence pullbacks of open embeddings by arbitrary smooth maps also exist in $\ncat{Man}$. It follows that $j_{\text{open}}, j_\omega, j_{\text{sub}}$ and $j_{\text{sld}}$ are all pretopologies. Now let us show that all of these coverages are refinement equivalent.

Clearly $j_{\text{sub}} \leq j_{\text{sld}}$. Let $f : N \to M$ be a surjective submersion. Then there exists an open cover $\mathcal{U} = \{U_i \subseteq M \}$ of $M$ and sections $s_i : U_i \to N$ of $f$. Since $M$ is second countable, we can find a countable refinement $\mathcal{V} = \{V_j \subseteq M \}$ of $\mathcal{U}$ by \cite[Theorem 2.50]{lee2000introduction}. Then $\sum_j V_j$ is a finite dimensional smooth manifold -- $\ncat{Man}$ has countable coproducts of manifolds of the same dimension, but not all small coproducts -- and we obtain the following commutative diagram
\begin{equation*}
   \begin{tikzcd}
	{\sum_j V_j} && N \\
	& M
	\arrow["{s}", from=1-1, to=1-3]
	\arrow["{f'}"', from=1-1, to=2-2]
	\arrow["f", from=1-3, to=2-2]
\end{tikzcd} 
\end{equation*}
where $s$ is defined componentwise by the local sections, and the map $f'$ is a surjective local diffeomorphism. Hence $j_{\text{sld}} \leq j_{\text{sub}}$. This argument also shows that $j_{\omega} \leq j_{\text{open}}$, and clearly $j_{\text{open}} \leq j_{\omega}$. Each component $f'_j : V_j \to M$ is an open embedding, and hence we also see that $j_{\text{open}} \leq j_{\text{sub}}$. Conversely given an open embedding family, we can find a countable refinement of the family of the image inclusions, and hence $j_{\text{sld}} \leq j_{\text{open}}$. Hence all of the coverages considered above are refinement equivalent
\begin{equation*}
    j_{\text{open}} \simeq j_{\omega} \simeq j_{\text{sub}} \simeq j_{\text{sld}}.
\end{equation*}
Furthermore, the coverage $j_{\text{sub}}$ is refinement closed. Indeed, given a commutative diagram
\begin{equation*}
    \begin{tikzcd}
	P && N \\
	& M
	\arrow["h", from=1-1, to=1-3]
	\arrow["f"', from=1-1, to=2-2]
	\arrow["g", from=1-3, to=2-2]
\end{tikzcd}
\end{equation*}
where $f$ is a surjective submersion, then $g$ is surjective. By the local sections theorem \cite[Theorem 4.26]{lee2003smooth}, $g$ is a submersion if and only if admits local sections. Since $f$ is a submersion, we can compose its local sections with $h$ and obtain local sections of $g$. Thus $j_{\text{sub}}$ is a saturated pretopology.

We summarize the properties of the various coverages on $\ncat{Man}$ in the following table.
\begin{center}    
\begin{tabular}{|c|c|c|c|c|c|c|} 
 \hline
 cvg & ref. cl & comp. cl & sat & cofib & plbk & Verdier\\ [0.5ex] 
 \hline\hline
$j_{\text{open}}$, $j_{\omega}$ & $\times$ & $\checkmark$ & $\times$ & $\checkmark$ & $\checkmark$ & $\checkmark$\\ 
 \hline
 $j_{\text{sub}}$ & $\checkmark$ & $\checkmark$ & $\checkmark$ & $\checkmark$ & $\checkmark$ & $\times$  \\ 
 \hline
 $j_{\text{sld}}$ & $\times$ & $\checkmark$ & $\times$ & $\checkmark$ & $\checkmark$ & $?$ \\ 
 \hline
\end{tabular}
\end{center}
\end{Ex}

\begin{Ex} \label{ex coverages on Cart}
Let $\ncat{Cart}$ denote the full subcategory on those manifolds that are diffeomorphic to $\R^n$ for some $n \geq 0$. We call the objects of this category \textbf{cartesian spaces}.
We can similarly define coverages $j_{\text{open}}$ and $j_{\omega}$ for $\ncat{Cart}$ as we did above for $\ncat{Man}$, though the covers must now be cartesian covers, i.e. they must be open covers whose open subsets are cartesian spaces. Furthermore, the preimage of a cartesian space by a smooth map is no longer a cartesian space in general, and hence they are not Grothendieck pretopologies on $\ncat{Cart}$, though they are composition-closed. Let us also mention the coverage $j_{\text{good}}$ of good open covers, described in detail in \cite[Introduction and Example 8.29]{Minichiello2025}. This coverage, defined on $\ncat{Cart}$ is neither refinement or composition closed, but it is a cofibrant site (Definition \ref{def cofibrant site}), which is a property the other coverages on $\ncat{Cart}$ lack. We discuss $(\ncat{Cart}, j_{\text{good}})$ and its relation to higher differential geometry in more detail in Section \ref{section higher differential geometry}.

We summarize the properties of the various coverages on $\ncat{Cart}$ in the following table.
\begin{center}    
\begin{tabular}{|c|c|c|c|c|c|c|c|} 
 \hline
 cvg & ref. cl & comp. cl & sat & cofib & plbk & Verdier \\ [0.5ex] 
 \hline\hline
$j_{\text{open}}$, $j_{\omega}$ & $\times$ & $\checkmark$ & $\times$ & $\times$ & $\times$ & $\times$ \\ 
 \hline
 $j_{\text{good}}$ & $\times$ & $\times$ & $\times$ & $\checkmark$ & $\times$ & $\times$ \\ 
 \hline
\end{tabular}
\end{center}
\end{Ex}

\begin{Ex} \label{ex coverages on top manifolds}
Let $\ncat{TopMan}$ denote the category whose objects are finite-dimensional topological manifolds and whose morphisms are continuous maps. We define $j_{\text{open}}$ and $j_{\omega}$ as we did for $\ncat{Man}$ in Example \ref{ex coverages on Man}. Now we define $j_{\text{sub}}$ to consist of singleton families of maps with local sections as in Example \ref{ex coverages on Top}. 
\begin{center}    
\begin{tabular}{|c|c|c|c|c|c|c|} 
 \hline
 cvg & ref. cl & comp. cl & sat & cofib & plbk & Verdier\\ [0.5ex] 
 \hline\hline
$j_{\text{open}}$, $j_{\omega}$ & $\times$ & $\checkmark$ & $\times$ & $\checkmark$ & $\checkmark$ & $\checkmark$\\ 
 \hline
 $j_{\text{sub}}$ & $\checkmark$ & $\checkmark$ & $\checkmark$ & $\checkmark$ & $\checkmark$ & $\times$  \\ 
 \hline
\end{tabular}
\end{center}
\end{Ex}

\begin{Ex} \label{ex coverages on generalized manifolds}
We can also augment the surjective submersion coverage on $\ncat{Man}$ to other more generalized kinds of manifolds. The following sites are Grothendieck pretopologies:
\begin{enumerate}
    \item $(\ncat{HilMan}, j_{\text{sub}})$ the site of Hilbert manifolds,
    \item $(\ncat{BanMan}, j_{\text{sub}})$ the site of Banach manifolds,
    \item $(\ncat{FreMan}, j_{\text{sub}})$ the site of Fr\'echet manifolds, and
    \item $(\ncat{ConMan}, j_{\text{sub}})$ the site of locally convex manifolds.
\end{enumerate}
See \cite[Section 9.3]{meyer2015groupoids} for more details.
\end{Ex}

\begin{Ex} \label{ex coverages on complex sites}
Let $\C\ncat{Man}$ denote the category of finite dimensional complex manifolds and holomorphic maps. Let $\ncat{Stein}$ and $\ncat{\C Disk}$ denote the full subcategories of Stein manifolds and polydisks with their corresponding coverages $j_{\text{open}}, j_\omega$ in \cite[Example 8.32]{Minichiello2025}.
\begin{center}    
\begin{tabular}{|c|c|c|c|c|c|c|} 
 \hline
 cvg & ref. cl & comp. cl & sat & cofib & plbk & Verdier \\ [0.5ex] 
 \hline\hline
 $j^{\ncat{\C Man}}_{\text{open}}, j^{\ncat{\C Man}}_{\omega}$ & $\times$ & $\checkmark$ & $\times$ & $\checkmark$ & $\checkmark$ & $\checkmark$ \\ 
 \hline
 $ j^{\ncat{Stein}}_{\text{open}}, j^{\ncat{Stein}}_{\omega}$ & $\times$ & $\checkmark$ & $\times$ & ? & $\checkmark$ & $\checkmark$ \\ 
 \hline
 $j^{\ncat{\C Disk}}_{\text{open}}, j^{\ncat{\C Disk}}_{\omega}$ & $\times$ & $\checkmark$ & $\times$ & ? & $\times$ & ?\\ 
 \hline
\end{tabular}
\end{center}
\end{Ex}

Now we wish to consider coverages on the categories $\ncat{Top}$ and $\ncat{Diff}$, the category of topological spaces and diffeological spaces respectively. There is a set-theoretical issue however. This is explained in detail in \cite[Section 8.4]{Minichiello2025}. The categories $\ncat{Top}$ and $\ncat{Diff}$ are essentially large categories, and hence the definition of coverages (Definition \ref{def coverage}) does not apply. We deal with this using Grothendieck universes \cite[Definition A.15]{Minichiello2025}. Suppose we have three Grothendieck universes $\mathbbm{u} \in \mathbb{U} \in \mathbb{V}$. We call the objects in $\mathbbm{u}$ \textbf{extra small}, the objects in $\mathbbm{U}$ \textbf{small} and the objects in $\mathbbm{V}$ \textbf{large}. We say a category $\cat{C}$ is $\mathbbm{W}$-small for $\mathbbm{W} \in \{ \mathbbm{u}, \mathbb{U}, \mathbb{V} \}$ if the sets $\text{Obj}(\cat{C})$, $\text{Mor}(\cat{C})$ of objects and morphisms respectively belong to $\mathbb{W}$. If $\cat{C}$ is a concrete category, i.e. has a faithful functor $\pi : \cat{C} \to \ncat{Set}$, then we let $\cat{C}_{\mathbbm{W}}$ denote the full subcategory on those objects $U \in \cat{C}$ such that $\pi(U) \in \mathbbm{W}$.

\begin{Ex} \label{ex coverages on Top}
Let $\ncat{Top}$ denote the category of nonempty $\mathbbm{u}$-small topological spaces. This is a small category, and so we can equip it with a coverage. We list a few possible coverages on topological spaces
\begin{enumerate}
    \item $j_{\text{open}}$ consisting of families of open embeddings $r = \{r_i : U_i \to U \}$, i.e. those maps that are homeomorphisms onto their image and $\bigcup_i r_i(U_i) = U$,
    \item $j_{\text{sub}}$ consisting of singleton families $\{f : V \to U \}$ that have local sections, i.e. for every point $x \in U$ there is an open set $W \subseteq U$ containing $x$ and a map $s_x : W \to V$ such that $f s_x = 1_W$.
    \item $j_{\text{slh}}$ consisting of surjective local homeomorphisms, also called surjective \'etale maps,
    \item $j_{\omega}$ consisting of countable open covers,
    \item $j_{\text{surj}}$ consisting of continuous surjections, and
    \item $j_{\text{proj}}$ consisting of proper surjections.
\end{enumerate}
Now $\ncat{Top}$ has all $\mathbbm{u}$-small coproducts.  
The same arguments as in Example \ref{ex coverages on Man} show that (1) - (4) are all refinement equivalent. Since $\ncat{Top}$ has pullbacks, each of the above coverages are Grothendieck pretopologies and $j_{\text{sub}}$ is also saturated. We note that not all open covers of an arbitrary topological space have a countable refinement, and hence while  $j_{\text{open}} \leq j_{\omega}$, it is not the case that $j_{\omega}$ is refinement equivalent to any of the above coverages. For more possible coverages on $\ncat{Top}$ see \cite[Section 9.2]{meyer2015groupoids} and \cite[Section 5.5]{waldorf2024internalgeometry}. We note that $j_{\text{surj}} \leq j$ for all coverages $j$ mentioned above.
\end{Ex}

\begin{Ex} \label{ex coverages on diffeological spaces}
Let $\ncat{Diff}$ denote the category of $\mathbbm{u}$-small diffeological spaces. There are several options for coverages on diffeological spaces:
\begin{enumerate}
    \item let $j_{\text{open}}$ denote the collection of families consisting of $D$-open covers,
    \item let $j_{\text{subd}}$ denote the collection of singleton families consisting of subductions, and 
    \item let $j_{\text{sub}}$ denote the collection of singleton families which admits $D$-local sections.
\end{enumerate}
All of the above coverages are Grothendieck pretopologies. Furthermore $j_{\text{subd}}$ and $j_{\text{sub}}$ are saturated. See \cite[Section 5.4]{waldorf2024internalgeometry} for more details.
\end{Ex}

\begin{Ex} \label{ex sites}
Here we collect several sites that will be our main interest in this paper: Let $X$ be a topological space
\begin{equation*}
  (\ncat{TopMan}, j_{\text{open}}), \quad (\mathcal{O}(X), j_X), \quad (\ncat{Man}, j_{\text{open}}), \quad (\ncat{Cart}, j_{\text{open}/\text{good}}), \quad (\C\ncat{Man}, j_{\text{open}}), \quad (\ncat{Stein}, j_{\text{open}}). 
\end{equation*}
These sites will be the focus of our main result, Theorem \ref{th many sites are strictly hypercomplete}.
\end{Ex}

\section{Simplicial Presheaves} \label{section simplicial presheaves}

We quickly overview some notation for simplicial presheaves. We have the following adjoint triples.
\begin{equation} \label{eq adjunctions for simplicial presheaves}
\begin{tikzcd}
	{\sPre(\cat{C})} && {\ncat{sSet}}
	\arrow[""{name=0, anchor=center, inner sep=0}, "{(-)_c}"{description}, from=1-3, to=1-1]
	\arrow[""{name=1, anchor=center, inner sep=0}, "{\text{colim}_{\cat{C}^\op}}", curve={height=-30pt}, from=1-1, to=1-3]
	\arrow[""{name=2, anchor=center, inner sep=0}, "{\text{lim}_{\cat{C}^{\op}}}"', curve={height=30pt}, from=1-1, to=1-3]
	\arrow["\dashv"{anchor=center, rotate=-93}, draw=none, from=0, to=2]
	\arrow["\dashv"{anchor=center, rotate=-87}, draw=none, from=1, to=0]
\end{tikzcd},
\qquad
\begin{tikzcd}
	{\sPre(\cat{C})} && {\Pre(\cat{C})}
	\arrow[""{name=0, anchor=center, inner sep=0}, "{{}^c(-)}"{description}, from=1-3, to=1-1]
	\arrow[""{name=1, anchor=center, inner sep=0}, "{\pi_0}", curve={height=-30pt}, from=1-1, to=1-3]
	\arrow[""{name=2, anchor=center, inner sep=0}, "{(-)_0}"', curve={height=30pt}, from=1-1, to=1-3]
	\arrow["\dashv"{anchor=center, rotate=-89}, draw=none, from=1, to=0]
	\arrow["\dashv"{anchor=center, rotate=-91}, draw=none, from=0, to=2]
\end{tikzcd}
\end{equation}
The fully faithful embedding ${}^c(-) : \Pre(\cat{C}) \hookrightarrow \sPre(\cat{C})$ is defined objectwise as $\left( {}^cF \right)(U) = {}^c F(U)$ for all $U \in \cat{C}$, where ${}^c F(U)$ denotes the discrete simplicial set on the set $F(U)$, with all face and degeneracy maps given by the identity. We call a simplicial presheaf of this form a \textbf{discrete simplicial presheaf}. This functor has a left adjoint $\pi_0: \sPre(\cat{C}) \to \Pre(\cat{C})$, defined objectwise by
$$ (\pi_0 X)(U) = \text{coeq} \left( \begin{tikzcd}
	{X(U)_0} & {X(U)_1}
	\arrow["d_0"', shift right=2, from=1-2, to=1-1]
	\arrow["d_1", shift left=2, from=1-2, to=1-1]
\end{tikzcd} \right),$$
and a right adjoint $(-)_0: \sPre(\cat{C}) \to \Pre(\cat{C})$ defined objectwise by $X_0(U) = X(U)_0$. Note that limits and colimits in $\sPre(\cat{C})$ are computed objectwise. If $U \in \cat{C}$, then we let $y(U) \in \Pre(\cat{C})$ denote its corresponding representable presheaf, with $y(U)(V) = \cat{C}(V, U)$. There is also a functor $(-)_c : \ncat{sSet} \to \sPre(\cat{C})$ defined objectwise by $X_c(U) = X$ for every $U \in \cat{C}$, and where for every $f : U \to V$ in $\cat{C}$, $X_c(f)$ is the identity. We call a simplicial presheaf of this form a \textbf{constant simplicial presheaf}. We will often hide the notation ${}^c(-)$ and $(-)_c$ for discrete and constant simplicial presheaves when it is clear from context. The functor $(-)_c$ has a left and right adjoint given by taking the (co)limit over $\cat{C}^\op$. We call the functor $ \Gamma = \lim_{\cat{C}^\op}$ the \textbf{global sections functor}. Since $\Gamma = \colim_{\cat{C}}$, if $\cat{C}$ has a terminal object, then $\Gamma(X) = X(*)$.

The category of simplicial presheaves on $\cat{C}$ is canonically $\ncat{sSet}$-enriched. Let $X$ and $Y$ be simplicial presheaves, then let $\u{\sPre}(\cat{C})(X,Y)$ denote the simplicially-enriched hom, defined degreewise by
$$\u{\sPre}(\cat{C})(X,Y)_n = \sPre(\cat{C})(\Delta^n_c \times X, Y).$$
It is furthermore tensored and powered over $\ncat{sSet}$. Indeed for every $K \in \ncat{sSet}$ and $X \in \sPre(\cat{C})$, let $K \otimes X \coloneqq K_c \times X$ and define $X^K$ to be the simplicial presheaf defined objectwise by $(X^K)(U) = \u{\ncat{sSet}}(K, X(U))$. There are isomorphisms
\begin{equation} \label{eq (co)tensoring isos}
    \u{\sPre}(\cat{C})(K \otimes X, Y) \cong \u{\ncat{sSet}}(K, \u{\sPre}(\cat{C})(X,Y)) \cong \u{\sPre}(\cat{C})(X, Y^K),
\end{equation}
and we call $K \otimes X$ the \textbf{tensoring} of $X$ by $K$, and $X^K$ the \textbf{powering} of $X$ by $K$. With the simplicially enriched hom, there is an enriched version of the Yoneda lemma.

\begin{Lemma}
Given $U \in \cat{C}$ and a simplicial presheaf $X$ on $\cat{C}$,
\begin{equation*}
    \u{\sPre}(\cat{C})(y(U), X) \cong X(U).
\end{equation*}
\end{Lemma}

\begin{proof}
First note that $\sPre(\cat{C})$ is itself a presheaf category $\sPre(\cat{C}) = \ncat{Fun}((\cat{C} \times \mathsf{\Delta})^\op, \ncat{Set})$, and hence for every $n \geq 0$,
\begin{equation*}
    \u{\sPre}(\cat{C})(y(U), X)_n \cong \sPre(\cat{C})(y(U) \times \Delta^n, X) \cong X(U)_n.
\end{equation*}
\end{proof}

Note that $\ncat{sPre}(\cat{C})$ is also internally enriched. Indeed, let $\Map_{\ncat{sPre}(\cat{C})}(X,Y)$ denote the simplicial presheaf defined objectwise by
\begin{equation*}
    \Map_{\ncat{sPre}(\cat{C})}(X, Y)(U) = \u{\sPre}(\cat{C})(X \times y(U), Y).
\end{equation*}
In particular if $\cat{C}$ has a terminal object $*$ then
\begin{equation*}
\Gamma(\Map_{\ncat{sPre}(\cat{C})}(X, Y)) \cong \u{\sPre}(\cat{C})(X,Y).
\end{equation*}
This internal hom also reduces to the cotensoring by simplicial sets
\begin{equation*}
\Map_{\sPre(\cat{C})}(K_c, Y) \cong Y^K.
\end{equation*}

We say a map $f: X \to Y$ of simplicial presheaves is an \textbf{objectwise weak equivalence} if $f_U: X(U) \to Y(U)$ is a weak equivalence of simplicial sets for every $U \in \cat{C}$. Similarly an \textbf{objectwise fibration} is a map $f: X \to Y$ of simplicial presheaves such that $f_U: X(U) \to Y(U)$ is a Kan fibration of simplicial sets for every $U \in \cat{C}$. 

\begin{Prop}[{\cite[Page 314]{Bousfield1972}, \cite[Section A.2.6]{Lurie2009}}] \label{prop projective model structure}
Given a small category $\cat{C}$, there is a combinatorial, proper, simplicial model structure on $\sPre(\cat{C})$, which we call the \textbf{projective model structure} or \textbf{Bousfield-Kan model structure}, whose weak equivalences are the objectwise weak equivalences, and whose fibrations are the objectwise fibrations. Furthermore
\begin{equation*}
\begin{aligned}
    I_{\text{proj}} & = \{ \partial \Delta^n \otimes y(U) \to \Delta^n \otimes y(U) \, | \, U \in \cat{C}, n \geq 1 \} \cup \{\varnothing \to \Delta^0 \otimes y(U) \} \\
       J_{\text{proj}} & = \{ \Lambda^n_k \otimes y(U) \to \Delta^n \otimes y(U) \, | \, U \in \cat{C}, 0 \leq k \leq n, n \geq 1 \},
\end{aligned}
\end{equation*}
are sets of generating projective cofibrations and projective trivial cofibrations, respectively. The fibrant objects in this model structure are precisely those simplicial presheaves that are objectwise Kan complexes, which we call \textbf{projective fibrant} simplicial presheaves. Similarly we refer to the cofibrant objects as \textbf{projective cofibrant} simplicial presheaves.
\end{Prop}

\begin{Rem}
There is a Quillen equivalent model structure on simplicial presheaves where the cofibrations and weak equivalences are objectwise, which is called the injective or Heller model structure. The fibrations for this model structure have historically been hard to characterize, but recently there has been a characterization given in \cite[Theorem 8.22]{Shulman2009}. See \cite{Blander2001} for an overview of the different model structures on simplicial presheaves.
\end{Rem}

In what follows, if $\cat{C}$ is a small category, then we will let $\H(\cat{C})$ denote\footnote{This particular notation is inspired by the corresponding notation for $\infty$-toposes in \cite{Schreiber2013}. Given two objects $X, A$ in an $\infty$-topos $\H$, the corresponding derived mapping space is denoted $\H(X,A)$, and one of the driving philosophies of \cite{Schreiber2013} is that $H^0(X,A) = \pi_0\H(X,A)$ is the $0$th cohomology of $X$ with coefficients in $A$.} the category $\ncat{sPre}(\cat{C})$ of simplicial presheaves equipped with the projective model structure.

When $\cat{C}$ has a terminal object $*$, it is easy to see that the adjunction
\begin{equation*}
\begin{tikzcd}
	{\ncat{sSet}} && {\H(\cat{C})}
	\arrow[""{name=0, anchor=center, inner sep=0}, "{(-)_c}", shift left=2, from=1-1, to=1-3]
	\arrow[""{name=1, anchor=center, inner sep=0}, "\Gamma", shift left=2, from=1-3, to=1-1]
	\arrow["\dashv"{anchor=center, rotate=-90}, draw=none, from=0, to=1]
\end{tikzcd} 
\end{equation*}
is Quillen, since $\Gamma$ preserves fibrations and trivial fibrations.

\subsection{Split Simplicial Presheaves}

We now wish to characterize the conditions that make a simplicial presheaf projective cofibrant.

\begin{Def}[{\cite[Definition 4.13]{Dugger2004}}] \label{def split}
Given a category $\cat{C}$ with finite coproducts, a simplicial object $X : \mathsf{\Delta}^\op \to \cat{C}$ is \textbf{split} if there exist subobjects $N_k \hookrightarrow X_k$ such that the canonical map
\begin{equation*}
   \sum_{\sigma : [n] \twoheadrightarrow [m]} N_m \to X_n  
\end{equation*}
where the coproduct is indexed by all surjections in $\mathsf{\Delta}$ with domain $[n]$, is an isomorphism for every $n \geq 0$. We call such a collection of isomorphisms a \textbf{splitting}.
\end{Def}

\begin{Lemma} \label{lem simplicial sets are split}
If we consider simplicial sets as simplicial objects in $\ncat{Set}$, then every simplicial set $X$ is split.
\end{Lemma}

\begin{proof}
Given a simplicial set $X$, let $N_k$ denote the set of non-degenerate $k$-simplices of $X$. Then by the Eilenberg-Zilber lemma \cite[Page 26]{Gabriel1967}, for every surjection $\sigma: [n] \to [m]$, the corresponding map $X_m \to X_n$ is injective when restricted to $N_m$. In other words, $X_n$ can be written as a coproduct $X_n \cong N_n + \sum_{\sigma: [n] \twoheadrightarrow [m]} X(\sigma)(N_m)$.
\end{proof}

Let $\ncat{sSet}_{\nd}$ denote the category whose objects are simplicial sets and whose morphisms are those maps of simplicial sets that send non-degenerate simplices to non-degenerate simplices.

\begin{Lemma} \label{lem split simplicial presheaf iff factors through nondegen}
Given a category $\cat{C}$, a simplicial presheaf $X$ over $\cat{C}$ is split if and only if as a functor $X : \cat{C}^\op \to \ncat{sSet}$ it factors through $\ncat{sSet}_{\nd} \hookrightarrow \ncat{sSet}$.
\end{Lemma}

\begin{proof}
By Lemma \ref{lem simplicial sets are split}, we know that for every object $U \in \cat{C}$, $X(U)$ splits appropriately. We need only know then that for a map $f: V \to U$ in $\cat{C}$, the corresponding morphism $X(U) \to X(V)$ preserves the splitting. This is precisely what $X$ factoring through $\ncat{sSet}_{\nd}$ guarantees.
\end{proof}

\begin{Rem}
Note that this means that if $X$ is a split simplicial presheaf and $x \in X_n(U)$ is a nondegenerate $n$-simplex, then it may still be the case that some faces $d_i x$ are degenerate. In other words, the face maps may not respect the splitting, but the induced maps $X(f) : X(U) \to X(V)$ coming from the category $\cat{C}$ must respect the splitting.
\end{Rem}

\begin{Ex}[{\cite[Lemma 14.18.6]{Stacksprojectauthors2025}}]
In an abelian category $\cat{A}$, every simplicial object $X$ has a canonical splitting with $N_0 = X_0$ and
\begin{equation*}
   N_m = \cap_{i = 0}^{m-1}  \; \text{ker}(d_i)
\end{equation*}
where each $d_i : X_m \to X_{m-1}$ is a face map.
\end{Ex}

\begin{Def}
We say an object $P$ in a category $\cat{C}$ is \textbf{projective} if there is a lift in every diagram of the form
\begin{equation*}
    \begin{tikzcd}
	& X \\
	P & Y
	\arrow["f", from=1-2, to=2-2]
	\arrow["h", dashed, from=2-1, to=1-2]
	\arrow["g"', from=2-1, to=2-2]
\end{tikzcd}
\end{equation*}
where $f$ is an epimorphism. Equivalently $P$ is projective if $\cat{C}(P, -) : \cat{C} \to \ncat{Set}$ preserves epimorphisms.
\end{Def}

\begin{Def} \label{def semi-representable presheaf}
Given a small category $\cat{C}$, we say that a presheaf $X$ on $\cat{C}$ is \textbf{semi-representable} if it is isomorphic to a coproduct of representable presheaves $X \cong \sum_{i \in I} y(U_i)$. We call $I$ the \textbf{index set} of $X$.
\end{Def}

\begin{Lemma} \label{lem projective presheaves}
The projective objects in $\ncat{Pre}(\cat{C})$ are precisely those presheaves $X$ which are retracts of semi-representable presheaves, equivalently are coproducts of retracts of representable presheaves.
\end{Lemma}

\begin{proof}
By the Density Lemma/coYoneda Lemma \cite[Lemma A.75]{Minichiello2025}, $X$ can be written as a colimit
\begin{equation*}
   X \cong \ncolim{y(U) \to X} \, y(U). 
\end{equation*}
But every colimit can be written as a coequalizer of coproducts
\begin{equation*}
\sum_{y(V) \to y(U) \to X} y(V) \rightrightarrows \sum_{y(U) \to X} y(U) \xrightarrow{p} X. 
\end{equation*}
But then $p$ is an epimorphism, and $X$ is projective, so there is a lift in the diagram
\begin{equation*}
    \begin{tikzcd}
	& {\sum_{y(U) \to X} y(U)} \\
	X & X
	\arrow["p", from=1-2, to=2-2]
	\arrow["i", dashed, from=2-1, to=1-2]
	\arrow[equals, from=2-1, to=2-2]
\end{tikzcd}
\end{equation*}
Therefore $pi = 1_X$, so $X$ is a retract of a coproduct of representables.

We can also consider $X$ as a coproduct of retracts of representables. Indeed, since $X$ is a sub-object of a semi-representable presheaf $\sum_i y(U_i)$, we know that $X(U) \subseteq \sum_i \cat{C}(U_i, U)$. Hence we can write $X(U) = \sum_i X_i(U)$, where $X_i(U) = X(U) \cap \cat{C}(U_i, U)$. Since $X$ is a sub-object, this decomposition is functorial, i.e. $X \cong \sum_i X_i$. The retraction 
\begin{equation*}
X \xhookrightarrow{\iota} \sum_i y(U_i) \xrightarrow{r} X   
\end{equation*}
induces retractions
\begin{equation*}
    X_i \xhookrightarrow{\iota} y(U_i) \xrightarrow{r_i} X_i.
\end{equation*}
Hence $X$ is also a coproduct of retracts of representables.
\end{proof}

\begin{Prop}[{\cite{Garner2013}}] \label{prop projective cofibrant}
Given a small category $\cat{C}$, a simplicial presheaf $X$ on $\cat{C}$ is projective cofibrant if and only if the following conditions hold:
\begin{enumerate}
    \item every $X_n$ is a projective object in $\Pre(\cat{C})$, i.e. a  retract of a semi-representable presheaf and
    \item $X$ is split.
\end{enumerate}
\end{Prop}

\begin{proof}
$(\Rightarrow)$ By definition, a projective cofibrant simplicial presheaf $X$ is a retract of a relative $I_{\text{proj}}$-cell complex $X'$. Hence $X'$ is a transfinite composition of coproducts of pushouts along maps in $I_{\text{proj}}$, see \cite[Definition 10.5.8]{Hirschhorn2009}. Note that for any simplicial set $K$ and any representable presheaf $y(U)$, the simplicial presheaf $K \otimes y(U)$ is degreewise a coproduct of representables: 
\begin{equation*}
(K \otimes y(U))_n = K_n \times y(U)_n = \sum_{x \in K_n} y(U).
\end{equation*}
Note that splitness is a condition that can be checked degreewise. So let us prove by induction that $X'$ is degreewise semi-representable and is split. Splitness is a vacuous condition in degree $0$, so we wish only to show $X'_0$ is a coproduct of representables. Since $X'$ is a $I_{\text{proj}}$-cell complex, $X'_0$ must be a pushout of the form
\begin{equation*}
    \begin{tikzcd}
	\varnothing & \varnothing \\
	{\sum_{i \in I_0} \Delta^0 \otimes y(U_i)} & {X'_0 = \text{sk}_0 X'}
	\arrow[from=1-1, to=1-2]
	\arrow[from=1-1, to=2-1]
	\arrow[from=1-2, to=2-2]
	\arrow[from=2-1, to=2-2]
	\arrow["\lrcorner"{anchor=center, pos=0.125, rotate=180}, draw=none, from=2-2, to=1-1]
\end{tikzcd}
\end{equation*}
in other words $X'_0 \cong \sum_{i \in I_0} y(U_i)$.

Now suppose that $\text{sk}_{n-1} X'$ is semi-representable and split. 

The $n$-skeleton $\text{sk}_n X'$ is a pushout of the form
\begin{equation*}
    \begin{tikzcd}
	{\sum_{i \in I_n} \partial \Delta^n \otimes y(U_i)} & {\text{sk}_{n-1} X'} \\
	{\sum_{i \in I_n} \Delta^n \otimes y(U_i)} & {\text{sk}_n X'}
	\arrow[from=1-1, to=1-2]
	\arrow[from=1-1, to=2-1]
	\arrow[from=1-2, to=2-2]
	\arrow[from=2-1, to=2-2]
	\arrow["\lrcorner"{anchor=center, pos=0.125, rotate=180}, draw=none, from=2-2, to=1-1]
\end{tikzcd}
\end{equation*}
So it is easy to see that $\text{sk}_n X'$ is degreewise a coproduct of representables. Now each of the above simplicial presheaves factor through $\ncat{sSet}_{\nd}$, hence so does their pushout $\text{sk}_n$. Thus by Lemma \ref{lem split simplicial presheaf iff factors through nondegen}, $\text{sk}_n X'$ is split.

Hence, we have proven that $X'$ is degreewise semi-representable and split. Since $X$ is a retract of $X'$, it is degreewise semi-representable and also split, as a subobject of a  simplicial presheaf factoring through $\ncat{sSet}_{\nd}$ must also factor through $\ncat{sSet}_{\nd}$.

$(\Leftarrow)$ Now suppose that $X$ is degreewise semi-representable and split. Then by definition, there exists a simplicial presheaf $X'$ that $X$ is a retract of, and since $X$ is split, $X'$ is as well. Then it is easy to see that we can construct $X'$ as a $I_{\text{proj}}$-cell complex using its splitting as the attaching maps.
\end{proof}

\begin{Cor} \label{cor representables are cofibrant}
If $U \in \cat{C}$, then $y(U)$ is projective cofibrant. Furthemore if $X$ is a retract of a representable, then it is projective cofibrant.
\end{Cor}

\begin{Lemma}[{\cite[Lemma 6.5.6]{borceux1994handbook}}] \label{rem projective presheaves}
If $\cat{C}$ is a Cauchy complete category, i.e. all of its idempotents split, then a presheaf on $\cat{C}$ is projective if and only if it is semi-representable.
\end{Lemma}

\begin{Cor}
If $\cat{C}$ is a small Cauchy complete category and $X$ is a simplicial presheaf on $\cat{C}$, then $X$ is projective cofibrant if and only if $X$ is degreewise semi-representable and split.
\end{Cor}

\begin{Rem}
Note in particular that all posets are Cauchy complete, as the only idempotents are the identity maps. Furthermore, the category $\ncat{Man}$ is Cauchy complete\footnote{In fact, $\ncat{Man}$ is the Cauchy completion of the category $\ncat{Open}$ whose objects are open subsets of some $\R^n$ with smooth maps. The map $p : \R^2 \setminus \{0 \} \to \R^2 \setminus \{0 \}$ given by $p(x) = \frac{x}{|x|}$ is idempotent, but not split in this category. But obviously it is split in $\ncat{Man}$.}, but the proof of this fact is nontrivial \cite{nlab:karoubi_envelope}.
\end{Rem}

\section{The \v{C}ech and Local Model Structures}

\subsection{The \v{C}ech Model Structure}

\begin{Def} \label{def cech nerve}
Given an object $U \in \cat{C}$ and a family of morphisms $r = \{r_i : U_i \to U \}_{i \in I}$ over $U$, let $\check{C}(r)$ denote the simplicial presheaf defined degreewise by
\begin{equation*}
    \check{C}(r)_n = \sum_{i_0, \dots, i_{n-1} \in I} y(U_{i_0}) \times_{y(U)} \dots \times_{y(U)} y(U_{i_{n-1}}),
\end{equation*}
with face maps given by projection and degeneracies given by diagonal maps. We call this the \textbf{\v{C}ech nerve} of the family $r$. There is a map $p_r : \check{C}(r) \to y(U)$ given by $r_i$ on each component of $\check{C}(r)_0 = \sum_{i \in I} y(U_i)$. Note that 
\begin{equation} \label{eq pi0 of cech nerve}
    \pi_0 \check{C}(r) \cong \overline{r} \cong \text{coeq} \left( \sum_{i,j \in I} y(U_i) \times_{y(U)} y(U_j) \rightrightarrows \sum_i y(U_i) \right).
\end{equation}
where $\overline{r}$ is the sifted closure (Definition \ref{def sifted closure}) of $r$.
\end{Def}

\begin{Lemma} \label{lem cech nerves are split}
For any family of morphisms $r = \{r_i : U_i \to U \}$, the \v{C}ech nerve is a split simplicial presheaf.
\end{Lemma}

\begin{proof}
Degenerate simplices of $\check{C}(r)$ are sections of components of the form $y(U_{i_0}) \times_{y(U)} y(U_{i_1}) \times_{y(U)} \dots \times_{y(U)} y(U_{i_n})$ where $i_p = i_q$ for some $0 \leq p, q \leq n$. Given a map $ f : U \to V$ in $\cat{C}$, $\check{C}(r)(V) \to \check{C}(r)(U)$ is just degreewise given by precomposition, so nondegenerate simplices cannot be mapped to degenerate ones.
\end{proof}

\begin{Lemma}[{\cite[Proposition A.1]{Dugger2004}}] \label{lem cech nerve weak equiv}
Given a small category $\cat{C}$ and a family of morphisms $r = \{U_i \to U \}$ over an object $U$, the coequalizer map
\begin{equation*}
    c_r: \check{C}(r) \to \pi_0 \check{C}(r) \cong \overline{r}
\end{equation*}
is an objectwise weak equivalence of simplicial presheaves.
\end{Lemma}

We let $p_r : \check{C}(r) \to y(U)$ denote the composite map
\begin{equation*}
    \check{C}(r) \xrightarrow{c_r} \pi_0 \check{C}(r) \cong \overline{r} \hookrightarrow y(U).
\end{equation*}
We note that the \v{C}ech nerve construction is functorial. Indeed, given a small category $\cat{C}$ and an object $U$, if $f : r \to t$ is a refinement of families of morphisms over $U$, then we obtain a map
\begin{equation*}
\check{C}(f) : \check{C}(r) \to \check{C}(t)    
\end{equation*}
of simplicial presheaves. This is most easily verified using the following observation. Given a family of morphisms $r = \{r_i : U_i \to U \}$ we obtain the map $ r : \sum_{i \in I} y(U_i) \to y(U)$ of presheaves, equivalently a $0$-truncated augmented presheaf (Definition \ref{def truncated, augmented simplicial presheaf}). Then the augmented simplicial presheaf $p_r : \check{C}(r) \to y(U)$ is precisely the augmented $0$-coskeleton $\cosk_0^+(r)$, see Section \ref{section augmented and truncated simplicial presheaves}.

\begin{Lemma}
Given any object $U \in \cat{C}$ and any family $r = \{ r_i :U_i \to U \}_{i \in I}$ over $U$, if the iterated pullbacks $U_{i_0} \times_U \dots \times_U U_{i_n}$ exist in $\cat{C}$, then the simplicial presheaf $\check{C}(r)$ is projective cofibrant.
\end{Lemma}

\begin{proof}
If the iterated pullbacks $U_{i_0} \times_U \dots \times_U U_{i_n}$ exist in $\cat{C}$, then since the Yoneda embedding $y: \cat{C} \to \Pre(\cat{C})$ preserves limits, $\check{C}(r)$ is degreewise a coproduct of representables, and it is split by Lemma \ref{lem cech nerves are split}. Thus by Proposition \ref{prop projective cofibrant} it is projective cofibrant.
\end{proof}

Given a site $(\cat{C}, j)$ we will now localize the projective model structure $\H(\cat{C})$ by the maps of \v{C}ech nerves. To do this we will use left Bousfield localization.

\begin{Def} \label{def derived hom}
Let $\cat{M}$ be a simplicial model category \cite[Chapter 9]{Hirschhorn2009} with $X,Y \in \cat{M}$. Fix a cofibrant and fibrant replacement functor $Q$ and $R$ respectively. Then we call the simplicial set
\begin{equation*}
    \R \cat{M}(X,Y) = \u{\cat{M}}(QX,RY)
\end{equation*}
the \textbf{derived mapping space} between $X$ and $Y$, where $\u{\cat{M}}(-,-)$ denotes the simplicially enriched Hom of $\cat{M}$. Note that up to homotopy equivalence, this construction does not depend on the particular choice of $Q$ or $R$.
\end{Def}

If $X$ is already cofibrant or $Y$ is already fibrant, then we do not need to replace them in the simplicially enriched Hom. The properties of simplicial model categories ensure that $\R \cat{M}(X,Y)$ is a Kan complex, and weak equivalences $X' \to X$ or $Y \to Y'$ induce weak equivalences on the mapping space. The derived mapping space can also be defined for arbitrary model categories, but more sophisticated constructions must be used \cite[Chapter 17]{Hirschhorn2009}.

\begin{Def}
Given a model category $\cat{M}$ and a class of maps $S$, we say that an object $Z \in \cat{M}$ is \textbf{$S$-local} if for every map $f : X \to Y$ in $S$ the induced map
\begin{equation*}
    \R\cat{M}(Y,Z) \to \R \cat{M}(X,Z)
\end{equation*}
is a weak equivalence of simplicial sets. We say that $g : A \to B$ is an \textbf{$S$-local weak equivalence} if for every $S$-local object $C$, the induced map
\begin{equation*}
    \R \cat{M}(B,C) \to \R \cat{M}(A,C)
\end{equation*}
is a weak equivalence of simplicial sets.
\end{Def}

\begin{Def}
Given a model category $\cat{M}$ with underlying category $\cat{C}$ and a class of morphisms $S$, we say that a model structure $L_S \cat{M}$ on the same underlying category $\cat{C}$ is a \textbf{left Bousfield localization} of $\cat{M}$ at $S$ if 
\begin{enumerate}
    \item the class of weak equivalences in $L_S\cat{M}$ is precisely the class of $S$-local weak equivalences, and
    \item $L_S\cat{M}$ and $\cat{M}$ have the same cofibrations.
\end{enumerate}
This automatically implies that the fibrations in $\cat{N}$ are those maps that right lift against the cofibrations that are also $S$-local weak equivalences. Clearly left Bousfield localizations are unique, and hence we can speak of \textit{the} left Bousfield localization.
\end{Def}

\begin{Lemma}[{\cite[Chapter 3.3, 3.4]{Hirschhorn2009}}] \label{lem props of bousfield localizations}
Given a model category $\cat{M}$ and a class of maps $S$, if the left Bousfield localization $L_S \cat{M}$ of $\cat{M}$ at $S$ exists, then
\begin{enumerate}
    \item every weak equivalence in $\cat{M}$ is a weak equivalence in $L_S \cat{M}$, and every morphism in $S$ is a weak equivalence in $L_S \cat{M}$,
    \item the class of trivial fibrations of $L_S \cat{M}$ equals the class of trivial fibrations of $\cat{M}$,
    \item every fibration in $L_S \cat{M}$ is a fibration in $\cat{M}$,
    \item every trivial cofibration of $L_S \cat{M}$ is a trivial cofibration of $\cat{M}$,
    \item the identity functor gives a Quillen adjunction
\begin{equation*}
   \begin{tikzcd}
	{\cat{M}} && {L_S\cat{M}}
	\arrow[""{name=0, anchor=center, inner sep=0}, "{1_{\cat{C}}}", shift left=2, from=1-1, to=1-3]
	\arrow[""{name=1, anchor=center, inner sep=0}, "{1_{\cat{C}}}", shift left=2, from=1-3, to=1-1]
	\arrow["\dashv"{anchor=center, rotate=-90}, draw=none, from=0, to=1]
\end{tikzcd} 
\end{equation*}
\item given a map $f : X \to Y$ in $\cat{M}$ between $S$-local objects, $f$ is a fibration in $L_S \cat{M}$ if and only if it is a fibration in $\cat{M}$,
\item if $\cat{M}$ is left proper, then an object $X$ is fibrant in $L_S \cat{M}$ if and only if it is $S$-local.
\end{enumerate}
\end{Lemma}

There are many sufficient conditions that ensure that a left Bousfield localization exists. The following result gives one of the more convenient conditions for our purposes.

\begin{Prop}[{Combine \cite[Proposition A.3.7.3]{Lurie2009} and \cite[Theorem 4.7]{barwick2010left}}] \label{prop bousfield loc exists}
Given a left proper, combinatorial, simplicial model category $\cat{M}$ and a (small) set $S$ of morphisms, then the left Bousfield localization $L_S \cat{M}$ exists and furthermore
\begin{enumerate}
    \item $L_S \cat{M}$ is a left proper, combinatorial, simplicial model category with the same simplicial enrichment and generating cofibrations as $\cat{M}$, and
    \item an object $X$ is fibrant in $L_S \cat{M}$ if and only if it is fibrant in $\cat{M}$ and is $S$-local.
\end{enumerate}
\end{Prop}

We will now use Proposition \ref{prop bousfield loc exists} to construct a model structure on simplicial presheaves that takes into account the underlying site structure. Suppose that $(\cat{C}, j)$ is a site, and consider the set of maps 
\begin{equation} \label{eq set of cech cover maps}
   \check{C}(j) \coloneqq \left \{ p_r : \check{C}(r) \to y(U) \; | \; U \in \cat{C}, \; r \in j(U) \right \}.
\end{equation}
Since $\check{C}(j)$ is a set, and $\H(\cat{C})$ is left proper, combinatorial and simplicial by Proposition \ref{prop projective model structure}, then by Proposition \ref{prop bousfield loc exists}, the left Bousfield localization of $\H(\cat{C})$ by $\check{C}(j)$ exists, and is left proper, combinatorial and simplicial. 
 
\begin{Def} \label{def cech model structure}
Let $\check{\H}(\cat{C}, j)$ denote the left Bousfield localization of $\H(\cat{C})$ by the set $\check{C}(j)$. We call this the \textbf{(projective) \v{C}ech model structure} on $\ncat{sPre}(\cat{C})$. We call the fibrant objects of this model structure \textbf{$\infty$-stacks}.
\end{Def}

\begin{Def} \label{def cofibrant site}
We say that a site $(\cat{C}, j)$ is \textbf{cofibrant} if for every $U \in \cat{C}$ and every $r = \{ r_i :U_i \to U \}$ in $j(U)$ the iterated pullbacks $U_{i_0 \dots i_n} \coloneqq U_{i_0} \times_U \dots \times_U U_{i_n}$ exist in $\cat{C}$ for all $n \geq 0$.
\end{Def}

\begin{Ex}
All of the sites in Example \ref{ex coverages on Man} are cofibrant. Indeed, given an open cover $\mathcal{U} = \{U_i \subseteq M \}$ of a finite dimensional smooth manifold $M$, then the fiber products $U_{i_0} \times_M \dots \times_M U_{i_n}$ exist and are given by the intersections $U_{i_0} \cap \dots \cap U_{i_n}$. Given a surjective submersion $p : N \to M$, the \v{C}ech nerve is the simplicial presheaf with $\check{C}(p)_k = N \times_M \dots \times_M N$, which is often considered rather than open covers in the theory of bundle gerbes, for example in \cite{waldorf2006bundlegerbe}, \cite{bunk2021gerbes}. Similarly, the sites in Examples \ref{ex coverages on top manifolds}, \ref{ex coverages on complex sites}, \ref{ex coverages on Top}, and \ref{ex coverages on diffeological spaces} are all cofibrant.

Now for the sites in Example \ref{ex coverages on Cart}, the pullback $U_{i_0} \times_U \dots \times_U U_{i_n}$ will not in general exist in $\ncat{Cart}$. However, if $\mathcal{U}$ is a good open cover, then this pullback exists and is isomorphic to the intersection $U_{i_0} \cap \dots \cap U_{i_n}$. Hence $(\ncat{Cart}, j_{\text{good}})$ is a cofibrant site.
\end{Ex}

\begin{Cor}
If $(\cat{C}, j)$ is a cofibrant site, then for every covering family $r \in j$, the \v{C}ech nerve $\check{C}(r)$ is projective cofibrant and hence is cofibrant in $\check{\H}(\cat{C}, j)$.
\end{Cor}

Proposition \ref{prop bousfield loc exists} gives us a criterion for when an object is fibrant in $\check{\H}(\cat{C}, j)$. Given a cofibrant site $(\cat{C}, j)$, a simplicial presheaf $X$ on $\cat{C}$ is a fibrant object in $\check{\H}(\cat{C}, j)$, i.e. $X$ is an $\infty$-stack, if and only if it is projective fibrant and for every $U \in \cat{C}$ and every covering family $r \in j(U)$ the map
\begin{equation} \label{eqn cech descent}
    \u{\sPre}(\cat{C})(y(U), X) \to \u{\sPre}(\cat{C})(\check{C}(r), X),
\end{equation}
is a weak equivalence of simplicial sets\footnote{If $(\cat{C}, j)$ is not a cofibrant site, then we must cofibrantly replace $\check{C}(r)$.}. We also say that $X$ satisfies \textbf{\v{C}ech descent}. If $(\cat{C}, j)$ is not cofibrant, then we must cofibrantly replace $\check{C}(r)$, but the definition of \v{C}ech descent still makes sense. More generally, we have the following definition.

\begin{Def} \label{def S-descent}
Given a simplicial presheaf $X$ over a small category $\cat{C}$ and a class $S$ of maps of simplicial presheaves, we say that $X$ satisfies \textbf{$S$-descent} if $X$ is $S$-local in the projective model structure $\H(\cat{C})$.
\end{Def}

Hence when $X$ is projective fibrant, and the domain and codomains of the maps in $S$ are cofibrant, checking $S$-descent is straightforward. In \cite[Theorem A5]{Dugger2004} it is shown that there are several classes of maps such that the left Bousfield localization at these classes gives the \v{C}ech model structure $\check{\H}(\cat{C}, j)$.

\begin{Prop} \label{prop equiv coverages}
Given a small category $\cat{C}$ and two coverages $j, j'$ that are equivalent, $j \simeq j'$ (Definition \ref{def equiv coverages}), a projective fibrant simplicial presheaf $X$ is a $\infty$-stack on $(\cat{C}, j)$ if and only if it is an $\infty$-stack on $(\cat{C}, j'$).
\end{Prop}

\begin{proof}
By \cite[Theorem A.6]{Dugger2004}, to show that $X$ is an $\infty$-stack over a site $(\cat{C}, j)$, it is enough to show that for every simplicial presheaf $F$, the induced map
\begin{equation*}
   \u{\ncat{sPre}}(QF, X) \to \u{\ncat{sPre}}(Qa_jF, X) 
\end{equation*}
is a weak equivalence of simplicial sets, where $a_j$ denotes the sheafification functor with respect to $j$. If $j \simeq j'$, then $a_j \cong a_{j'}$ since $a_j$ is left adjoint to inclusion, and adjoints are unique up to isomorphism, so if $X$ satisfies descent with respect to $F \to a_j F$ then it satisfies descent with respect to $F\to a_{j'}F$ and vice versa. Hence $X$ is a $j$-$\infty$-stack if and only if it is a $j'$-$\infty$-stack.
\end{proof}

\begin{Cor}
Given a small category $\cat{C}$ and two equivalent coverages $j \simeq j'$, then $\check{\H}(\cat{C}, j) = \check{\H}(\cat{C}, j')$.
\end{Cor}

\begin{proof}
The previous result shows that $\check{\H}(\cat{C}, j)$ and $\check{\H}(\cat{C}, j')$ have the same fibrant objects. Since they are both left bousfield localizations of $\H(\cat{C})$, they have the same cofibrations. Hence by \cite[Proposition 1.38]{joyal2020modelstructure}, they are equal model structures.
\end{proof}

\begin{Prop}
Given a site $(\cat{C}, j)$
\begin{equation*}
    \check{\H}(\cat{C}, j) = \check{\H}(\cat{C}, \text{sat}(j)) = \check{\H}(\cat{C}, \text{Gro}(j)).
\end{equation*}
\end{Prop}

\begin{proof}
By \cite{Minichiello2025} $\ncat{Sh}(\cat{C}, j) = \ncat{Sh}(\cat{C}, \text{sat}(j)) = \ncat{Sh}(\cat{C}, \text{Gro}(j))$ and hence their corresponding sheafification functors are naturally isomorphic, so the same argument as in Proposition \ref{prop equiv coverages} gives the result.
\end{proof}

\subsection{The Local Model Structure}
In this section, we introduce the notion of local weak equivalence, first introduced by Joyal \cite{Joyal1984} and Jardine \cite{jardine1987simplical} as a notion of weak equivalence that blends homotopical information with site-theoretic information. We then define several variants of the main objects of study for this paper: hypercovers.

\begin{Def} \label{def lifting problem}
Let $\cat{C}$ be a category, and let $i : K \to L$ and $p : X \to Y$ be morphisms in $\cat{C}$. A \textbf{lifting problem} between $i$ and $p$, consists of maps $u : K \to X$ and $v : L \to Y$ such that the following diagram commutes
\begin{equation*}
    \begin{tikzcd}
	K & X \\
	L & Y
	\arrow["u", from=1-1, to=1-2]
	\arrow["i"', from=1-1, to=2-1]
	\arrow["p", from=1-2, to=2-2]
	\arrow["v"', from=2-1, to=2-2]
\end{tikzcd}
\end{equation*}
A solution of the above lifting problem is a map $h : L \to X$, which we often denote with a dashed line as below
\begin{equation*}
    \begin{tikzcd}
	K & X \\
	L & Y
	\arrow["u", from=1-1, to=1-2]
	\arrow["i"', from=1-1, to=2-1]
	\arrow["p", from=1-2, to=2-2]
	\arrow["h"{description}, dashed, from=2-1, to=1-2]
	\arrow["v"', from=2-1, to=2-2]
\end{tikzcd}
\end{equation*}
such that $ph = v$ and $hi = u$. If such a solution exists we say that $i$ left lifts against $p$ and $p$ right lifts against $i$. We also say that $i$ and $p$ lift against each other.
\end{Def}

\begin{Rem} \label{rem space of lifting problems}
Note that for maps $i : K \to L$ and $p : X \to Y$ in a category $\cat{C}$, then $i$ and $p$ lift against each other if and only if the function 
\begin{equation} \label{eq map for solution of lifting problems}
   \pi(i,p) : \cat{C}(L, X) \to \cat{C}(L, Y) \times_{\cat{C}(K,Y)} \cat{C}(K,X)
\end{equation}
defined by $\pi(i,p)(h) = (hi, ph)$ has a section, i.e. is surjective.

More generally, if $i : K \to L$ is a map of simplicial sets and $p : X \to Y$ is a map of simplicial presheaves, and there is a lifting problem of the form
\begin{equation*}
\begin{tikzcd}
	K & X \\
	L & Y
	\arrow["u", from=1-1, to=1-2]
	\arrow["i"', from=1-1, to=2-1]
	\arrow["p", from=1-2, to=2-2]
	\arrow["v"', from=2-1, to=2-2]
\end{tikzcd}    
\end{equation*}
then we obtain a map
\begin{equation*}
    \u{\pi}(i, p) : X^L \to Y^L \times_{Y^K} X^K
\end{equation*}
induced by the universal property of the pullback, which on vertices is precisely the map $(\u{\pi}(i, p))_0 = \pi(i,p)$ from (\ref{eq map for solution of lifting problems}). Thus a solution to any lifting problem between $i$ and $p$ is equivalent to a section on the vertices of the map $\u{\pi}(i,p)$.
\end{Rem}

\begin{Def} \label{def local lifting problem}
Given a site $(\cat{C}, j)$ and a lifting problem
\begin{equation*}
    \begin{tikzcd}
	K & X \\
	L & Y
	\arrow["u", from=1-1, to=1-2]
	\arrow["i"', from=1-1, to=2-1]
	\arrow["p", from=1-2, to=2-2]
	\arrow["v"', from=2-1, to=2-2]
\end{tikzcd}    
\end{equation*}
where $i : K \to L$ is a map of simplicial sets and $p : X \to Y$ a map of simplicial presheaves, we say that $f$ has the \textbf{local right lifting property} against $i$ if for every object $U \in \cat{C}$ there exists a $j$-covering family $r = \{r_i : U_i \to U \}_{i \in I}$ and a solution to the following lifting problem
\begin{equation*}
\begin{tikzcd}
	K & {X(U)} & {X(U_i)} \\
	L & {Y(U)} & {Y(U_i)}
	\arrow[from=1-1, to=1-2]
	\arrow["i"', from=1-1, to=2-1]
	\arrow[from=1-2, to=1-3]
	\arrow["{{p(U_i)}}", from=1-3, to=2-3]
	\arrow[ dashed, from=2-1, to=1-3]
	\arrow[from=2-1, to=2-2]
	\arrow[from=2-2, to=2-3]
\end{tikzcd}
\end{equation*}
of simplicial sets, where the solution is shown as a dashed line above. We say that there is a \textbf{local solution} to the original lifting problem if there is a solution as above for every $U \in \cat{C}$.
\end{Def}

\begin{Rem}
Note that if $i : K \to L$ is a map of simplicial sets, then a map $p : X \to Y$ of simplicial presheaves has the local right lifting property with respect to $i : K \to L$ if for every $U \in \cat{C}$ and every lifting problem of the form
\begin{equation*}
    \begin{tikzcd}
	{K \otimes y(U)} & X \\
	{L \otimes y(U)} & Y
	\arrow[from=1-1, to=1-2]
	\arrow["{i \otimes 1_{y(U)}}"', from=1-1, to=2-1]
	\arrow["p", from=1-2, to=2-2]
	\arrow[from=2-1, to=2-2]
\end{tikzcd}
\end{equation*}
there is a $j$-covering family $r = \{r_i : U_i \to U \}_{i \in I}$ such that for every $i \in I$ there is a solution of the following lifting problem
\begin{equation*}
   \begin{tikzcd}
	{K \otimes y(U_i)} & {K \otimes y(U)} & X \\
	{L \otimes y(U_i)} & {L \otimes y(U)} & Y
	\arrow["{K \otimes r_i}", from=1-1, to=1-2]
	\arrow["{i \otimes 1_{y(U_i)}}"', from=1-1, to=2-1]
	\arrow[from=1-2, to=1-3]
	\arrow["p", from=1-3, to=2-3]
	\arrow[dashed, from=2-1, to=1-3]
	\arrow["{L \otimes r_i}"', from=2-1, to=2-2]
	\arrow[from=2-2, to=2-3]
\end{tikzcd} 
\end{equation*}
\end{Rem}

\begin{Def} \label{def j local fibration}
Let $(\cat{C}, j)$ be a site. A map $p: X \to Y$ of simplicial presheaves over $\cat{C}$ is a \textbf{local fibration} if every lifting problem of the form
\begin{equation*}
\begin{tikzcd}
	{\Lambda^n_k} & X \\
	{\Delta^n} & Y
	\arrow["u", from=1-1, to=1-2]
	\arrow["{i}"', from=1-1, to=2-1]
	\arrow["p", from=1-2, to=2-2]
	\arrow["v"', from=2-1, to=2-2]
\end{tikzcd}
\end{equation*}
has a local solution. We say a simplicial presheaf $X$ over a site $(\cat{C}, j)$ is \textbf{locally fibrant} if the unique map $X \to *$ is a local fibration.
\end{Def}

\begin{Def}[{\cite[Definition 3.1]{Dugger2002}}]
A lifting problem of simplicial sets
\begin{equation} \label{eq comm square for RLHP}
    \begin{tikzcd}
	K & X \\
	L & Y
	\arrow["u", from=1-1, to=1-2]
	\arrow["i"', from=1-1, to=2-1]
	\arrow["p", from=1-2, to=2-2]
	\arrow["v"', from=2-1, to=2-2]
\end{tikzcd}
\end{equation}
is said to have a \textbf{relative homotopy solution} if there exists a map $h : L \to X$ such that $hi = u$ and such that there exists a homotopy $ph \sim v$ rel $K$.

This can equivalently be expressed as asking for dashed maps making the following diagram commute
\begin{equation*}
\begin{tikzcd}
	& K \\
	L & L & X \\
	{L +_K K \times \Delta^1} && Y
	\arrow["i"', from=1-2, to=2-1]
	\arrow["i"', from=1-2, to=2-2]
	\arrow["u", from=1-2, to=2-3]
	\arrow[dashed, from=2-1, to=2-3]
	\arrow["j"', from=2-1, to=3-1]
	\arrow["j", from=2-2, to=3-1]
	\arrow["v", from=2-2, to=3-3]
	\arrow["p", from=2-3, to=3-3]
	\arrow[dashed, from=3-1, to=3-3]
\end{tikzcd}
\end{equation*}
where $j : L \to L +_K K \times \Delta^1$ is the inclusion map in the pushout
\begin{equation*}
    \begin{tikzcd}
	{K \times \Delta^0 \cong K} & {K \times \Delta^1} \\
	{L \times \Delta^0 \cong L} & {L+_K K \times \Delta^1}
	\arrow["{\iota_0}", from=1-1, to=1-2]
	\arrow["i"', from=1-1, to=2-1]
	\arrow[from=1-2, to=2-2]
	\arrow["j"', from=2-1, to=2-2]
	\arrow["\lrcorner"{anchor=center, pos=0.125, rotate=180}, draw=none, from=2-2, to=1-1]
\end{tikzcd}
\end{equation*}
\end{Def}

\begin{Lemma}[{\cite[Proposition 4.1]{Dugger2002}}]
A map $p : X \to Y$ between Kan complexes is a weak equivalence if and only if every lifting problem of the form
\begin{equation*}
    \begin{tikzcd}
	{\partial \Delta^n} & X \\
	{\Delta^n} & Y
	\arrow["u", from=1-1, to=1-2]
	\arrow["i"', from=1-1, to=2-1]
	\arrow["p", from=1-2, to=2-2]
	\arrow["v"', from=2-1, to=2-2]
\end{tikzcd}
\end{equation*}
has a relative homotopy solution.
\end{Lemma}

\begin{Def}[{\cite[Theorem 6.15]{Dugger2002}}] 
\label{def j local weak equiv}
A map $p: X \to Y$ of objectwise fibrant simplicial presheaves\footnote{We can extend this to arbitrary simplicial presheaves by applying a fibrant replacement functor objectwise.} over a site $(\cat{C}, j)$ is a \textbf{local weak equivalence} if for every lifting problem of the form
\begin{equation*}
        \begin{tikzcd}
	{\partial \Delta^n} & X \\
	{\Delta^n} & Y
	\arrow[from=1-1, to=1-2]
	\arrow["{i}"', from=1-1, to=2-1]
	\arrow["p", from=1-2, to=2-2]
	\arrow[from=2-1, to=2-2]
\end{tikzcd}
\end{equation*}
and for every $U \in \cat{C}$, there exists a covering family $r = \{r_i : U_i \to U \}$ and a relative homotopy solution of the associated lifting problem
\begin{equation*}
\begin{tikzcd}
	{\partial \Delta^n} & {X(U)} & {X(U_i)} \\
	{\Delta^n} & {Y(U)} & {Y(U_i)}
	\arrow["u", from=1-1, to=1-2]
	\arrow["i"', from=1-1, to=2-1]
	\arrow[from=1-2, to=1-3]
	\arrow["p", from=1-3, to=2-3]
	\arrow[dashed, from=2-1, to=1-3]
	\arrow["v"', from=2-1, to=2-2]
	\arrow[from=2-2, to=2-3]
\end{tikzcd}    
\end{equation*}
We call such a relative homotopy solution a \textbf{local relative homotopy solution}.
\end{Def}

We can define local weak equivalences equivalently using homotopy sheaves by \cite{Dugger2002}.

\begin{Def}
Given a site $(\cat{C}, j)$ and a simplicial presheaf $X$ on $\cat{C}$, let $\pi_0(X)$ denote the presheaf on $\cat{C}$ defined objectwise by $\pi_0(X)(U) = \pi_0(X(U))$, equivalently the coequalizer of the two face maps $d_i : X_1(U) \to X_0(U)$. For $n \geq 1$, let $\pi_n(X)$ denote the presheaf on $\cat{C}$ defined objectwise by
\begin{equation*}
    \pi_n(X)(U) = \sum_{x \in X_0(U)} \pi_n(X(U),x)
\end{equation*}
where here $\pi_n(X(U),x)$ is the $n$th-homotopy group of the geometric realization of $X(U)$ or equivalently of some Kan fibrant replacement of $X(U)$. There is a canonical map $\pi_n(X) \to X_0$.

We say that a map $f : X \to Y$ of simplicial presheaves is a \textbf{local weak equivalence} if the induced dashed map in the following commutative diagram
\begin{equation*}
\begin{tikzcd}
	{\pi_n(X)} \\
	& {X_0 \times_{Y_0} \pi_n(Y)} & {\pi_n(Y)} \\
	& {X_0} & {Y_0}
	\arrow[dashed, from=1-1, to=2-2]
	\arrow["{\pi_n(f)}", curve={height=-12pt}, from=1-1, to=2-3]
	\arrow[curve={height=12pt}, from=1-1, to=3-2]
	\arrow[from=2-2, to=2-3]
	\arrow[from=2-2, to=3-2]
	\arrow[from=2-3, to=3-3]
	\arrow["{f_0}"', from=3-2, to=3-3]
\end{tikzcd}
\end{equation*}
is a local isomorphism of presheaves.
\end{Def}

\begin{Rem}
Originally, the notion of local weak equivalence of simplicial presheaves was given by Brown in his thesis \cite{Brown1973}, where he defines a map $f : X \to Y$ of simplicial presheaves on the site $(\mathcal{O}(T), j_T)$ underlying a topological space $T$ to be a local weak equivalence if the induced map on stalks $f_x : X_x \to Y_x$ is a weak equivalence for every point $x \in T$.

Over time this definition was generalized by Joyal \cite{Joyal1984} and Jardine \cite{jardine1986grothendieck} to sites that may not have a good notion of stalks/points. Hence, now there are many equivalent ways of defining local weak equivalences of simplicial presheaves: using homotopy group sheaves as in \cite[Section 4.1]{jardine2015local} and above, topological/combinatorial weak equivalences \cite[Page 66]{jardine2015local}, stalkwise weak equivalences in the case of sites with enough points \cite{Brown1973}, Boolean localization \cite[Section 4.3]{jardine2015local} or local liftings \cite{Dugger2002}.
\end{Rem}

\begin{Def} \label{def j local trivial fibration}
We say a map $p: X \to Y$ is a \textbf{local trivial fibration} if it is both a local fibration and a local weak equivalence.
\end{Def}

\begin{Lemma} [{\cite[Proposition 7.2]{Dugger2002}}]
A map $p : X \to Y$ of simplicial presheaves on a site $(\cat{C}, j)$ is a local trivial fibration if for every $U \in \cat{C}$, there is a local solution to any lifting problem of the form
\begin{equation*}
    \begin{tikzcd}
	{\partial \Delta^n} & X \\
	{\Delta^n} & Y
	\arrow[from=1-1, to=1-2]
	\arrow["{i}"', from=1-1, to=2-1]
	\arrow["p", from=1-2, to=2-2]
	\arrow[from=2-1, to=2-2]
\end{tikzcd}
\end{equation*}
\end{Lemma}

\begin{Lemma}
Give a site $(\cat{C}, j)$, a map $p : X \to Y$ of simplicial presheaves on $\cat{C}$ is a local trivial fibration if and only if for every $n \geq 0$, the canonical map
\begin{equation} \label{eq matching object condition}
X_n \to Y_n \times_{M_n Y} M_n X
\end{equation}
is a local epimorphism of presheaves (Definition \ref{def local epi}), where $M_n X$ is the $n$-matching object (Definition \ref{def latching and matching objects}).
\end{Lemma}

\begin{proof}
Let $i : \partial \Delta^n \to \Delta^n$ denote the inclusion map. The map $p : X \to Y$ right lifts against $i \otimes 1_{y(U)}$ if and only if the map
\begin{equation*}
   \u{\pi}(i,  p) : X^{\Delta^n} \to Y^{\Delta^n } \times_{Y^{\partial \Delta^n }} X^{\partial \Delta^n } 
\end{equation*}
has a section on its vertices objectwise as in Remark \ref{rem space of lifting problems}. It is not hard to see that for $p$ to have the local right lifting property against $i$, then there needs to be a section on the vertices of $\u{\pi}(i, p)$ when restricting to a covering family. In other words, $p$ has the local right lifting property if and only if when we apply $(-)_0 : \sPre(\cat{C}) \to \Pre(\cat{C})$ to $\u{\pi}(i,p)$, we obtain a local epimorphism.
\end{proof}

\begin{Lemma} \label{lem local triv fib iff matching map loc epi}
Given a site $(\cat{C}, j)$, an augmented simplicial presheaf of the form $p : X \to {}^c Y$ is a local trivial fibration if and only if for every $n \geq 0$, the augmented matching map
\begin{equation} \label{eq augmented matching map}
    m^+_n : X_n \to M_n^+ X,
\end{equation}
is a local epimorphism.
\end{Lemma}

\begin{proof}
When $n = 0$, then $\partial \Delta^0 = \varnothing$, so $M_0 {}^cY = M_0 X = *$. But
\begin{equation*}
M_0^+X(U) = \ncat{sSet}_+(\varnothing, X(U)) = \ncat{Set}(*, X_{-1}(U)) \cong \ncat{Set}(*,Y(U)) \cong Y(U).
\end{equation*}
In other words $M_0^+ X \cong Y$. When $n = 1$, then $\partial \Delta^1 = \Delta^0 + \Delta^0$, and $M_1^+X \cong X_0 \times_Y X_0$. But (\ref{eq matching object condition}) becomes $Y \times_Y M_1 X \cong Y \times_Y (X_0 + X_0) \cong X_0 \times_Y X_0$. So in degree $0$ and $1$, (\ref{eq augmented matching map}) is a local epi if and only if (\ref{eq matching object condition}) is. For $n > 1$, $M_n^+X \cong M_nX$, and $Y \cong M_n Y$, so the equivalence holds.
\end{proof}

\begin{Def}
Given a site $(\cat{C}, j)$ and $U \in \cat{C}$, a \textbf{hypercover} of $U$, with respect to $j$ is a local trivial fibration $p : H \to y(U)$.
\end{Def}

The notion of a hypercover becomes more useful the more one can connect it to the underlying category $\cat{C}$. We next introduce further convenient conditions for a hypercover to satisfy.

\begin{Def} \label{def semi-representable map}
Recall the notion of a semi-representable presheaf (Definition \ref{def semi-representable presheaf}).
We say that a map $f : X \to Y$ of presheaves is semi-representable if both $X$ and $Y$ are semi-representable.
\end{Def}

\begin{Lemma}
Suppose that $f : X \to Y$ is a semi-representable map of presheaves, with $X \cong \sum_{i \in I} y(U_i)$ and $Y \cong \sum_{i' \in I'} y(V_{i'})$. Then there exists a function $\alpha : I \to I'$ and a family of morphisms $f_i : U_i \to V_{\alpha(i)}$ such that the following diagram commutes
\begin{equation} \label{eq basal map}
\begin{tikzcd}[ampersand replacement=\&]
	X \&\& Y \\
	{\sum_{i \in I} y(U_i)} \&\& {\sum_{i' \in I'} y(V_{i'})}
	\arrow["f", from=1-1, to=1-3]
	\arrow["\cong"', from=1-1, to=2-1]
	\arrow["\cong", from=1-3, to=2-3]
	\arrow["{\sum_{i \in I} y(f_i)}"', from=2-1, to=2-3]
\end{tikzcd}    
\end{equation}    
\end{Lemma}

\begin{proof}
This follows from the Yoneda Lemma. Indeed, we have
\begin{equation*}
\begin{aligned}
    \ncat{Pre}(\cat{C})(X,Y) &\cong \prod_{i \in I} \Pre(\cat{C})(y(U_i), \sum_{i' \in I'} y(V_{i'})) \\
    & \cong \prod_{i \in I} \sum_{i' \in I'} \cat{C}(U_i, V_{i'}).
\end{aligned}
\end{equation*}
Therefore an element of $\Pre(\cat{C})(X,Y)$ consists of an assignment to every $i \in I$ some map $f_i : U_i \to V_{i'}$. Letting $\alpha : I \to I'$ denote the corresponding function on the index sets gives the result.
\end{proof}

Thus if $f : X \to Y$ is a semi-representable map of presheaves, then there is a corresponding set $\mathbf{f} = \{f_i : U_i \to V_{\alpha(i)} \}$ of families of morphisms in $\cat{C}$. We call this the \textbf{family of morphisms induced by $f$}. This family can be written as a collection of (possibly empty) families of morphisms $\mathbf{f}(i') = \{f_i : U_i \to V_{i'} \}_{i \in \alpha^{-1}(i')}$ for each fixed $i' \in I'$.

\begin{Def} \label{def basal and covering basal maps}
Let $(\cat{C}, j)$ be a site, define the following
\begin{itemize}
    \item a \textbf{basal morphism} $f : U \to V$ in $\cat{C}$ is a morphism such that there exists a covering family $r \in j(V)$ such that $f \in r$,
    \item a \textbf{basal map of presheaves} $f : X \to Y$ on $\cat{C}$ is either an initial map, or it is semi-representable, and furthermore each induced map $f_i : U_i \to V_{\alpha(i)}$ is basal,
    \item a \textbf{covering basal map of presheaves} $f : X \to Y$ on $(\cat{C}, j)$ is a basal map of presheaves such that for each fixed $i'$ index of $Y$, if the family $\mathbf{f}(i') = \{f_i : U_i \to V_{i'}\}_{i \in \alpha^{-1}(i')}$ is nonempty, then it is $j$-covering.
\end{itemize}
\end{Def}

\begin{Def} \label{def DHI-hypercover, basal and covering basal}
We say that a map $p : H \to y(U)$ of simplicial presheaves on a site $(\cat{C}, j)$ is a
\begin{itemize}
    \item \textbf{DHI-hypercover} if it is a local trivial fibration and each $H_n$ is semi-representable,
    \item \textbf{basal hypercover} if it is a DHI-hypercover and such that each augmented matching map $H_n \to M_n^+ H$ is a basal map, and
    \item \textbf{Verdier hypercover} if it is a DHI-hypercover such that each augmented matching map $H_n \to M_n^+ H$ is a covering basal map.
\end{itemize}
We let $\ncat{Hyp}$, $\ncat{DHIHyp}$, $\ncat{BHyp}$ and $\ncat{VHyp}$ denote the classes of hypercovers, DHI-hypercovers, basal hypercovers and Verdier hypercovers respectively. If $S$ is one of the classes of maps defined above, and $y(U)$ is a representable, then we let $S_{U}$ denote the full subcategory of $\ncat{sPre}(\cat{C})_{/y(U)}$ on the maps in $S$.
\end{Def}

\begin{Rem}
The notion of DHI-hypercover and basal hypercover is due to \cite{Dugger2004}. There, the basal hypercover defintion sufficed, because they only used Grothendieck coverages, on which Verdier hypercovers and basal hypercovers agree (Lemma \ref{lem hierarchy of hypercovers}). In \cite[Section 01FZ]{Stacksprojectauthors2025}, covering basal maps of presheaves are called coverings. Furthermore, our Verdier hypercovers are precisely the Stacks Project's notion of hypercover. Ultimately this definition is due to Verdier \cite{Verdier1972}.
\end{Rem}

We have inclusions
\begin{equation*}
\ncat{VHyp}_U \subseteq \ncat{BHyp}_U \subseteq \ncat{DHIHyp}_U \subseteq \ncat{Hyp}_U
\end{equation*}
each of which is strict in general. In certain cases there are reverse inclusions.

\begin{Lemma} \label{lem basal on saturated site is covering basal}
Let $f : X \to Y$ be a basal map of presheaves on a saturated site $(\cat{C}, j)$ that is also a local epimorphism. Then $f$ is covering basal.
\end{Lemma}

\begin{proof}
Since $f$ is basal, we can write it as a collection of maps $\sum_i f_i : \sum_i y(U_i) \to \sum_{j} y(V_j)$. Thus for every fixed $j_0 \in J$, we have a commutative diagram
\begin{equation*}
    \begin{tikzcd}[ampersand replacement=\&]
	{\sum_{i \in \alpha^{-1}(j_0)} y(U_i)} \&\& {y(V_{j_0})} \\
	{\sum_i y(U_i)} \&\& {\sum_j y(V_j)}
	\arrow["{\sum_{i \in \alpha^{-1}(j_0)} f_i}", from=1-1, to=1-3]
	\arrow[hook, from=1-1, to=2-1]
	\arrow[hook, from=1-3, to=2-3]
	\arrow[""{name=0, anchor=center, inner sep=0}, "f"', from=2-1, to=2-3]
	\arrow["\lrcorner"{anchor=center, pos=0.125}, draw=none, from=1-1, to=0]
\end{tikzcd}
\end{equation*}
which is a pullback. Since $f$ is a local epimorphism, so is the top horizontal map. Thus by \cite[Lemma 6.18]{Minichiello2025} the family $\mathbf{f}(j_0)$ is $j$-covering.
\end{proof}

\begin{Lemma} \label{lem hierarchy of hypercovers}
Suppose that $(\cat{C}, j)$ is a site with $U \in \cat{C}$. If
\begin{enumerate}
    \item every map in $\cat{C}$ is basal, then $\ncat{BHyp}_U = \ncat{DHIHyp}_U$,
    \item the coverage $j$ is saturated or a Grothendieck coverage, then $\ncat{BHyp}_U = \ncat{VHyp}_U$.
\end{enumerate}
\end{Lemma}

\begin{proof}
The first claim is obvious, and the first part of the second claim follows from Lemma \ref{lem basal on saturated site is covering basal}. The second part of the second claim follows from the same argument as Lemma \ref{lem basal on saturated site is covering basal} along with \cite[Lemma 6.28]{Minichiello2025}.
\end{proof}

\begin{Cor} \label{cor DHI-hypercovers are Verdier hypercovers on top spaces}
Let $X$ be a topological space and let $p : H \to y(U)$ be a DHI-hypercover on the site $(\mathcal{O}(X), j_X)$, then $p$ is a Verdier hypercover.
\end{Cor}

\begin{proof}
This follows from Lemma \ref{lem hierarchy of hypercovers}, as all maps in $\mathcal{O}(X)$ are basal, and $(\mathcal{O}(X), j_X)$ is a saturated site.
\end{proof}

The \v{C}ech model structure $\check{\H}(\cat{C}, j)$ on simplicial presheaves over a site $(\cat{C}, j)$ (Definition \ref{def cech model structure}) mixes together the homotopy theory of simplicial sets and the sheaf theory of the underlying site. However, as discussed in Section \ref{section history}, it was not the first model structure considered for simplicial presheaves that encorporated the underlying site structure. That was obtained by Jardine in \cite{Jardine1986}. He constructed what is now called the local injective model structure on simplicial presheaves. Here we will consider the Quillen equivalent local projective model structure.

\begin{Th}[{\cite[Theorem 1.6]{Blander2001}}] \label{th existence of local projective model structure}
Given a site $(\cat{C}, j)$, there is a model structure on $\ncat{sPre}(\cat{C}, j)$ with the weak equivalences being the local weak equivalences and the cofibrations being the projective cofibrations. We call this the \textbf{local projective model structure}, and denote it by $\hat{\H}(\cat{C}, j)$. We call the fibrant objects in this model structure \textbf{$\infty$-hyperstacks}.
\end{Th}

\begin{Prop}[{\cite[Theorem 10.3.27]{Low2014Notes}}] \label{prop local proj is localization at local weak equivs}
The local projective model structure $\hat{\H}(\cat{C},j)$ is equal to the left Bousfield localization of $\H(\cat{C},j)$ at the class of local weak equivalences.
\end{Prop}

\begin{Rem} \label{rem local fibrations}
The local fibrations of Definition \ref{def j local trivial fibration} are \textit{not} the fibrations in the local projective model structure. Instead, the local weak equivalences and the local fibrations form a category of fibrant objects structure on $\ncat{sPre}(\cat{C},j)$, see Definition \ref{def cat of fibrant objects}.
\end{Rem}

\begin{Rem} \label{rem local projective fibrations}
We refer to the (trivial) fibrations in the local projective model structure $\hat{\H}(\cat{C}, j)$ as the local projective (trivial) fibrations. By Lemma \ref{lem props of bousfield localizations}, we know that every local projective fibration is a projective fibration (i.e. an objectwise fibration) and the local projective trivial fibrations are precisely the projective trivial fibrations (i.e. the objectwise trivial Kan fibrations).

In general, the local projective fibrations are not easily described. However, when the underlying site $(\cat{C}, j)$ arises from a complete, bounded, regular cd-structure on $\cat{C}$ in the sense of \cite{Voevodsky2010}, then the local projective fibrations are more easily characterized, see \cite[Section 4]{Blander2001}. In particular, given a topological space $X$, its site of open subsets $(\mathcal{O}(X), j_X)$ has a canonical complete, bounded, regular cd-structure whose distinguished squares are pullback squares in $\mathcal{O}(X)$ of the form
\begin{equation*}
\begin{tikzcd}
	{U \cap V} & V \\
	U & W
	\arrow[from=1-1, to=1-2]
	\arrow[from=1-1, to=2-1]
	\arrow[from=1-2, to=2-2]
	\arrow[from=2-1, to=2-2]
\end{tikzcd} 
\end{equation*}
where $U \cup V$ is a cover of $W$. In fact, it is enough to check the local projective fibration property on those distinguished squares of the form above when $V = W$, and $U \subseteq V$ is an inclusion. In that case, a local projective fibration is a map $f : X \to Y$ of simplicial presheaves such that for every map $g : y(U) \to y(V)$ the canonical map
\begin{equation*}
   X(V) \to X(U) \times_{Y(V)} Y(U)
\end{equation*}
is a Kan fibration, see \cite[Corollary 4.4]{Blander2001}. This is precisely what Brown-Gersten call global fibrations in \cite{Brown1973A}, and they are the fibrations for the local projective model structure on a fixed topological space. Also see \cite{Pavlov2022} for a more modern treatment of the Brown-Gersten property and its consequences for differential geometry.
\end{Rem} 

Our interest in hypercovers mostly stems from the following result.

\begin{Th}[{\cite[Theorem 6.2]{Dugger2004}}] \label{th bousfield localization at DHI hypercovers}
Given a site $(\cat{C}, j)$, the left Bousfield localization of the projective model structure $\H(\cat{C})$ by the class $\ncat{DHI}$ of DHI-hypercovers exists and is equal to the local projective model structure $\hat{\H}(\cat{C}, j)$.
\end{Th}

Let us note that we could also localize at the local trivial fibrations rather than the DHI-hypercovers.

\begin{Prop} \label{prop different descriptions of local proj}
Each of the following left Bousfield localizations on $\H(\cat{C}, j)$ are equal to the local projective model structure
\begin{equation*}
    L_{\ncat{DHIHyp}} \H(\cat{C}, j) = L_{\ncat{Hyp}} \H(\cat{C}, j) = L_{\ncat{LWeak}} \H(\cat{C}, j) = \hat{\H}(\cat{C}, j).
\end{equation*}
\end{Prop}

\begin{proof}
 Since $\H(\cat{C})$ is left proper \cite[Proposition A.3.7.3]{Lurie2009} and \cite[Theorem 4.7]{barwick2010left} imply that $S$-fibrant objects are precisely the $S$-local objects. Now we have
\begin{equation*}
    \ncat{DHIHyp} \subseteq \ncat{Hyp} \subseteq\ncat{LWeak}
\end{equation*}
which implies that 
\begin{equation*}
    \ncat{LWeak}-\text{local objects} \subseteq \ncat{Hyp}-\text{local objects} \subseteq \ncat{DHIHyp}-\text{local objects}.
\end{equation*}
However, we know that that 
\begin{equation*}
\ncat{LWeak}-\text{local objects} = \ncat{DHIHyp}-\text{local objects}
\end{equation*} 
since $\ncat{DHIHyp}$ and $\ncat{LWeak}$ have the same left Bousfield localization by Theorem \ref{th bousfield localization at DHI hypercovers} and Proposition \ref{prop local proj is localization at local weak equivs}. Hence a simplicial presheaf is $\ncat{Hyp}$-local if and only if it is an $\infty$-hyperstack, and the $\ncat{Hyp}$-local weak equivalences are precisely the local weak equivalences. Thus $L_{\ncat{Hyp}} \H(\cat{C}, j)$ exists and is equal to $\hat{\H}(\cat{C},j)$.
\end{proof}

Theorem \ref{th bousfield localization at DHI hypercovers} shows that we can obtain the local projective model structure $\hat{\H}(\cat{C}, j)$ on a site by localizing at the (DHI-)hypercovers. However, DHI-hypercovers are still rather unwieldy objects, and in some cases it can be convenient to work with the more concrete Verdier hypercovers.

In \cite[Section 9]{Dugger2004}, conditions are put on a Grothendieck pretopology $(\cat{C}, j)$--which is then called a Verdier site--that guarantee that localizing at basal hypercovers also produces the local projective model structure. Here we slightly augment these assumptions and the arguments to obtain the analogous statement for Verdier hypercovers.

\begin{Def}\label{def verdier site}
Given a small category $\cat{C}$, a \textbf{Verdier pretopology} $j$ on $\cat{C}$ is a Grothendieck pretopology (\cite[Definition 8.11]{Minichiello2025}) such that if $X \to y(U)$ is a covering basal map of presheaves, then so is the induced diagonal map
\begin{equation*}
    X \to X \times_{y(U)} X.
\end{equation*}
We call a site $(\cat{C}, j)$ a \textbf{Verdier site} if $j$ is a Verdier pretopology.
\end{Def}

\begin{Rem}
The only difference between the above definition and that of \cite[Definition 9.1]{Dugger2004} is that we replace the term basal with covering basal. This is because we allow for arbitrary coverages rather than just Grothendieck coverages. We provide full proofs of the results of this section in Appendix \ref{section verdier sites}.
\end{Rem}

\begin{Lemma}
All of the sites mentioned in Example \ref{ex sites} except for $(\ncat{Cart}, j_{\text{open}})$ and $(\ncat{Cart}, j_{\text{good}})$ are Verdier sites.
\end{Lemma}

\begin{Prop} \label{prop can refine DHI-hypercover by verdier hypercover on Verdier site}
Let $p : H \to y(U)$ be a DHI-hypercover on a Verdier site $(\cat{C}, j)$. Then it can be refined by a split, Verdier hypercover.
\end{Prop}

\begin{Cor} \label{cor localizing at DHI and split Verdier is equiv}
Given a Verdier site $(\cat{C}, j)$, the local projective model structure $\hat{\H}(\cat{C}, j)$ is equal to the left Bousfield localizations of $\H(\cat{C})$ at the class of split, Verdier hypercovers.
\end{Cor}

\begin{proof}
This follows from Proposition \ref{prop can refine DHI-hypercover by verdier hypercover on Verdier site} and \cite[Theorem 6.2]{Dugger2004}.
\end{proof}

\section{Strict Hypercompleteness} \label{section strictly hypercomplete}

Given a site $(\cat{C}, j)$, since $\check{C}(j) \subseteq \ncat{DHIHyp}$ (\ref{eq set of cech cover maps}), we see that $\hat{\H}(\cat{C}, j)$ is both a left Bousfield localization of $\check{\H}(\cat{C}, j)$ and a left Bousfield localization of $\H(\cat{C})$ at $\ncat{DHIHyp}$. Hence we have Quillen adjunctions
\begin{equation*}
\begin{tikzcd}
	{\hat{\H}(\cat{C}, j)} & {\check{\H}(\cat{C}, j)} & {\H(\cat{C})}
	\arrow[shift right, from=1-1, to=1-2]
	\arrow[shift right, from=1-2, to=1-1]
	\arrow[shift right, from=1-2, to=1-3]
	\arrow[shift right, from=1-3, to=1-2]
\end{tikzcd}
\end{equation*}
where each functor is the identity on the underlying category $\ncat{sPre}(\cat{C})$.

\begin{Def}
We say that a site $(\cat{C}, j)$ is \textbf{hypercomplete} if the above Quillen adjunction between the \v{C}ech projective and local projective model structure is a Quillen equivalence, and \textbf{strictly hypercomplete} if they coincide
\begin{equation*}
    \check{\H}(\cat{C}, j) = \hat{\H}(\cat{C}, j).
\end{equation*}
\end{Def}

\begin{Rem}
The notion of hypercompleteness is also called $t$-completeness in \cite{Rezk2010} and \cite{Toen2004}.
\end{Rem}

One of the main results of this paper, Theorem \ref{th many sites are strictly hypercomplete} implies that many sites of interest in differential geometry are strictly hypercomplete. Our argument relies on the following results.

\begin{Lemma} \label{lem refine local epi by covering basal map}
Given a site $(\cat{C}, j)$, a local epimorphism $f : X \to Y$ of presheaves, a semi-representable presheaf $Z$ and a map $g : Z \to Y$ of presheaves, there exists a covering basal map $q : P \to Z$ of presheaves making the following diagram commute
\begin{equation}
\begin{tikzcd}[ampersand replacement=\&]
	{P} \& X \\
	Z \& Y
	\arrow["p", from=1-1, to=1-2]
	\arrow["q"', from=1-1, to=2-1]
	\arrow["f", from=1-2, to=2-2]
	\arrow["g"', from=2-1, to=2-2]
\end{tikzcd} 
\end{equation}
\end{Lemma}

\begin{proof}
We can rewrite the righthand span as
\begin{equation*}
\begin{tikzcd}
	& X \\
	{\sum_{k \in K} y(W_k)} & Y
	\arrow["f", from=1-2, to=2-2]
	\arrow["{{\sum_k g_k}}"', from=2-1, to=2-2]
\end{tikzcd}
\end{equation*}
Since $f$ is a local epimorphism, this means that for every $k \in K$, there exists a covering family $a_k = \{A^k_\ell \to W_k \}_{\ell \in L_k}$ and maps $s^k_\ell : A^k_\ell \to X$ such that the following diagram commutes
\begin{equation*}
\begin{tikzcd}
	{y(A^k_\ell)} & X \\
	{y(W_k)} & Y
	\arrow["{{s^k_\ell}}", from=1-1, to=1-2]
	\arrow["{{a^k_\ell}}"', from=1-1, to=2-1]
	\arrow["f", from=1-2, to=2-2]
	\arrow[from=2-1, to=2-2]
\end{tikzcd}
\end{equation*}
Now let $P = \sum_{k \in K} \sum_{\ell \in L_k} y(A^k_\ell)$, and define the maps $p : P \to Z$ using the covering maps $a^k_\ell$ and $q : P \to X$ using the maps $s^k_\ell$. Then $q$ is clearly covering basal, and $fp = gq$.
\end{proof}

\begin{Def}[{\cite[Section 50]{Munkres2000}}]
Let $X$ be a topological space and $\mathcal{U} = \{U_i \subseteq X \}_{i \in I}$ an open cover of $X$. We say that $\mathcal{U}$ has \textbf{order} $n$ if $n$ is the smallest integer such that for every $(n+1)$-fold collection of distinct indices $\{ i_1, \dots, i_{n+1} \} \subseteq I$, it follows that $U_{i_1} \cap \dots \cap U_{i_{n+1}} = \varnothing$. The \textbf{covering dimension} of $X$ is the smallest $n$ such every open cover $\mathcal{U}$ of $X$ has a refinement $\mathcal{V} \leq \mathcal{U}$ by an open cover $\mathcal{V}$ of order $n+1$.
\end{Def}

\begin{Prop} \label{prop manifolds have covering dim n}
If $M$ is an $n$-dimensional topological manifold, then it has covering dimension $n$.
\end{Prop}

\begin{proof}
From the definition of small inductive dimension in \cite[Section 1.4]{Robinson2010}, it is easy to see that any $n$-dimensional topological manifold has small inductive dimension $n$. Then \cite[Theorems 4.3.5, 4.5.8]{VanMill1988} prove that the small inductive dimension and the Lebesgue covering dimension of any topological space agree.
\end{proof}

The next result is a relatively involved point-set topology result that we call Lurie's lemma. It will be key to our argument.

\begin{Lemma}[{\cite[Lemma 7.2.3.5]{Lurie2009}}]
Let $X$ be a paracompact Hausdorff space with an open cover $\mathscr{U} = \{U_i \}_{i \in I}$. Fix $k \geq 1$, and suppose that for every $(k+1)$-subindex $I_0 \subseteq I$ there is an open cover $\mathscr{V}_{I_0} = \{ V_b \}_{b \in B(I_0)}$ of the intersection $U(I_0)$. Then there exists an open cover $\mathscr{W} = \{W_{\overline{i}}\}_{\overline{i} \in \overline{I}}$ of $X$, which is a refinement of $\mathscr{U}$ under $\pi: \overline{I} \to I$ satisfying the following property: If $\overline{I}_0$ is a $(k+1)$-subindex of $\overline{I}$ with $\pi(\overline{I}_0) = I_0$, then there exists a $b \in B(I_0)$ such that $W(\overline{I}_0) \subseteq V_b$.
\end{Lemma}

\begin{Def}
Given maps $f : X \to Y$ and $g : Z \to Y$, we say that $f$ \textbf{refines} $g$ if there exists a map $h : X \to Z$ such that $gh = f$.
\end{Def}

\begin{Th} \label{th refine local trivs on top mfds by cech nerves}
Let $(\cat{C}, j) = (\ncat{TopMan}, j_{\text{open}})$ and let $M$ be an $n$-dimensional topological manifold. Then any $j$-local trivial fibration $p : H \to y(M)$ can be refined by the \v{C}ech nerve $\pi : \check{C}(\cat{W}) \to y(U)$ of an open covering family $\cat{W}$ of $M$. 
\end{Th}

\begin{proof}
In degree $0$, we have a cospan
\begin{equation*}
    \begin{tikzcd}
	& {H_0} \\
	{y(M)} & {y(M)}
	\arrow["{p_0}", from=1-2, to=2-2]
	\arrow[equals, from=2-1, to=2-2]
\end{tikzcd}
\end{equation*}
where $p_0$ is a local epimorphism. Hence there exists an open cover $\mathcal{U} = \{U_i \hookrightarrow M \}_{i \in I}$ of $M$ such that the diagram
\begin{equation*}
    \begin{tikzcd}
	{\sum_i y(U_i)} & {H_0} \\
	{y(M)} & {y(M)}
	\arrow[from=1-1, to=1-2]
	\arrow[from=1-1, to=2-1]
	\arrow["{p_0}", from=1-2, to=2-2]
	\arrow[equals, from=2-1, to=2-2]
\end{tikzcd}
\end{equation*}
commutes and we obtain a map $f_0 : \text{tr}_0\check{C}(\mathcal{U}) \to \text{tr}_0 H$ of simplicial presheaves over $y(M)$.

Now suppose that $\mathcal{U}_k = \{U_i \hookrightarrow M \}_{i \in I_k}$ is an open cover of $M$, for $0 \leq k \leq n-1$ and there is a map $f_k : \text{tr}_k \check{C}(\mathcal{U}_k) \to \text{tr}_k H$ over $y(M)$. We have a cospan
\begin{equation*}
    \begin{tikzcd}
	&& {H_{k+1}} \\
	{\check{C}(\mathcal{U}_k)_{k+1}} & {M_{k+1}^+\check{C}(\mathcal{U}_k)} & {M_{k+1}^+H}
	\arrow["{m_{k+1}}", from=1-3, to=2-3]
	\arrow[equals, from=2-1, to=2-2]
	\arrow["{M_{k+1}^+(f_k)}"', from=2-2, to=2-3]
\end{tikzcd}
\end{equation*}
where $m_{k+1}$ is a local epimorphism and $\check{C}(\mathcal{U}_k)_{k+1}$ is semi-representable. So from Lemma \ref{lem refine local epi by covering basal map} we obtain a covering basal map $q_{k+1} : \widetilde{H}_{k+1} \to \check{C}(\mathcal{U}_k)$ making the following diagram commute
\begin{equation*}
\begin{tikzcd}
	\widetilde{H}_{k+1} & {H_{k+1}} \\
	{\check{C}(\mathcal{U}_k)_{k+1}} & {M_{k+1}^+H}
	\arrow[from=1-1, to=1-2]
	\arrow[from=1-1, to=2-1]
	\arrow["{m_{k+1}}", from=1-2, to=2-2]
	\arrow[from=2-1, to=2-2]
\end{tikzcd}    
\end{equation*}
So for every $(k+1)$-subindex $J = \{i_0, \dots, i_k\} \subseteq I_k$, we obtain an open cover $\mathcal{V}^J = \{V^{i_0 \dots i_k}_\alpha \hookrightarrow U_{i_0 \dots i_k} \}_{\alpha \in I^J}$. Hence by Lurie's Lemma \ref{lurie's lemma}, there exists an open cover $\mathcal{U}_{k+1} = \{U_{\overline{i}} \hookrightarrow M \}_{\overline{i} \in \overline{I}_k}$ of $M$ and a refinement $\pi : \overline{I}_k \to I_k$, $\mathcal{U}_{k+1} \leq \mathcal{U}_k$ such that for every $(k+1)$-subindex $\overline{J} = \{\overline{i}_0, \dots, \overline{i}_k \} \subseteq \overline{I}_k$, where we let $\pi(\overline{i}) = i$ and $\pi(\overline{J}) = J$, $U_{\overline{i}_0 \dots \overline{i}_k} \subseteq V^{i_0 \dots i_k}_\alpha$ for some $\alpha \in I^J$. Hence we obtain a map $f_{k+1} : \text{tr}_{k+1} \check{C}(\mathcal{U}_{k+1}) \to \text{tr}_{k+1} H$ on the nondegenerate simplices via the map $\check{C}(\mathcal{U}_{k+1})_{k+1} \to \widetilde{H}_{k+1} \to H_{k+1}$, which automatically commutes with the face maps of $H$ since it is defined via the matching objects. On degenerate simplices it will already be determined by the previously defined maps since $\check{C}(\mathcal{U}_{k+1})$ is split.

So we have managed to construct an open cover $\mathcal{U}_n$ of $M$ and a map $f_n : \text{tr}_n \check{C}(\mathcal{U}_n) \to \text{tr}_n H$ over $y(M)$. Now since $M$ has covering dimension $n$ by Proposition \ref{prop manifolds have covering dim n}, this implies that $\mathcal{U}_n$ has a refinement by an open cover $\mathcal{W}$ which has order $n+1$. In other words every $(n+2)$-fold intersection of distinct elements in $\mathcal{W}$ is empty. This means that we can extend the composite map $\text{tr}_n \check{C}(\mathcal{W}) \to \text{tr}_n \check{C}(\mathcal{U}_n) \to \text{tr}_n H$ to a map $\phi : \check{C}(\mathcal{W}) \to H$ over $y(M)$. This is because the components of $\check{C}(\mathcal{W})_{n+1}$ corresponding to empty $(n+2)$-fold intersections are isomorphic to the initial presheaf and hence vanish. The elements in $\check{C}(\mathcal{W})_{n+1}$ that consist of an $(n+2)$-fold intersection with non-distinct indices are degenerate, and by construction $\phi$ is already completely determined on degenerate simplices.
\end{proof}

We note that the above proof applies \textit{mutatis mutandis} to the sites of Example \ref{ex sites}. We codify this in the following result.

\begin{Th} \label{th many sites are strictly hypercomplete}
Let $X$ be a paracompact Hausdorff topological space with finite covering dimension. Then all of the sites from Example \ref{ex sites} are strictly hypercomplete:
\begin{equation*}
    (\ncat{TopMan}, j_{\text{open}}), \quad (\mathcal{O}(X), j_X), \quad (\ncat{Man}, j_{\text{open}}), \quad (\ncat{Cart}, j_{\text{open}/\text{good}}), \quad (\C\ncat{Man}, j_{\text{open}}), \quad (\ncat{Stein}, j_{\text{open}}).
\end{equation*}
\end{Th}

\begin{proof}
Given one of the above sites $(\cat{C}, j)$ an object $U \in \cat{C}$ and a $j$-local trivial fibration $\pi : H \to y(U)$, then using the argument of Theorem \ref{th refine local trivs on top mfds by cech nerves}, we obtain an open cover $\cat{W}$ of the underlying topological space of $U$ that refines $H$.

In the case of $(\ncat{Cart}, j_\text{good})$ we need only perform one more step. We have an open cover $\mathcal{W}$ of a cartesian space $U$ refining the $j$-local trivial fibration $\pi : H \to y(U)$. Of course $\mathcal{W}$ is not guaranteed to be differentially good, but by \cite[Corollary 5.2]{Bott2013} there exists a differentiably good open cover $\mathcal{W}'$ refining $\mathcal{W}$. Hence we are able to refine $\pi$ by $\mathcal{W}'$.

Therefore in all of these sites, we are able to refine any $j$-local trivial fibration by the \v{C}ech nerve of a $j$-covering family. The result then follows from \cite[Theorem 6.2]{Dugger2004}.
\end{proof}

\section{The Plus Construction} \label{section plus construction}

Given a site $(\cat{C}, j)$ and a presheaf of sets $X$, the plus construction produces a presheaf $X^+$ defined objectwise by
\begin{equation}
X^+(U) = \ncolim{R \in \text{Gro}(j)(U)^\op} \ncat{Pre}(\cat{C})(R, X),
\end{equation}
where here $\text{Gro}(j)$ denotes the Grothendieck closure of $j$ (Definition \ref{def Grothendieck closure of a coverage}).

In general, one must apply the plus construction twice to a presheaf in order to obtain a sheaf \cite[Section 7.4]{Minichiello2025}. The plus construction defines a functor $(-)^+ : \ncat{Pre}(\cat{C}) \to \ncat{Pre}(\cat{C})$ and we call the composite functor $a = (-)^{++}$ \textbf{sheafification}.

However, there is a different formula for the sheafification functor, obtained by noticing that $\ncat{Sh}(\cat{C})(ay(U), aX) \cong aX(U)$ can be computed using the theory of calculi of fractions/left multiplicative systems \cite[Section 16.3]{Kashiwara2006}:
\begin{equation}
aX(U) \cong \ncolim{Y \to y(U)} \ncat{Pre}(\cat{C})(Y, X) \end{equation}
where the colimit is indexed over the category whose objects are local isomorphisms $f : Y \to y(U)$ and whose morphisms $g : f \to f'$ consist of arbitrary maps $g : Y \to Y'$ making the corresponding triangle commute.

There is an analogue of the above result obtained by Zhen Lin Low in the context of simplicial presheaves.

\begin{Prop}[{\cite[Proposition 6.6]{Low2015}}] \label{prop low's theorem}
Given a site $(\cat{C}, j)$ and locally fibrant simplicial presheaves $X$ and $Y$, then
\begin{equation} \label{eq low's formula 1}
\R \hat{\H}(\cat{C}, j)(X, Y) \simeq \nhocolim{H \in \ncat{Triv}_X^\op} \u{\ncat{sPre}}(\cat{C})(H, Y)
\end{equation}
where $\ncat{Triv}_X$ denotes the full subcategory of $\ncat{sPre}(\cat{C})_{/X}$ on the local trivial fibrations $p : H \to X$.
\end{Prop}

\begin{Rem}
We give here a proof sketch of this result for the reader's convenience, skipping over some details about categories of fibrant objects, see Appendix \ref{section homotopical categories}. We encourage the reader to read Low's full argument.
\end{Rem}

\begin{proof}[Proof Sketch]
The first ingredient for the proof is Thomason's Homotopy Colimit Theorem \ref{th thomason homotopy colimit theorem}. Let $F: \cat{C}^{\op} \to \ncat{Cat}$ be a diagram of small categories. Consider its \textbf{Grothendieck construction} $\int F$, which is defined as follows. Its objects are pairs $(U, x)$ where $U \in \cat{C}$ and $x \in F(U)$, and morphisms $(U,x) \to (V,y)$ are pairs $(f,g)$ where $f : U \to V$ and $g: x \to F(f)(y)$ is a morphism in $F(U)$. There is a canonical functor $\int F \to \cat{C}$ given on objects by $(U,x) \mapsto U$.

\begin{Th}[{\cite[Theorem 1.2]{Thomason1979}}] \label{th thomason homotopy colimit theorem}
Let $F: \cat{C}^{\op} \to \ncat{Cat}$ be a small (strict) diagram of small categories. Then there is a weak equivalence of simplicial sets
\begin{equation}
    \underset{U \in \cat{C}^{\op}}{\text{hocolim}} \; NF(U) \xrightarrow{\simeq} N \int F
\end{equation}
where $N$ denotes the nerve functor.
\end{Th}

Now applying this theorem with $F = {}^c \cat{C}(\pi(-),A)$, where $\pi: \ncat{Triv}_X \to \cat{C}$ is defined by $\pi(H \to X) = H$, and ${}^c \cat{C}(H,A)$ denotes the underlying hom set thought of as a discrete category, we have a weak equivalence
$$\underset{H \in \ncat{Triv}_X^{\op}}{\text{hocolim}} \; N {}^c\cat{C}(H,A) = \underset{H \in \ncat{Triv}_X^{\op}}{\text{hocolim}} \; {}^c \cat{C}(H,A) \xrightarrow{\simeq} N \int \cat{C}(\pi(-),A).$$

The objects in the category $\int \cat{C}(\pi(-), A)$ consist of pairs $(f, g)$ where $f: H \to X$ is a trivial fibration and $g : H \to A$ is an arbitrary map in $\cat{C}$, and a morphism in $\int \cat{C}(\pi(-),A)$ between $(f,g)$ and $(f',g')$ consists of a map $\alpha : H \to H'$ making the following diagram commute
\begin{equation}
    \begin{tikzcd}
	& H \\
	X && A \\
	& {H'}
	\arrow["f"', two heads, from=1-2, to=2-1]
	\arrow["{f'}", two heads, from=3-2, to=2-1]
	\arrow["g", from=1-2, to=2-3]
	\arrow["{g'}"', from=3-2, to=2-3]
	\arrow["\alpha"{description}, from=1-2, to=3-2]
\end{tikzcd}
\end{equation}
But this is precisely a description of the cocycle category $\ncat{Cocycle}(X,A)$ (Definition \ref{def cocycle category}). Hence we have a weak equivalence
\begin{equation}\label{eqn cocycle nerve is hocolim}
    \underset{H \in \ncat{Triv}_X^{\op}}{\text{hocolim}} \; N{}^c \cat{C}(H,A) \xrightarrow{\simeq} N \ncat{Cocycle}(X,A).
\end{equation}

Now notice that the inclusion map $\Delta^0 \hookrightarrow \Delta^k$ into the $0$th vertex is a trivial cofibration of simplicial sets. Thus, by the axioms of a simplicial category of fibrant objects \cite[Definition 5.1]{Low2015}, the map $A^{\Delta^k} \to A$ is a trivial fibration. By Proposition \ref{prop cocycle homs are homotopical}, we have a weak equivalence
$$N \ncat{Cocycle}(X,A^{\Delta^k}) \xrightarrow{\simeq} N \ncat{Cocycle}(X,A).$$
This implies that the induced map on homotopy colimits 
\begin{equation} \label{eqn hocolim of nerve of cocycle cat is cocycle cat}
\underset{k \in \mathsf{\Delta}^{\op}}{\text{hocolim}} \; N \ncat{Cocycle}(X,A^{\Delta^k}) \xrightarrow{\simeq} \underset{k \in \mathsf{\Delta}^{\op}}{\text{hocolim}} \; N \ncat{Cocycle}(X,A) 
\end{equation}
is a weak equivalence. Now by \cite[Corollary 1.9.6]{Low2014Notes}, we have
\begin{equation*}
 \underset{k \in \mathsf{\Delta}^{\op}}{\text{hocolim}} \; N \ncat{Cocycle}(X,A) \cong N \ncat{Cocycle}(X,A) \times N \mathsf{\Delta}^\op   
\end{equation*}
but $N\mathsf{\Delta}^\op$ is contractible, since $\mathsf{\Delta}^\op$ has a terminal object. Hence the projection map
\begin{equation*}
    \underset{k \in \mathsf{\Delta}^{\op}}{\text{hocolim}} \; N \ncat{Cocycle}(X,A) \xrightarrow{\simeq} N \ncat{Cocycle}(X,A)
\end{equation*}
is a weak equivalence.
Now since all bisimplicial sets are Reedy cofibrant \cite[Corollary 15.8.8]{Hirschhorn2009}, using the Bousfield-Kan theorem \cite[Theorem 18.7.4]{Hirschhorn2009} we can model the homotopy colimit of a simplicial diagram in $\ncat{sSet}$ using the diagonal of the corresponding bisimplicial set. Thus
\begin{equation*}
\underset{k \in \mathsf{\Delta}^{\op}}{\text{hocolim}} \; {}^c \cat{C}(H,A^{\Delta^k}) \cong \cat{C}(H,A^{\Delta^\bullet}), 
\end{equation*}
where the right hand side is the simplicial set defined degreewise by $(\cat{C}(H,A^{\Delta^\bullet}))_n \coloneqq \cat{C}(H, A^{\Delta^n})$.
Now $\u{\cat{C}}(H,A)_k = \cat{C}(H \otimes \Delta^k, A) \cong \cat{C}(H, A^{\Delta^k})$, and thus
\begin{equation} \label{eqn hocolim of mapping space is mapping space}
    \underset{k \in \mathsf{\Delta}^{\op}}{\text{hocolim}} \; {}^c \cat{C}(H,A^{\Delta^k}) \cong \u{\cat{C}}(H,A).
\end{equation}
Now replacing $A$ in (\ref{eqn cocycle nerve is hocolim}) with $A^{\Delta^k}$ and then taking the homotopy colimit of the resulting diagram over $k$ gives us the weak equivalence
\begin{equation}
    \underset{k \in \mathsf{\Delta}^{\op}}{\text{hocolim}} \,\underset{H \in \ncat{Triv}_X^{\op}}{\text{hocolim}} \; {}^c \cat{C}(H,A^{\Delta^k}) \xrightarrow{\simeq} \underset{k \in \mathsf{\Delta}^{\op}}{\text{hocolim}} \; N \ncat{Cocycle}(X,A^{\Delta^k}).
\end{equation}
Interchanging the homotopy colimits gives us by (\ref{eqn hocolim of mapping space is mapping space}) and (\ref{eqn hocolim of nerve of cocycle cat is cocycle cat}) a weak equivalence
\begin{equation}
\underset{H\in \ncat{Triv}_X^{\op}}{\text{hocolim}} \; \u{\cat{C}}(H,A) \xrightarrow{\simeq} \underset{k \in \mathsf{\Delta}^{\op}}{\text{hocolim}} \; N \ncat{Cocycle}(X,A^{\Delta^k}) \xrightarrow{\simeq} N \ncat{Cocycle}(X,A).
\end{equation}
Which gives us the formula (\ref{eq low's formula 1}).
\end{proof}

\begin{Prop} \label{prop low's theorem for simplicial presheaves}
Given a site $(\cat{C}, j)$ and projective fibrant simplicial presheaves $X,Y$ on $\cat{C}$, then there are weak equivalences of simplicial sets
\begin{equation*}
N \ncat{Cocycle}(X,Y) \xrightarrow{\simeq} L_W \cat{C}(X,Y) \simeq
\R \hat{\H}(\cat{C}, j)(X,Y).  
\end{equation*}
\end{Prop}

\begin{proof}
For the left hand equivalence
combine \cite[Theorem 3.12]{Low2015} with \cite[Lemma 6.4]{Low2015}. For the right hand equivalence see \cite[Corollary 4.7]{Dwyer1980a}, and note that the right-hand side is a zig-zag of weak equivalences.
\end{proof}

\begin{Rem} \label{rem size considerations for low's theorem}
In order for the above statements to be valid, we must really be careful about the size of the categories we are dealing with. To not belabor this point, we refer to \cite[Section 6]{Low2015} and \cite[Section 6.5]{Dugger2004}.
\end{Rem}

Given a site $(\cat{C}, j)$, the full subcategory $\sPre(\cat{C})_{j\text{fib}}$ of locally fibrant simplicial presheaves can be equipped with a simplicial category of fibrant objects structure as in \cite{Low2015}, where the weak equivalences are the local weak equivalences and the fibrations are the local fibrations. 

We note however that the map $y(U) \to *$ is a projective fibration since all discrete simplicial sets are Kan complexes, and hence a local trivial fibration. Thus if $p : H \to y(U)$ is a local fibration, then so is $H \to *$. Thus every local fibration to a representable has domain a locally fibrant simplicial presheaf.

\begin{Cor} \label{cor low's formula}
Given a site $(\cat{C}, j)$ with $U \in \cat{C}$, and a locally fibrant simplicial presheaf $A$, let $\hat{A}$ denote a fixed fibrant replacement of $A$ for the local projective model structure $\hat{\H}(\cat{C}, j)$. Then the following simplicial sets are weak equivalent
\begin{equation} \label{eq low's formula}
    \underset{H \in \ncat{Triv}^{\op}_{y(U)}}{\text{hocolim}} \, \u{\sPre}(\cat{C}, j)(H, A) \simeq \R \hat{\H}(\cat{C},j)(y(U), A) \simeq \hat{A}(U).
\end{equation}
\end{Cor}

This implies the following result.

\begin{Cor}[{\cite[Proposition 6.6]{Low2015}}]
Given a site $(\cat{C}, j)$ with $U \in \cat{C}$ and a locally fibrant simplicial presheaf $X$, let $\hat{X}$ denote the fibrant replacement of $X$ in $\hat{\H}(\cat{C},j)$, then
\begin{equation}
    \hat{X}(U) \simeq \nhocolim{H \in \ncat{Hyp}_U^\op} \u{\ncat{sPre}}(\cat{C})(H, X).
\end{equation}
\end{Cor}

This result seems less known than it ought to be. In this section we will show that it implies, for example, that on strictly hypercomplete sites, the plus construction applied once to a presheaf already results in a sheaf.

Recall the adjunction $\pi_0 : \ncat{sPre}(\cat{C}) \rightleftarrows \ncat{Pre}(\cat{C}) : {}^c (-)$ from Section \ref{section simplicial presheaves}. It is easy to see that this is a simplicially-enriched adjunction. If $(\cat{C}, j)$ is a site, and $A$ is a presheaf of sets (which implies that it is locally fibrant, since discrete simplicial sets are Kan complexes), then
\begin{equation*}
\u{\ncat{sPre}}(\cat{C})(H, {}^c A) \cong {}^c \ncat{Pre}(\cat{C})(\pi_0 H, A).
\end{equation*}
Hence
\begin{equation} \label{eq Low's formula for discrete}
     \hat{A}(U) \simeq \nhocolim{H \in \ncat{Hyp}^\op_{U}} \u{\ncat{sPre}}(\cat{C})(H, {}^c A) \cong \nhocolim{H \in \ncat{Hyp}^\op_{U}} {}^c \ncat{Pre}(\cat{C})(\pi_0 H, A).
\end{equation}

Now given $U \in \cat{C}$, the category $\ncat{Hyp}_U$ is canonically simplicially enriched. Let $[\ncat{Hyp}_U]$ denote the homotopy category of $\ncat{Hyp}_U$, i.e. $[\ncat{Hyp}_U](f,g) = \pi_0 \u{\ncat{Hyp}_U}(f,g)$. Let $q: \ncat{Hyp}_U \to [\ncat{Hyp}_U]$ denote the quotient functor. In other words, the objects of $[\ncat{Hyp}_U]$ are local trivial fibrations over $X$, and the morphisms are simplicial homotopy classes of morphisms between local trivial fibrations over $X$.

\begin{Prop}[{\cite[Proposition 8.5]{Dugger2004}}] \label{prop homotopy cat of hypercovers is cofiltered}
Given a site $(\cat{C}, j)$, the categories $[\ncat{Hyp}_U]$ and $[\ncat{DHIHyp}_U]$ are finitely cofiltered\footnote{By cofiltered here we mean that every finite diagram has a cone.}.
\end{Prop}

\begin{Rem}
The above result is quite old, proved in \cite[Expos\'{e} V, 7.3.2]{Grothendieck1960SGA} and \cite[Corollary 8.13]{Artin1969}. Further proofs are given in \cite[Proposition 2.16]{Beke2004} and \cite[Proposition 10.5.7]{Low2014Notes}.
\end{Rem}

The functor $G = \ncat{Pre}(\cat{C})(\pi_0(-), A) : \ncat{Hyp}_U^\op \to \ncat{Set}$ clearly factors through this homotopy category, so we obtain a unique functor $G'$ making the following diagram commute
\begin{equation} \label{eq G factors through hocat}
    \begin{tikzcd}
	{\ncat{Hyp}_U^\op} && {\ncat{sSet}} \\
	& {[\ncat{Hyp}_U^\op]}
	\arrow["G", from=1-1, to=1-3]
	\arrow["q^\op"', from=1-1, to=2-2]
	\arrow["{G'}"', from=2-2, to=1-3]
\end{tikzcd}
\end{equation}

\begin{Def}
We say that a functor $F: \cat{C} \to \cat{D}$ is \textbf{homotopy final} if for every $V \in \cat{D}$, the nerve of the category $(V \downarrow F)$ is weakly contractible. We say it is \textbf{homotopy initial} if $F^\op$ is homotopy final, i.e. if the nerve of $(F \downarrow V)$ is weakly contractible.
\end{Def}

\begin{Prop}[{\cite[Theorem 1.10.26]{Low2014Notes}}] \label{prop homotopy final iff induces weak equiv on hocolim}
A functor $F : \cat{C} \to \cat{D}$ is homotopy final if and only if for every diagram $G : \cat{D} \to \ncat{sSet}$ the canonical map
\begin{equation*}
\hocolim \,G \to \hocolim \,GF
\end{equation*}
is a weak equivalence.
\end{Prop}

\begin{Lemma} \label{lem quotient functor of trivial fibrations is right aspherical}
Given a site $(\cat{C}, j)$ with $U \in \cat{C}$, the functor $q : \ncat{Hyp}_{U} \to [\ncat{Hyp}_{U}]$ is homotopy initial.
\end{Lemma}

\begin{proof}
If $f : H \to y(U)$ is a local trivial fibration, then the category $(q \downarrow  f)$ is finitely cofiltered. Indeed, if $d : I \to (q \downarrow f)$ is a finite diagram, then we obtain a finite diagram $[d]$ in $[\ncat{Hyp}_U]$. By Proposition \ref{prop homotopy cat of hypercovers is cofiltered}, since $[\ncat{Hyp}_U]$ is cofiltered, then $[d]$ has a cone. The diagram $d$ gives us chosen representatives of all morphisms in the cone over $[d]$, hence we obtain a cone over $d$ in $(q \downarrow f)$, as the diagram only has to commute up to homotopy. Hence $(q \downarrow f)$ is cofiltered. Hence by \cite[Corollary 2 page 9]{Quillen1972}, $(q \downarrow f)$ is weakly contractible. Therefore $q$ is homotopy initial.
\end{proof} 

Applying Proposition \ref{prop homotopy final iff induces weak equiv on hocolim} and Lemma \ref{lem quotient functor of trivial fibrations is right aspherical} to (\ref{eq G factors through hocat}), we obtain a weak equivalence of simplicial sets
\begin{equation} \label{eq homotopy initial functor}
    \nhocolim{H \in [\ncat{Hyp}_U^\op]} {}^c \ncat{Pre}(\pi_0 H, A) \simeq \hocolim \, G' \to \hocolim\, G'q = \hocolim \, G \simeq \nhocolim{H \in \ncat{Hyp}_U^\op} {}^c \ncat{Pre}(\pi_0 H, A).
\end{equation}
We have managed to reduce the indexing category from $\ncat{Hyp}_U$ to $[\ncat{Hyp}_U]$. Now, since $[\ncat{Hyp}_U]^\op$ is finitely cofiltered, then by \cite[Proposition 7.3]{Dugger2001a}, it follows from (\ref{eq homotopy initial functor}) that
\begin{equation} \label{eq simplified hypersheafification for presheaves}
    \hat{A}(U) \simeq \nhocolim{H \in [\ncat{Hyp}_U^\op]} {}^c \ncat{Pre}(\cat{C})(\pi_0 H, A) \simeq \ncolim{H \in [\ncat{Hyp}_U^\op]}{}^c \ncat{Pre}(\cat{C})(\pi_0 H, A).
\end{equation}

Given a site $(\cat{C}, j)$ and $U \in \cat{C}$, let $[\ncat{\check{C}ech}_U]$ denote the category whose objects are \v{C}ech nerves $\check{C}(r)$, with $r \in j(U)$ and whose morphisms are simplicial homotopy classes of maps.

\begin{Prop} \label{prop on mfd cech to triv initial}
If $M$ is a finite dimensional topological manifold, considered as an object in the site $(\ncat{TopMan}, j_{\text{open}})$, then the inclusion functor $i : [\ncat{\check{C}ech}_M] \to [\ncat{Hyp}_M]$ is initial. 
\end{Prop}

\begin{proof}
We wish to show that for every trivial fibration $p : H \to y(M)$, the category $(i \downarrow p)$ is connected. First off it is nonempty by Theorem \ref{th refine local trivs on top mfds by cech nerves}. Then given any pair of refinements $\check{C}(\mathcal{U}) \to p \leftarrow \check{C}(V)$, we can always obtain a cone by considering the cover $\check{C}(\mathcal{U} \cap \mathcal{V})$.
\end{proof}

Now we come to the main result of this section.

\begin{Th} \label{th only need plus construction once on man}
If $A$ is a presheaf of sets on $(\cat{C}, j) = (\ncat{TopMan}, j_{\text{open}})$, then $A^+$ is a sheaf of sets.
\end{Th}

\begin{proof}
As a discrete simplicial presheaf $A$, following from (\ref{eq simplified hypersheafification for presheaves}) and Proposition \ref{prop on mfd cech to triv initial}, we have a weak equivalence
\begin{equation}
    \hat{A}(U) \simeq \ncolim{\check{C}(\mathcal{U}) \in [\ncat{\check{C}ech}_M^\op]} \ncat{Pre}(\cat{C})(\pi_0 \check{C}(\mathcal{U}), A).
\end{equation}
Now by \cite[Proposition 3.4]{Beke2004} there is an equivalence of categories $J(M) \simeq [\ncat{\check{C}ech}_{y(M)}]$, where $J(M)$ is the poset of $\text{Gro}(j_{\text{open}})$-covering sieves (Definition \ref{def Grothendieck closure of a coverage}). Furthermore, (\ref{eq pi0 of cech nerve}) shows that $R = \pi_0 \check{C}(\mathcal{U})$ is the $J(M)$-covering sieve generated by the covering family $\mathcal{U}$. Hence we have a weak equivalence of simplicial sets
\begin{equation}
    \hat{A}(U) \simeq \ncolim{R \in J(M)^\op} \ncat{Pre}(\cat{C})(R, A)
\end{equation}
and the right-hand side is precisely the plus construction $A^+(U)$. Hence we have an objectwise weak equivalence $\hat{A} \simeq A^+$, which implies that $A^+$ satisfies descent with respect to \v{C}ech covers, and hence is a sheaf.
\end{proof}

Again, the same argument can be used to prove the following result.

\begin{Th} \label{th plus construction is sheafification}
Given a presheaf of sets $A$ over any of the sites mentioned in Theorem \ref{th many sites are strictly hypercomplete}, the presheaf $A^+$ is a sheaf.
\end{Th}

\begin{Rem}
One might reasonably wonder if we could have simply derived the above result using the Verdier Hypercovering Theorem (\ref{eq verdier hypercovering theorem}). However, this would only allow us to obtain an isomorphism
\begin{equation*}
    \pi_0 \hat{A}(U) \cong A^+(U).
\end{equation*}
\end{Rem}

\section{Applications} \label{section applications}

The main advantage of a site $(\cat{C}, j)$ being strictly hypercomplete is the fact that one gets the best aspects of both $\check{\H}(\cat{C}, j)$ and $\hat{\H}(\cat{C}, j)$, namely one has simple and explicit descriptions of both the fibrant objects (the $\infty$-stacks) and the weak equivalences (the local weak equivalences). Besides this obvious improvement over the not strictly hypercomplete case, in which one has to choose between nice fibrant objects or nice weak equivalences, in this section we list some other interesting consequences of strict hypercompleteness.

\subsection{Deriving Cocycles in Higher Differential Geometry} \label{section higher differential geometry}

In this section we will explain how the \v{C}ech and local model structures show up in differential geometry. The main goal of this homotopical-sheaf machinery is to compute cocycles in various kinds of cohomology. Given any site $(\cat{C}, j)$ and a pair $X,A$ of simplicial presheaves on $(\cat{C}, j)$, we define the $0$th cohomology of $X$ with values in $A$ by
\begin{equation*}
    \check{H}^0_\infty(X,A) = \pi_0 \R \check{\H}(\cat{C},j)(X, A).
\end{equation*}
To obtain higher cohomology, we must assume that $A$ is $n$-fold deloopable, and then we define
\begin{equation*}
    \check{H}^n_\infty(X,A) = \pi_0 \R \check{\H}(\cat{C},j)(X,\overline{W}^n A)
\end{equation*}
see \cite[Section 5.3]{Minichiello2024} for more. With various coefficient $\infty$-stacks $A$, we obtain a huge variety of cohomology theories such as differential and nonabelian cohomology. For many examples, see \cite{Minichiello2024Obstruction}, \cite{Fiorenza2011} and \cite{Schreiber2013}.

Now let $M$ be a finite dimensional smooth manifold. We wish to set up a convenient framework to study various cocycles of $M$. Let us first consider $M$ as a representable sheaf on $(\ncat{Man}, j_{\text{open}})$. Then as above, given an $\infty$-stack $A$ on $\ncat{Man}$, we have
\begin{equation*}
    \check{H}^0_\infty(M, A) = \pi_0 \R \check{\H}(\ncat{Man}, j_{\text{open}})(M, A) = \pi_0 \u{\sPre}(\ncat{Man})({}^c M, A).
\end{equation*}
Now for concreteness, let $A$ be the $\infty$-stack of $U(1)$-bundle gerbes. More precisely, let $A$ be the $2$-nerve of the presheaf of $2$-categories $\text{Grb}_{U(1)}$ that assigns to every manifold $N$ its $2$-category of $U(1)$-bundle gerbes \cite{bunk2021gerbes}. Since $M$ is representable, it is projective cofibrant, so in order to compute $\check{H}^0(M, \text{Grb}_{U(1)})$, we must ask if $\text{Grb}_{U(1)}$ is fibrant. This is not such an easy thing to check. 

In \cite{Nikolaus2011} they first construct a bicategory of trivial $U(1)$-bundle gerbes, and then apply a plus construction for presheaves of bicategories by hand. This gives a presheaf of bicategories weak equivalent to $\text{Grb}_{U(1)}$, and allows for computation of the above cohomology. Unfortunately, this kind of strategy can not be easily generalized to higher stacks, as the complexity of the resulting higher categories explodes. Furthermore, this construction is already quite cumbersome even for presheaves of bicategories.

A different strategy, emblematic of \cite{Fiorenza2011}, is the following. Think of $M$ as a sheaf not on $(\ncat{Man}, j_{\text{open}})$, but on $(\ncat{Cart}, j_{\text{good}})$. The simplicial presheaf $\text{Grb}_{U(1)}$, considered now on $\ncat{Cart}$, now has a useful property: it is objectwise weak equivalent to a significantly simpler simplicial presheaf $\B^2 U(1)$, which can be described as applying the Dold-Kan correspodence to the presheaf of chain complexes $[U(1) \to 0 \to 0]$, see \cite[Example 6.6]{Minichiello2024Obstruction}. This is best understood using the cocycle description of $U(1)$-bundle gerbes \cite{bunk2021gerbes}.

So now we wish to compute
\begin{equation*}
    \check{H}^0_\infty(M, \text{Grb}_{U(1)}) = \pi_0 \R \check{\H}(\ncat{Cart}, j_{\text{good}})(M, \text{Grb}_{U(1)}).
\end{equation*}

Since we have an objectwise weak equivalence $\text{Grb}_{U(1)} \to \B^2 U(1)$ on $\ncat{Cart}$, then
\begin{equation*}
    \check{H}^0_{\infty}(M, \text{Grb}_{U(1)}) \cong \pi_0 \R \check{\H}(\ncat{Cart}, j_{\text{good}})(M, \B^2U(1)).
\end{equation*}

So now we ask again, is $\B^2 U(1)$ an $\infty$-stack, i.e. is it \v{C}ech fibrant? Now here we have the recent, very powerful theorem \cite[Proposition 4.13]{Pavlov2022}, \cite[Proposition 3.3.29]{sati2022equivariant} which tells us that whenever $A$ is a sheaf of abelian groups on $(\ncat{Cart}, j_{\text{good}})$, then its $n$-fold delooping, $\overline{W}^n A$--more commonly denoted $\mathbf{B}^n A$--is an $\infty$-stack on $(\ncat{Cart}, j_{\text{good}})$. Furthermore, \cite[Proposition 3.3.30]{sati2022equivariant} implies that if $A$ is a presheaf of simplicial groups which is furthermore an $\infty$-stack on $(\ncat{Cart}, j_{\text{good}})$, then $\overline{W} A = \mathbf{B}A$ is also an $\infty$-stack on $(\ncat{Cart}, j_{\text{good}})$.

Hence to compute $\check{H}^0_\infty(M, \B^2 U(1))$ we need only to find a projectively cofibrant simplicial presheaf $X$ and a \v{C}ech weak equivalence $X \to M$. The obvious candidate is the \v{C}ech nerve $\pi : \check{C}(\mathcal{U}) \to M$ where $\mathcal{U}$ is a good cover. Clearly $\check{C}(\mathcal{U})$ is projectively cofibrant, and it is easy to see that the map $\pi$ is a local trivial fibration. However, local trivial fibrations are not \textit{a priori} \v{C}ech weak equivalences\footnote{It may seem like \cite[Proposition A.4]{Dugger2004} applies here, but this result requires the codomain to be representable, which arbitrary manifolds over $(\ncat{Cart}, j_{\text{good}})$ are not.}. 

However Theorem \ref{th many sites are strictly hypercomplete} implies that \v{C}ech weak equivalences are precisely the local weak equivalences on $(\ncat{Cart}, j_{\text{good}})$. So we can finally conclude that $ \pi : \check{C}(\mathcal{U}) \to M$ is a \v{C}ech weak equivalence, and thus
\begin{equation*}
    \check{H}^0(M, \text{Grb}_{U(1)}) \cong \pi_0 \u{\sPre}(\ncat{Cart})(\check{C}(\mathcal{U}), \B^2 U(1)).
\end{equation*}
Now using the methods of \cite[Appendix B]{Minichiello2024Obstruction}, we can use this description to obtain actual cocycle data on $\mathcal{U}$.

Obviously this pipeline for computing cohomology classes is not limited to bundle gerbes or even to smooth manifolds. We see that Theorem \ref{th many sites are strictly hypercomplete} can be used to obtain a concrete description of cohomology classes from $\infty$-stacks.

\subsection{Isomorphism between \v{C}ech and Sheaf Cohomology}

Given a topological space $X$ and a sheaf $A$ of abelian groups on $X$, the classical $n$th sheaf cohomology $H^n(X,A)$ of $X$ with values in $A$ is computed by finding an injective resolution of $A$, and taking the $n$th cohomology of its global sections. Alternatively, one can also compute the $n$th \v{C}ech cohomology $\check{H}^n(X,A)$ of $X$ with values in $A$ as explained in \cite[Chapter II]{Bott2013}. There is a natural map $\check{H}^n(X,A) \to H^n(X,A)$, and by an old theorem of Godemont \cite[Theorem 5.10.1]{Godement1958}\footnote{But see \cite[Section 13.3]{Gallier2022} for a modern, readable discussion and proof.}, if $X$ is a Hausdorff, paracompact topological space, then the above map is an isomorphism for all $n \geq 0$ and all abelian sheaves $A$.

However, it was proven by Brown in \cite{Brown1973} that \begin{equation*}
    H^n(X,A) \cong \pi_0 \R \hat{\H}(\mathcal{O}(X), j_X)(X, \B^n A),
\end{equation*}
see \cite{nlab:abelian_sheaf_cohomology} for more details.

Theorem \ref{th many sites are strictly hypercomplete} implies that
\begin{equation*}
    \check{H}^n(X,A) \cong \pi_0 \R \check{\H}(\mathcal{O}(X), j_X)(X, \B^n A) \cong \pi_0 \R \hat{\H}(\mathcal{O}(X), j_X)(X, \B^n A) \cong H^n(X,A)
\end{equation*}
when $X$ is Hausdorff, paracompact and has finite covering dimension. Of course this is weaker than Godemont's theorem, but we can extend the notion of \v{C}ech cohomology and abelian sheaf cohomology from topological spaces to all sites in the obvious way. Then for the sites in Theorem \ref{th many sites are strictly hypercomplete}, \v{C}ech and sheaf cohomology coincide.

\subsection{Verdier's Hypercovering Theorem}

Now let us show how to derive the Verdier Hypercovering Theorem--originally proven by Verdier in \cite[Expose V, Section 7.4]{Grothendieck1960SGA}--in the form given in \cite{Jardine2012}, from Low's formula (Proposition \ref{prop low's theorem}). Given locally fibrant simplicial presheaves $X$ and $A$, we wish to prove that there is an isomorphism
\begin{equation} \label{eq verdier hypercovering theorem}
   \ncolim{H \in [\ncat{Triv}_X]^\op} \pi_0 \u{\ncat{sPre}}(\cat{C})(H,A) \cong [X,A], 
\end{equation}
where $[X,A]$ denotes the set of morphisms from $X$ to $A$ in the homotopy category $\text{ho} \hat{\H}(\cat{C}, j)$.

By Proposition \ref{prop low's theorem}, there is a weak equivalence
\begin{equation*}
    \nhocolim{H \in \ncat{Triv}^\op_X} \u{\ncat{sPre}}(\cat{C})(H, A) \simeq \R \hat{\H}(X,A).
\end{equation*}
By \cite[Proposition 9.5.24]{Hirschhorn2009} we know that
\begin{equation*}
    \pi_0 \R \hat{\H}(X,A) \cong [X,A].
\end{equation*}

Now, as explained in the proof of \cite[Proposition 3.1.31]{Cisinski2019}, the functor $\pi_0 : \ncat{sSet} \to \ncat{Set}$ is left Quillen when $\ncat{sSet}$ is equipped with the Kan-Quillen model structure, and $\ncat{Set}$ is equipped with the model structure where every map is a cofibration and a fibration, and where the weak equivalences are the bijections\footnote{See \cite{Camarena2020} for more about model structures on $\ncat{Set}$.}. In $\ncat{Set}$, all diagrams are derived, and hence all (co)limits are homotopy (co)limits.

Hence we have
\begin{equation}
\pi_0 \nhocolim{H \in \ncat{Triv}_X^\op} \u{\ncat{sPre}}(\cat{C})(H,A) \simeq \nhocolim{H \in \ncat{Triv}_X^\op} \pi_0 \u{\ncat{sPre}}(\cat{C})(H,A) \cong \ncolim{H \in \ncat{Triv}^\op_X} \pi_0 \u{\ncat{sPre}}(\cat{C})(H,A) \cong [X,A].
\end{equation}
This is almost the form we want, except we wish to replace the category $\ncat{Triv}_X^\op$ with $[\ncat{Triv}_X^\op]$. Consider the quotient functor $\pi : \ncat{Triv}_X \to [\ncat{Triv}_X]$, and let $G: [\ncat{Triv}_X^\op] \to \ncat{sSet}$ be the functor $G = \pi_0 \u{\ncat{sPre}}(\cat{C})(-,A)$, which is well defined because if $f$ and $g$ are simplicially homotopic then $G(f) = G(g)$. There is a canonical map
\begin{equation*}
 \colim \, G \to \colim \, G \pi^\op,
\end{equation*}
which is an isomorphism by Lemma \ref{lem quotient functor of trivial fibrations is right aspherical}. Now
\begin{equation*}
    \ncolim{H \in \ncat{Triv}_X^\op} \pi_0 \u{\ncat{sPre}}(\cat{C})(H,A) \cong \colim G \pi^\op.
\end{equation*}
Hence we obtain the Verdier Hypercovering Theorem (\ref{eq verdier hypercovering theorem}).

\subsection{Truncated Hypercovers}

In this section we will consider the case of truncated simplicial presheaves and how we can reduce the homotopy colimit from (\ref{eq low's formula}) in this case. We consider this section to contain partial results rather than a fully fleshed-out theory.

See Definition \ref{def truncated hypercover} for the definition of a $n$-truncated hypercover. Given a site $(\cat{C}, j)$ and $U \in \cat{C}$, let $\ncat{DHI}^{\leq n}_{y(U)}$ denote the category of $n$-truncated DHI-hypercovers over $U$ and let $\ncat{VHyp}^{\leq n}_{y(U)}$ denote the category of $n$-truncated Verdier hypercovers over $U$.

Recall the skeleton-coskeleton adjoint triple for simplicial sets \cite[Section IV.3.2]{Goerss2009}. This extends objectwise to an adjoint triple for simplicial presheaves
\begin{equation*}
    \begin{tikzcd}
	{\ncat{sPre}_{\leq n}(\cat{C})} && {\ncat{sPre}(\cat{C})}
	\arrow[""{name=0, anchor=center, inner sep=0}, "{\text{sk}_n}", shift left=2, curve={height=-12pt}, from=1-1, to=1-3]
	\arrow[""{name=1, anchor=center, inner sep=0}, "{\text{cosk}_n}"', shift right=2, curve={height=12pt}, from=1-1, to=1-3]
	\arrow[""{name=2, anchor=center, inner sep=0}, "{\text{tr}_n}"{description}, from=1-3, to=1-1]
	\arrow["\dashv"{anchor=center, rotate=-90}, draw=none, from=0, to=2]
	\arrow["\dashv"{anchor=center, rotate=-90}, draw=none, from=2, to=1]
\end{tikzcd}
\end{equation*}

Now let $\text{Sk}_n = \text{sk}_n \text{tr}_n$ and $\text{Cosk}_n = \text{cosk}_n \text{tr}_n$. Then we obtain the adjunction
\begin{equation*}
\begin{tikzcd}
	{\ncat{sPre}(\cat{C})} && {\ncat{sPre}(\cat{C})}
	\arrow[""{name=0, anchor=center, inner sep=0}, "{\text{Sk}_n}", shift left=2, from=1-1, to=1-3]
	\arrow[""{name=1, anchor=center, inner sep=0}, "{\text{Cosk}_n}", shift left=2, from=1-3, to=1-1]
	\arrow["\dashv"{anchor=center, rotate=-90}, draw=none, from=0, to=1]
\end{tikzcd}    
\end{equation*}

Now let us suppose that $(\cat{C}, j)$ is a site and $A$ is an objectwise fibrant $(n-1)$-truncated simplicial presheaf. This means that for every $U \in \cat{C}$ and $x_0 \in X(U)_0$, $\pi_k(X(U), x_0) = 0$ for $k \geq n$. We want to simplify the homotopy colimit from (\ref{eq low's formula}) in this case.

Now since $A$ is $(n-1)$-truncated, the unit map $A \to \text{Cosk}_n A$ is an objectwise weak equivalence, and furthermore by \cite[Proposition 3.5.3.23]{Lurie2023} $\text{Cosk}_n A$ is objectwise fibrant. Hence the map
\begin{equation*}
   \nhocolim{H \in \ncat{Triv}_{y(U)}^\op} \u{\ncat{sPre}}(\cat{C})(H, A) \to \nhocolim{H \in \ncat{Triv}_{y(U)}^\op} \u{\ncat{sPre}}(\cat{C})(H, \text{Cosk}_n A) 
\end{equation*}
is a weak equivalence.

Now if the adjunction $\text{Sk}_n \dashv \text{Cosk}_n$ were simplicial, then we could move the $\text{Cosk}_n$ in the above to become a $\text{Sk}_n$. However, this adjunction is not simplicial. Hence we have to be more clever in how we manipulate the above expression.

Let us reduce to the case of simplicial sets first for the sake of clarity. 

\begin{Lemma}
Suppose that $A$ is an $n$-coskeletal\footnote{In other words, $A \cong \text{Cosk}_n(A)$.} simplicial set, and that $H$ is an arbitrary augmented simplicial set, then there is an isomorphism of simplicial sets
\begin{equation*}
    \u{\ncat{sSet}}(H, A) \cong \ncat{sSet}\left( \text{Sk}_n H,  A^{\text{Sk}_n \Delta^\bullet} \right).
\end{equation*}
\end{Lemma}

\begin{proof}
First note that
 \begin{equation*}
    \begin{aligned}
\u{\ncat{sSet}}(H, A)_\ell & \cong \ncat{sSet}(\Delta^\ell \times H, A) \\
& \cong \ncat{sSet}(\Delta^\ell \times H, \text{Cosk}_n A) \\
& \cong \ncat{sSet}(\text{Sk}_n (\Delta^\ell \times H), A). 
    \end{aligned}
\end{equation*}
Now in general, $\text{Sk}_n$ does not preserve products. However, for any pair of simplicial sets $X$ and $Y$ we have
\begin{equation*}
    \text{Sk}_n(X \times Y) \cong \text{Sk}_n ( \text{Sk}_n(X) \times \text{Sk}_n(Y) ),
\end{equation*}
and furthermore this isomorphism is natural in $X$ and $Y$, see \cite{Wofsey2018} for instance.

Therefore we have
\begin{equation*}
    \begin{aligned}
        \ncat{sSet}(\text{Sk}_n(\Delta^\ell \times H), A) & \cong \ncat{sSet}(\text{Sk}_n(\text{Sk}_n(\Delta^\ell) \times \text{Sk}_n(H)), A) \\
        & \cong \ncat{sSet}(\text{Sk}_n(\Delta^\ell) \times \text{Sk}_n(H), \text{Cosk}_n A) \\
        & \cong \ncat{sSet}(\text{Sk}_n (\Delta^\ell) \times \text{Sk}_n(H), A) \\
        & \cong \ncat{sSet}(\text{Sk}_n(H), A^{\text{Sk}_n(\Delta^\ell)} ).
    \end{aligned}
\end{equation*}
\end{proof}

Since we are applying the (co)skeleton functors objectwise to simplicial presheaves, the above result translates to an isomorphism of simplicial presheaves: suppose that $A$ is an objectwise fibrant, $n$-coskeletal simplicial presheaf, then 
\begin{equation*}
    \u{\ncat{sPre}}(\cat{C})(H, A) \cong \ncat{sPre}(\cat{C})(\text{Sk}_n H, A^{\text{Sk}_n\Delta^\bullet} ).
\end{equation*}
Since this is true for arbitrary $H$, we obtain a natural isomorphism of functors $\ncat{Triv}_{y(U)}^{\op} \to \ncat{sSet}$
\begin{equation}
    \u{\ncat{sPre}}(\cat{C})(-, A) \cong \ncat{sPre}(\cat{C})(\text{Sk}_n(-), A^{\text{Sk}_n \Delta^\bullet} ).
\end{equation}
Hence we obtain an isomorphism
\begin{equation*}
    \nhocolim{H \in \ncat{Triv}_{y(U)}^\op} \u{\ncat{sPre}}(\cat{C})(H, A) \cong \nhocolim{H \in \ncat{Triv}_{y(U)}^\op} \ncat{sPre}(\cat{C})(\text{Sk}_n H, A^{\text{Sk}_n \Delta^\bullet}).
\end{equation*}

Now let us try and reduce the homotopy colimit of the functor 
\begin{equation*}
 \ncat{sPre}(\cat{C})(\text{Sk}_n(-), A^{\text{Sk}_n \Delta^\bullet}) : \ncat{Triv}_{y(U)}^\op \to \ncat{sSet}.   
\end{equation*}

Consider the skeleton-coskeleton adjunction from Section \ref{section augmented and truncated simplicial presheaves}
\begin{equation*}
    \begin{tikzcd}
	{\ncat{sPre}^+_{\leq n}(\cat{C})} && {\ncat{sPre}^+(\cat{C})}
	\arrow[""{name=0, anchor=center, inner sep=0}, "{\text{sk}_n^+}", curve={height=-12pt}, from=1-1, to=1-3]
	\arrow[""{name=1, anchor=center, inner sep=0}, "{\text{cosk}^+_n}"', curve={height=12pt}, from=1-1, to=1-3]
	\arrow[""{name=2, anchor=center, inner sep=0}, "{\text{tr}_n}"{description}, from=1-3, to=1-1]
	\arrow["\dashv"{anchor=center, rotate=-90}, draw=none, from=0, to=2]
	\arrow["\dashv"{anchor=center, rotate=-90}, draw=none, from=2, to=1]
\end{tikzcd}
\end{equation*}

\begin{Lemma} \label{lem cosk of truncated hypercover is hypercover}
If $p : H \to y(U)$ is an $n$-truncated hypercover for $n \geq 0$, then $\text{cosk}^+_n(p): \text{cosk}^+_n(H) \to y(U)$ is a hypercover.
\end{Lemma}

\begin{proof}
If $m \leq n$, then $\cosk^+_n(p)_m \cong p_m$. Hence to show that $\cosk_n(p)$ is a hypercover, we need only to show that the augmented matching map $(\cosk^+_n(H))_m \to M_m^+(\cosk^+_n(H))$ is a local epimorphism for all $m > n$. 

If $m > n$, then $(\cosk^+_n(H))_m \cong \ncat{sSet}^+(\text{sk}_n\Delta^m, H(-))$. From the discussion in \cite[Section 4.5]{Dugger2004}, we see that $M_n^+(H) \cong (\cosk^+_{n-1}(H))_n$. Thus 
\begin{equation*}
  M_m^+(\cosk^+_n(H)) \cong (\cosk^+_{m-1}(\cosk^+_n(H)))_m  
\end{equation*}
However, by \cite[Corollary 3.11]{Conrad2003}, $\cosk^+_{m-1}(\cosk^+_n(H)) \cong \cosk^+_n(H)$. It is then easy to see that the matching map
\begin{equation*}
(\cosk^+_n(H))_m \to M_m^+(\cosk^+_n(H)) \cong (\cosk^+_n(H))_m
\end{equation*}
is isomorphic to the identity, and is therefore a local epimorphism.
\end{proof}

Hence there is an adjunction
\begin{equation*}
    \begin{tikzcd}
	{\ncat{Hyp}_{U}^{\leq n}} && {\ncat{Hyp}_{U}}
	\arrow[""{name=0, anchor=center, inner sep=0}, "{\text{cosk}^+_n}"', shift right=2, from=1-1, to=1-3]
	\arrow[""{name=1, anchor=center, inner sep=0}, "{\text{tr}_n}"', shift right=2, from=1-3, to=1-1]
	\arrow["\dashv"{anchor=center, rotate=-90}, draw=none, from=1, to=0]
\end{tikzcd}
\end{equation*}

Now consider the functor
\begin{equation*}
   G = \ncat{sPre}(\cat{C})(\text{sk}_n(-), A^{\text{Sk}_n \Delta^\bullet}) : \ncat{Hyp}_{U}^{\leq n} \to \ncat{sSet}
\end{equation*}
We see that if $H = \text{tr}_n H'$ for some $H' \in \ncat{Hyp}_{U}$, then
\begin{equation*}
    \ncat{sPre}(\cat{C})(\text{sk}_n H, A^{\text{Sk}_n \Delta^\bullet}) \cong \ncat{sPre}(\cat{C})(\text{Sk}_n H', A^{\text{Sk}_n \Delta^\bullet}).
\end{equation*}
In other words if we let
\begin{equation*}
    G' = \ncat{sPre}(\cat{C})(\text{Sk}_n(-), A^{\text{Sk}_n \Delta^\bullet}) : \ncat{Hyp}_{U} \to \ncat{sSet}
\end{equation*}
then $G' = G \, \text{tr}_n$. By \cite[Corollary 1.11.18]{Low2014Notes}, the left adjoint functor $\text{tr}_n$ is homotopy initial. Therefore the canonical map
\begin{equation*}
    \nhocolim{H' \in \ncat{Hyp}_U^{\leq n, \,  \op}} \ncat{sPre}(\cat{C})(\text{sk}_n H', A^{\text{Sk}_n \Delta^\bullet}) \to \nhocolim{H \in \ncat{Hyp}_{U}^\op} \ncat{sPre}(\cat{C})(\text{Sk}_n H, A^{\text{Sk}_n \Delta^\bullet})
\end{equation*}
is a weak equivalence.

Using the small object argument, there exists a cofibrant replacement functor $Q : \ncat{sPre}(\cat{C}) \to \ncat{sPre}(\cat{C})$ for the projective model structure for simplicial presheaves\footnote{Note that here we are not assuming that $Q$ is Dugger's functor from \cite[Section 9]{Dugger2001}, as we need an objectwise trivial fibration $q$.}. This means that for any $X \in \ncat{sPre}(\cat{C})$ we have an objectwise trivial Kan fibration $q : QX \to X$ where $QX$ is projective cofibrant. Now suppose that $p : H \to y(U)$ is a local trivial fibration. We have an objectwise trivial fibration $q : QH \to H$ given by cofibrant replacement. Since objectwise trivial Kan fibrations are in particular local trivial fibrations, and local trivial fibrations compose, the composite map $pq : QH \to y(U)$ is a local trivial fibration. This defines a functor that we denote $ Q: \ncat{Hyp}_U \to \ncat{Hyp}_U$. We obtain a natural transformation $q : Q \Rightarrow 1_{\ncat{Hyp}_U}$. Let $\ncat{QHyp}_{U}$ denote the image of this functor. Then by \cite[Lemma 1.11.16]{Low2014Notes} the inclusion
\begin{equation*}
   \ncat{QHyp}_{U} \hookrightarrow \ncat{Hyp}_{U} 
\end{equation*}
is homotopy initial.

Now we can extend $Q$ to a functor $Q^{\leq n} : \ncat{Hyp}_{U}^{\leq n} \to \ncat{Hyp}_{U}^{\leq n}$ by first taking any $n$-truncated local trivial fibration $p : H \to y(U)$, extending it to a local trivial fibration $\text{cosk}^+_n(p) : \text{cosk}^+_n H \to y(U)$, applying $Q$ and then truncating again, i.e. $Q^{\leq n} = \text{tr}_n Q \cosk^+_n$. By the natural transformation $q$ we obtain a map $q_{\cosk^+_n H} : Q \cosk^+_n H \to \cosk^+_n H$, but $ (\cosk^+_n H)_m \cong H_m$ for $m \leq n$. Hence we still obtain a natural transformation $q^{\leq n} : Q^{\leq n} \Rightarrow 1_{\ncat{Hyp}_{U}^{\leq n}}$. By the same argument as above we obtain a homotopy initial inclusion
\begin{equation*}
    \ncat{QHyp}_{U}^{\leq n} \hookrightarrow \ncat{Hyp}_{U}^{\leq n},
\end{equation*}
where $\ncat{QHyp}_{U}^{\leq n}$ is the image of $Q^{\leq n}$.

Hence we see that
\begin{equation*}
\nhocolim{H \in \ncat{Hyp}_{U}^\op} \u{\ncat{sPre}}(\cat{C})(H, A) \simeq \nhocolim{H' \in \ncat{QHyp}_{U}^{\leq n}} \ncat{sPre}(\cat{C})(\text{sk}_n(H'), A^{\text{Sk}_n \Delta^\bullet}). 
\end{equation*}

Now note that every object in $\ncat{QHyp}_{U}$ is a projective cofibrant simplicial presheaf. These are characterized precisely in Proposition \ref{prop projective cofibrant}. Unfortunately in general, projective cofibrant simplicial presheaves are not semi-representable. 

However, if the underlying category $\cat{C}$ of the site $(\cat{C}, j)$ is Cauchy complete, then the objects in $\ncat{QHyp}_{U}$ are local trivial fibrations of the form $p : H \to y(U)$ where $H$ is split and a coproduct of representables (rather than just a retract of coproducts of representables). In other words, every object of $\ncat{QHyp}_U$ is a split, DHI-hypercover. By Remark \ref{rem projective presheaves}, for any topological space $X$, the poset $\mathcal{O}(X)$ of open subsets and the category $\ncat{Man}$ are Cauchy complete.

Furthermore, if we restrict the underlying site to be the site $(\mathcal{O}(X), j_X)$ of open subsets of a topological space $X$, then every split DHI-hypercover on $(\mathcal{O}(X), j_X)$ is a split Verdier hypercover by Lemma \ref{lem hierarchy of hypercovers}. Hence the inclusion factors
\begin{equation*}
    \ncat{QHyp}_{U}^{\leq n} \xhookrightarrow{i} \ncat{VHyp}_{U}^{\leq n} \xhookrightarrow{j} \ncat{Hyp}_{U}^{\leq n}.
\end{equation*}
But we have shown that $ji$ is homotopy initial, and $j$ is fully faithful, hence by \cite[Proposition 1.11.20]{Low2014Notes}, $i$ is homotopy initial.

Thus we have proven the following result.

\begin{Th}\label{th truncated result}
If $A$ is a objectwise fibrant, $(n-1)$-truncated, simplicial presheaf over a site $(\cat{C}, j)$, then
\begin{itemize}
    \item if $(\cat{C}, j) = (\ncat{Man}, j_\text{open})$ we have
\begin{equation}
    \hat{A}(U) \simeq \nhocolim{H \in \ncat{DHI}_{y(U)}^{\leq n, \op}} \ncat{sPre}(\ncat{Man})(\text{sk}_n H, A^{\text{Sk}_n \Delta^\bullet}).
\end{equation}
\item if $(\cat{C}, j) = (\mathcal{O}(X), j_X)$ we have 
\begin{equation}
     \hat{A}(U) \simeq \nhocolim{H \in \ncat{VHyp}_{y(U)}^{\leq n, \op}} \ncat{sPre}(\mathcal{O}(X))(\text{sk}_n H, A^{\text{Sk}_n \Delta^\bullet}).   
\end{equation}
\end{itemize}
\end{Th}

For computations truncated DHI and Verdier hypercovers are significantly simpler to work with than arbitrary local trivial fibrations. For example, in \cite[Section 3]{Beke2004}, Beke introduces the notion of \v{C}ech $n$-cover for $0 \leq n \leq 3$. These are precisely our $n$-truncated Verdier hypercovers. On the site of a topological space $(\mathcal{O}(X), j_X)$, $n$-truncated Verdier hypercovers are especially simple to describe, being inductively defined by covers of intersections of covers.

\appendix

\section{Augmented and Truncated Simplicial Presheaves} \label{section augmented and truncated simplicial presheaves}

Let us introduce some variants of simplicial presheaves.

\begin{Def} \label{def truncated, augmented simplicial presheaf}
Recall the category $\mathsf{\Delta}$ of finite ordinals and order-preserving maps. Let $\cat{C}$ be a small category.
\begin{enumerate}
    \item Given $n \geq 0$, let $\mathsf{\Delta}_{\leq n}$ denote the full subcategory of $\mathsf{\Delta}$ on those finite ordinals $[k]$ for $k \leq n$. We call a functor $X : \mathsf{\Delta}_{\leq n}^\op \to \ncat{Pre}(\cat{C})$ an \textbf{$n$-truncated simplicial presheaf} on $\cat{C}$, and let $\ncat{sPre}_{\leq n}(\cat{C})$ denote the category of $n$-truncated simplicial presheaves on $\cat{C}$ with natural transformations as morphisms,
    \item Let $\mathsf{\Delta}^+$ denote the category $\mathsf{\Delta}$ of finite ordinals along with the empty set $[-1] = \varnothing$, which is a strictly initial object. We call a functor $X :(\mathsf{\Delta}^+)^\op \to \ncat{sPre}(\cat{C})$ an \textbf{augmented simplicial presheaf} on $\cat{C}$, and let $\ncat{sPre}^+(\cat{C})$ denote the category of augmented simplicial presheaves on $\cat{C}$ with natural transformations as morphisms,
    \item Given $n \geq 0$, let $\mathsf{\Delta}^+_{\leq n}$ denote the full subcategory of $\mathsf{\Delta}^+$ on those finite ordinals $[k]$ for $k \leq n$. We call a functor $X : (\mathsf{\Delta}^+)^\op_{\leq n} \to \ncat{Pre}(\cat{C})$ a \textbf{$n$-truncated, augmented simplicial presheaf} on $\cat{C}$, and let $\ncat{sPre}^+_{\leq n}(\cat{C})$ denote the category of $n$-truncated, augmented simplicial presheaves on $\cat{C}$ with natural transformations as morphisms.
\end{enumerate}
\end{Def}

There exist canonical inclusion functors
\begin{equation*}
    i^+: \mathsf{\Delta} \hookrightarrow \mathsf{\Delta}^+, \qquad i_n : \mathsf{\Delta}_{\leq n} \hookrightarrow \mathsf{\Delta}, \qquad i^+_n :\mathsf{\Delta}^+_{\leq n} \hookrightarrow \mathsf{\Delta}^+, 
\end{equation*}
which each induce functors on the respective categories of simplicial presheaves, each of which has a left and right adjoint.

If $X$ is an ($n$-truncated) augmented simplicial presheaf, then we call the unique map $X_0 \to X_{-1}$ the \textbf{augmentation}. Note that from the simplicial identities, an augmentation is equivalent to a map $X \to {}^cX_{-1}$, where $X$ is a ($n$-truncated) simplicial presheaf and ${}^c X_{-1}$ is a discrete simplicial presheaf.

From $i^+$ we obtain the following adjoint triple
\begin{equation*}
\begin{tikzcd}
	{\ncat{sPre}(\cat{C})} & {\ncat{sPre}(\cat{C})^+}
	\arrow[""{name=0, anchor=center, inner sep=0}, "{{\pi_0}}", curve={height=-24pt}, from=1-1, to=1-2]
	\arrow[""{name=1, anchor=center, inner sep=0}, "{{(-)_*}}"', curve={height=24pt}, from=1-1, to=1-2]
	\arrow[""{name=2, anchor=center, inner sep=0}, "{{U}}"{description}, from=1-2, to=1-1]
	\arrow["\dashv"{anchor=center, rotate=-93}, draw=none, from=0, to=2]
	\arrow["\dashv"{anchor=center, rotate=-87}, draw=none, from=2, to=1]
\end{tikzcd}
\end{equation*}
where $U = (i^+)^*$ forgets the augmentation, $\pi_0$ sets $X_{-1} = \pi_0 X$, and $(-)_*$ sets $X_{-1} = *$, the terminal presheaf in degree $-1$. We will deal mostly with $(-)_*$, which we call the \textbf{trivial augmentation} and if $X$ is a simplicial presheaf, then we will silently consider it as trivially augmented.

From $i_n$, we obtain the classical skeleton-coskeleton adjoint triple for simplicial presheaves \cite[Section IV.3.2]{Goerss2009}:
\begin{equation*}
    \begin{tikzcd}
	{\ncat{sPre}_{\leq n}(\cat{C})} && {\ncat{sPre}(\cat{C})}
	\arrow[""{name=0, anchor=center, inner sep=0}, "{\text{sk}_n}", shift left=2, curve={height=-12pt}, from=1-1, to=1-3]
	\arrow[""{name=1, anchor=center, inner sep=0}, "{\text{cosk}_n}"', shift right=2, curve={height=12pt}, from=1-1, to=1-3]
	\arrow[""{name=2, anchor=center, inner sep=0}, "{\text{tr}_n}"{description}, from=1-3, to=1-1]
	\arrow["\dashv"{anchor=center, rotate=-90}, draw=none, from=0, to=2]
	\arrow["\dashv"{anchor=center, rotate=-90}, draw=none, from=2, to=1]
\end{tikzcd}
\end{equation*}
where $\text{tr}_n = i_n^*$ restricts a simplicial presheaf, and its left adjoint is the \textbf{$n$-skeleton}\footnote{Dugger et al \cite[Section 4.7]{Dugger2004} call this the $n$-degeneration functor.} functor $\text{sk}_n$ which freely adds degenerate simplices in degrees $m > n$ and its right adjoint is the \textbf{$n$-coskeleton} functor $\text{cosk}_n$, which has a unique simplex in degrees $m > n$ for every possible boundary $\partial \Delta^m \to \text{cosk}_n$. Note that both $\text{sk}_n$ and $\text{cosk}_n$ are fully faithful\footnote{In fact, whenever $F : \cat{C} \to \cat{D}$ is fully faithful, so are the left and right adjoints to the precomposition functor $F^*$.}.

We similarly obtain an augmented version of this adjoint triple
\begin{equation*}
     \begin{tikzcd}
	{\ncat{sPre}^+_{\leq n}(\cat{C})} && {\ncat{sPre}(\cat{C})^+}
	\arrow[""{name=0, anchor=center, inner sep=0}, "{\text{sk}^+_n}", shift left=2, curve={height=-12pt}, from=1-1, to=1-3]
	\arrow[""{name=1, anchor=center, inner sep=0}, "{\text{cosk}^+_n}"', shift right=2, curve={height=12pt}, from=1-1, to=1-3]
	\arrow[""{name=2, anchor=center, inner sep=0}, "{\text{tr}^+_n}"{description}, from=1-3, to=1-1]
	\arrow["\dashv"{anchor=center, rotate=-90}, draw=none, from=0, to=2]
	\arrow["\dashv"{anchor=center, rotate=-90}, draw=none, from=2, to=1]
\end{tikzcd}   
\end{equation*}
with fully faithful left and right adjoints which can be described in the same way as above.

\begin{Lemma}
Given $n \geq 0$, $m \geq 0 \; (\geq -1)$ and an $n$-truncated (augmented) simplicial presheaf $X$, we have
\begin{equation*}
   \text{sk}_n^{(+)}(X)_m \cong \ncolim{\substack{[m] \twoheadrightarrow [k] \\ k \leq n}} X_k, \qquad \text{cosk}_n^{(+)}(X)_m \cong \lim_{\substack{[k] \hookrightarrow [m] \\ k \leq n}} X_k,
\end{equation*}
where the left hand colimit is indexed over the category of surjective order-preserving maps from $[m]$ to finite ordinals $[k]$, $k \leq n$ and the right hand limit is indexed over the category of injective order-preserving maps from $[k]$ to finite ordinals $[m]$, $k \leq n$.
\end{Lemma}

\begin{proof}
Using the classical descriptions of Kan extensions \cite[Lemma A.69]{Minichiello2025}, we have
\begin{equation*}
      \text{sk}_n^{(+)}(X)_m \cong \ncolim{\substack{[m] \to [k] \\ k \leq n}} X_k, \qquad \text{cosk}_n^{(+)}(X)_m \cong \lim_{\substack{[k] \to [m] \\ k \leq n}} X_k.
\end{equation*}
We want to reduce these (co)limits to epimorphisms and monomorphisms. Let us just consider the case of $\text{sk}_n$, as the argument for $\text{cosk}_n$ is dual.

The colimit is over the category $([m] \downarrow \mathsf{\Delta}^{(+)}_{\leq n})$. Consider the full subcategory $\cat{D} = ([m] \downarrow \mathsf{\Delta}^{(+)}_{\text{surj}, \leq n})$ consisting of maps $[m] \twoheadrightarrow [k]$. We want to show that the inclusion functor $i : \cat{D} \hookrightarrow ([m] \downarrow \mathsf{\Delta}_{\leq n}^{(+)})$ is final. Hence we must show that for every $\alpha : [m] \to [k]$, $k \leq n$, that the category $(\alpha \downarrow \cat{D})$ is connected.

Factor $\alpha$ as
\begin{equation*}
    [m] \overset{e}{\twoheadrightarrow} [\ell] \xhookrightarrow{j} [k].
\end{equation*}
We want to get a map $\alpha \to e$ in $(\alpha \downarrow \cat{D})$ to show its nonempty, but $j$ points in the wrong direction. But $j$ is a map of total orders, and hence preserves both joins and meets. Hence any poset map $f : [a] \to [b]$ of finite ordinals has both a left $f_!$ and right adjoint $f_*$, defined as follows
\begin{equation*}
    f_!(y) = \text{max} \{x \in [a] \, : y \leq f(x) \}, \qquad f_*(y) = \text{min} \{ x \in [a] \, : f(x) \leq y \}. 
\end{equation*}
Furthermore these adjoints have the property that if $f : [a] \to [b]$ is injective, then $f_!(f(x)) = f_*(f(x)) = x$, which implies that $f_!$ and $f_*$ are surjective, and if $f$ is surjective, then $f(f_!(y)) = f(f_*(y)) = y$, which implies that $f_!$ and $f_*$ are injective.

So given $\alpha$ as above we have a diagram
\begin{equation*}
\begin{tikzcd}
	{[m]} & {[\ell]} \\
	{[k]}
	\arrow["e", two heads, from=1-1, to=1-2]
	\arrow["\alpha"', from=1-1, to=2-1]
	\arrow[""{name=0, anchor=center, inner sep=0}, "j"', hook', from=1-2, to=2-1]
	\arrow[""{name=1, anchor=center, inner sep=0}, "{j_!}"', shift right=3, two heads, from=2-1, to=1-2]
	\arrow["\dashv"{anchor=center, rotate=135}, draw=none, from=1, to=0]
\end{tikzcd}    
\end{equation*}
with $j_! \alpha = j_! j e = e$. Hence $(\alpha \downarrow \cat{D})$ is nonempty. To show it is connected, it is enough to see that every epi-mono factorization is connected by a zig-zag, which implies the same for the mono's adjoint inverse, and this follows from \cite[Lemma 2.9]{riehl2014theory}.
\end{proof}

Now let $\text{Sk}^{(+)}_n = \text{sk}^{(+)}_n \text{tr}^{(+)}_n$ and $\text{Cosk}^{(+)}_n = \text{cosk}^{(+)}_n \text{tr}^{(+)}_n$. Then we obtain the adjunctions
\begin{equation*}
\begin{tikzcd}
	{\ncat{sPre}(\cat{C})} && {\ncat{sPre}(\cat{C})} && {\ncat{sPre}^+(\cat{C})} && {\ncat{sPre}^+(\cat{C})}
	\arrow[""{name=0, anchor=center, inner sep=0}, "{{\text{Sk}_n}}", shift left=2, from=1-1, to=1-3]
	\arrow[""{name=1, anchor=center, inner sep=0}, "{{\text{Cosk}_n}}", shift left=2, from=1-3, to=1-1]
	\arrow["{{\text{Sk}^+_n}}", shift left=2, from=1-5, to=1-7]
	\arrow["{{\text{Cosk}^+_n}}", shift left=2, from=1-7, to=1-5]
	\arrow["\dashv"{anchor=center, rotate=-90}, draw=none, from=0, to=1]
\end{tikzcd}  
\end{equation*}

\begin{Def} \label{def latching and matching objects}
Given $n \geq 0$ and a simplicial presheaf $X$ on a small category $\cat{C}$, we define its $n$-th (augmented) \textbf{latching object} as the presheaf
\begin{equation*}
 L_n^{(+)} X = (\text{Sk}^{(+)}_{n-1}(X))_n   
\end{equation*}
and similarly, the $n$-th (augmented) \textbf{matching object} is the presheaf
\begin{equation*}
    M_n^{(+)} X = (\text{Cosk}^{(+)}_{n-1}(X))_n.
\end{equation*}
By the universal properties of (co)limits, we obtain maps
\begin{equation*}
    L_n^{(+)} X \xrightarrow{\ell^{(+)}_n} X_n \xrightarrow{m_n^{(+)}} M_n^{(+)} X
\end{equation*}
of presheaves, called the (augmented) latching and matching maps, respectively.
\end{Def}

The next couple of results will allow us to manipulate latching and matching objects with greater ease.

\begin{Lemma}[{\cite[Examples 3.14, 3.22]{riehl2014theory}}]
Given an (augmented) simplicial presheaf $X$ on a small category $\cat{C}$, the (augmented) latching object $L_n^{(+)} X$ is the presheaf such that for $U \in \cat{C}$
\begin{equation*}
    L_n^{(+)}X(U) = \{\text{degenerate $n$-simplices of }X(U)\},
\end{equation*}
and similarly for the (augmented) matching object
\begin{equation*}
    M_n^{(+)} X(U) = \ncat{sSet}^{(+)}(\partial \Delta^n, X(U)),
\end{equation*}
with the convention that $\partial \Delta^n$ is given the trivial augmentation.
The (augmented latching map $\ell_n^{(+)} : L_n^{(+)} X \to X_n$ is the inclusion map for degenerate simplices and the (augmented) matching map $m_n^{(+)} : X_n \to M_n^{(+)} X$ takes a simplex $\Delta^n \to X$ to its boundary $\partial \Delta^n \hookrightarrow \Delta^n \to X$.
\end{Lemma}

\begin{Lemma}[{\cite[Section 4.5]{Dugger2004}}] \label{lem props of matching and latching}
Given an augmented simplicial presheaf $X$ over a small category $\cat{C}$, we have the following
\begin{enumerate}
    \item $M_0^+ X \cong X_{-1}$,
    \item $M_1^+ X \cong X_0 \times_{X_{-1}} X_0$,
    \item $M_n^+ X \cong M_n X$ for $n \geq 2$,
    \item $L_n^+ X \cong L_n X$ for $n \geq 0$, and
    \item $L_0^+ X \cong L_0 X \cong \varnothing$.
\end{enumerate}
\end{Lemma}

Using the Eilenberg-Zilber Lemma, we can restate the above results as follows. Given an (augmented) ($n$-truncated) simplicial presheaf $X$, we have an isomorphism
\begin{equation} \label{eq structure of latching object}
    L_k^{(+)}X \cong \sum_{\sigma : [k] \twoheadrightarrow [m]} X^\nd_m,
\end{equation}
where $X^{\nd}_m$ denotes the sub-presheaf of $X_m$ consisting of non-degenerate simplices, and where the coproduct is indexed over all surjective, order-preserving maps $[k] \to [m]$ with $0 \leq m < k \; (\leq n)$.

\begin{Rem}
The notion of split simplicial presheaves (Definition \ref{def split}) still makes sense for augmented simplicial presheaves, as there are no degeneracies in degrees $n = -1, 0$, and similarly for $n$-truncated and $n$-truncated, augmented simplicial presheaves.
\end{Rem}

\begin{Lemma}
An ($n$-truncated) (augmented) simplicial presheaf $X$ is split if and only if the canonical map
\begin{equation*}
X^{\nd}_k + L_k^{(+)} X \to X_k
\end{equation*}
is an isomorphism for every $0 \leq k \; (\leq n)$.
\end{Lemma}

Let $p : X \to {}^c Y$ be a degreewise semi-representable ($n$-truncated) augmented simplicial presheaf on a small category $\cat{C}$. Then $X_k \cong \sum_{i \in I_k} y(U_i)$ for every $-1 \leq k \; (\leq n)$. Let $\text{ind}(X)$ denote the underlying ($n$-truncated) augmented simplicial set obtained from the index sets of the coproducts in every degree. We call this the \textbf{index simplicial set} associated with $X$. Given a map $a: K \to \text{ind}(X)$, we obtain a diagram $a : \nd(K)^\op \to (\ncat{Pre}(\cat{C}) \downarrow Y)$, where $\nd(K)$ is the category of non-degenerate simplices of $K$, and $(\ncat{Pre}(\cat{C}) \downarrow
Y)$ is the category of presheaf maps to $Y$. More precisely, $\nd(K)$ is the category whose objects are simplices $x : \Delta^k \to K$ and morphisms are commutative triangles. The functor $a$ takes a non-degenerate $k$-simplex $\sigma$ of $K$ and gives the map $y(U_{a(\sigma)}) \to Y$, where $y(U_{a(\sigma)})$ is the representable corresponding to the simplex $a(\sigma) \in \text{ind}(X)_k$, and the map to $Y$ is given by the augmentation of $X$. Let $X(a)$ denote the limit of this diagram.

\begin{Lemma}[{\cite[Proposition 4.15]{Dugger2004}}] \label{lem characterization of matching object}
Given $0 \leq k \; (\leq n+1)$, and a degreewise semi-representable ($n$-truncated) augmented simplicial presheaf $X$, there is an isomorphism
\begin{equation*}
    M_k^+ X \cong  \sum_{a : \partial \Delta^k \to \text{ind}(X)} X(a),
\end{equation*}
furthermore for $0 \leq k \; (\leq n)$, the matching map $m^+_k : X_k \to M_k^+ X$ is given by the limit cone maps on the representables.
\end{Lemma}

\section{Verdier Sites} \label{section verdier sites}

In this section we provide a full proof of Proposition \ref{prop can refine DHI-hypercover by verdier hypercover on Verdier site}.

\begin{Rem}
We will need several technical results in this section, which follow implicitly from results in \cite{Dugger2004}. But since we are introducing new notions such as covering basal maps and Verdier hypercovers, we give full proofs. Let us also remind the reader that we are using the published version of \cite{Dugger2004} which has different numbering than the arXiv version.
\end{Rem}

\begin{Lemma}[{\cite[Lemma 9.1]{Dugger2004}}] \label{lem pulling back basal maps of presheaves}
Suppose that $(\cat{C}, j)$ is a  pullback-stable site, $f : X \to Y$ is a basal map of presheaves, $g : Z \to Y$ is a semi-representable map of presheaves, and the following commutative diagram is a pullback
\begin{equation} \label{eq pullback of basal is basal}
\begin{tikzcd}[ampersand replacement=\&]
	{P} \& X \\
	Z \& Y
	\arrow["p", from=1-1, to=1-2]
	\arrow["q"', from=1-1, to=2-1]
	\arrow["\lrcorner"{anchor=center, pos=0.125}, draw=none, from=1-1, to=2-2]
	\arrow["f", from=1-2, to=2-2]
	\arrow["g"', from=2-1, to=2-2]
\end{tikzcd} 
\end{equation}
then the map $q$ is basal. Furthermore if $f$ is covering basal, then so is $q$.
\end{Lemma}

\begin{proof}
We can rewrite the righthand span as
\begin{equation*}
\begin{tikzcd}[ampersand replacement=\&]
	\& {\sum_{i \in I} y(U_i)} \\
	{\sum_{k \in K} y(W_k)} \& {\sum_{j \in J} y(V_j)}
	\arrow["{\sum_i f_i}", from=1-2, to=2-2]
	\arrow["{\sum_k g_k}"', from=2-1, to=2-2]
\end{tikzcd} 
\end{equation*}
Now $\Pre(\cat{C})$ has universal colimits, hence
\begin{equation*}
\begin{aligned}
    \sum_{k \in K} y(W_k) \times_{\sum_{j \in J} y(V_j)} \sum_{i \in I} y(U_i) & \cong \sum_{k \in K} \left( y(W_k) \times_{\sum_{j \in J} y(V_j)} \sum_{i \in I} y(U_i) \right) \\
    & \cong \sum_{k \in K} \sum_{i \in I} \left( y(W_k) \times_{\sum_{j \in J} y(V_j)} y(U_i) \right)
\end{aligned}
\end{equation*}
We now wish to show that each summand is representable. So fix $i \in I$, $k \in K$ and let $Z \in \cat{C}$
\begin{equation*}
\begin{aligned}
    \Pre(\cat{C})(y(Z), y(U_i) \times_{\sum_j y(V_j)} y(W_k)) & \cong \cat{C}(Z, U_i) \times_{\sum_j \cat{C}(Z, V_j)} \cat{C}(Z, W_k) \\
\end{aligned}
\end{equation*}
Now the right hand side is isomorphic to the set of pairs of maps $p : Z \to U_i$ and $q : Z \to W_k$ such that $f_i p = g_k q$. Let $\alpha : I \to J$ and $\beta: K \to J$ be the corresponding index maps. If $\alpha(i) \neq \beta(k)$, then it is impossible for $f_i p = g_k q$, and therefore the summand $y(U_i) \times_{\sum_j y(V_j)} y(W_k)$ is isomorphic to the initial presheaf $\varnothing$. If however $\alpha(i) = \beta(k)$, then we have the commutative diagram
\begin{equation*}
    \begin{tikzcd}[ampersand replacement=\&]
	{y(W_k) \times_{y(V_{\alpha(i)})} y(U_i)} \& {y(U_i)} \\
	{y(W_k)} \& {y(V_{\alpha(i)})} \\
	{y(W_k)} \& {\sum_{j \in J} y(V_j)}
	\arrow[from=1-1, to=1-2]
	\arrow[from=1-1, to=2-1]
	\arrow["{f_i}", from=1-2, to=2-2]
	\arrow["{g_k}"', from=2-1, to=2-2]
	\arrow[equals, from=2-1, to=3-1]
	\arrow[hook, from=2-2, to=3-2]
	\arrow["{g_k}"', from=3-1, to=3-2]
\end{tikzcd}
\end{equation*}
where both squares are pullbacks, and hence the outer rectangle is a pullback. Since the Yoneda embedding preserves pullbacks, and pullbacks of basal maps are assumed to exist in $\cat{C}$, we have 
\begin{equation*}
    y(W_k) \times_{y(V_{\alpha(i)})} y(U_i) \cong y(W_k \times_{V_{\alpha(i)}} U_i),
\end{equation*}
and furthermore the map $W_k \times_{V_{\alpha(i)}} U_i \to W_k$ is basal.
Hence
\begin{equation*}
    \sum_{k \in K} \sum_{i \in I} \left( y(W_k) \times_{\sum_{j \in J} y(V_j)} y(U_i) \right) \cong \sum_{i \in I, k \in K, \alpha(i) = \beta(k)} y(W_k \times_{V_{\alpha(i)}} U_i).
\end{equation*}
Therefore the map $q$ in (\ref{eq pullback of basal is basal}), which is componentwise the result of pulling back a basal map, is basal. 

Now if $f$ is covering basal, then for each fixed $k$, the family $\{ W_k \times_{V_{\beta(k)}} U_i \to W_k \}_{i \in \alpha^{-1}(\beta(k))}$ is $j$-covering, since $j$ is a pullback-stable coverage and $\{U_{i'} \to V_{\beta(k)} \}_{i' \in \alpha^{-1}(\beta(k))}$ is $j$-covering by assumption. Thus the pulled back map $q$ is covering basal.
\end{proof}

\begin{Def} \label{def truncated hypercover}

Given a site $(\cat{C}, j)$, with $U \in \cat{C}$, let $\ncat{Hyp}^{\leq n}_{U}$ denote the category of \textbf{$n$-truncated hypercovers}. These are $n$-truncated, augmented simplicial presheaves $p : H \to y(U)$ (Section \ref{section augmented and truncated simplicial presheaves}) such that $H_k \to M_k^+ H$ is a local epimorphism for all $k \leq n$. Let us also extend this definition in the obvious way to $n$-truncated DHI, basal and Verdier hypercovers.
\end{Def}

We say that a simplicial set $K$ is \textbf{finite} if it has finitely many nondegenerate simplices. We say a finite simplicial set $K$ has \textbf{dimension $n$} if its largest nondegenerate simplex belongs to $K_n$. We say that $K$ has dimension $-1$ if it is the empty simplicial set.

If $H$ is an $n$-truncated, augmented simplicial presheaf, and $K$ is a simplicial set of dimension at most $n$, then let $H_+^K$ denote the presheaf of sets
\begin{equation*}
    H_+^K = \ncat{sSet}^+_{\leq n}(\text{tr}_n K, H(-)),
\end{equation*}
where $\text{tr}_n K$ denotes the $n$-truncation of $K$, which is assumed to be trivially augmented.

\begin{Lemma}[{\cite[Lemma 9.2]{Dugger2004}}] \label{lem matching object is representable for basal hypercovers}
Suppose that $p: H \to y(U)$ is an $n$-truncated basal hypercover on a pullback-stable site $(\cat{C}, j)$, and let $K$ be a finite simplicial set of dimension $n$. Then $H^K_+$ is semi-representable.
\end{Lemma}

\begin{proof}
We will prove this by induction on the dimension of $k$. If the dimension of $K$ is $-1$, then $K_m = \varnothing$ for all $m \geq 0$ and $K_{-1} = *$, so $H_+^K \cong H_{-1} \cong y(U)$. This proves the base case.

Let us now prove the inductive step. Suppose that $H_+^L$ is semi-representable for every simplicial set $L$ of dimension $n-1$. If $K$ is a simplicial set of dimension $n$, then we can construct $K$ from $L = \text{sk}_{n-1} K$ by adding finitely many $n$-simplices to $L$ via a sequence of pushouts of the form
\begin{equation*}
    \begin{tikzcd}[ampersand replacement=\&]
	{\partial \Delta^n} \& L \\
	{\Delta^n} \& {L'}
	\arrow[from=1-1, to=1-2]
	\arrow[hook, from=1-1, to=2-1]
	\arrow[from=1-2, to=2-2]
	\arrow[from=2-1, to=2-2]
	\arrow["\lrcorner"{anchor=center, pos=0.125, rotate=180}, draw=none, from=2-2, to=1-1]
\end{tikzcd}
\end{equation*}
Let us show that if $L'$ is produced via a pushout from an $(n-1)$-dimensional simplicial set $L$ as above, then $H^{L'}_+$ is semi-representable. We wish to show that the corresponding diagram
\begin{equation*}
    \begin{tikzcd}[ampersand replacement=\&]
	{H^{L'}_+} \& {H^{\Delta^n}_+} \\
	{H^L_+} \& {H^{\partial \Delta^n}_+}
	\arrow[from=1-1, to=1-2]
	\arrow[from=1-1, to=2-1]
	\arrow["\lrcorner"{anchor=center, pos=0.125}, draw=none, from=1-1, to=2-2]
	\arrow[from=1-2, to=2-2]
	\arrow[from=2-1, to=2-2]
\end{tikzcd}
\end{equation*}
is a pullback of presheaves. But this is clear because limits of presheaves are computed objectwise and $H^{(-)}_+ \cong \ncat{sSet}^+_{\leq n}(\text{tr}_n(-), H)$ sends colimits to limits since $\text{tr}_n$ is a left adjoint (see Section \ref{section augmented and truncated simplicial presheaves}).

Now by assumption $H_n \to M_n^+H$ is basal. But this is precisely the right hand vertical map above. By the induction hypothesis, $H^L_+$ is semi-representable. So by Lemma \ref{lem pulling back basal maps of presheaves}, $H^{L'}_+$ is semi-representable.
\end{proof}

\begin{Cor}
Suppose that $p: H \to y(U)$ is an $n$-truncated DHI-hypercover on a pullback-stable site $(\cat{C}, j)$. Fix $n \geq 0$ and suppose that the augmented matching map
\begin{equation*}
    H_k \to M_k^+ H
\end{equation*}
is basal for all $k \leq n$. Then the matching object $M_{n+1}^+ H$ is semi-representable.
\end{Cor}

To proceed further, we will need to put extra conditions on our underlying site $(\cat{C}, j)$. We note that the following definition is slightly different than that of \cite[Definition 9.1]{Dugger2004}. This slight difference has an important consequence however, as now we need to reprove \cite[Proposition 9.3 and Theorem 9.4]{Dugger2004} using this new definition. We need the covering basal condition rather than just the basal condition to obtain whole covering families, rather than just subsets of covering families.

\begin{Def}[{\cite[Definition 9.1]{Dugger2004}}] 
Given a small category $\cat{C}$, a \textbf{Verdier pretopology} $j$ on $\cat{C}$ is a Grothendieck pretopology (Definition \ref{def grothendieck pretopology}) such that if $X \to y(U)$ is a covering basal map of presheaves, then so is the induced diagonal map
\begin{equation*}
    X \to X \times_{y(U)} X.
\end{equation*}
We call a site $(\cat{C}, j)$ a \textbf{Verdier site} if $j$ is a Verdier pretopology.
\end{Def}

Many sites of interest are Verdier, see Section \ref{section examples of sites} for the full collection of Verdier sites of interest. Let us prove one example to show how this works.

\begin{Lemma}
The site $(\ncat{Man}, j_{\text{open}})$ is a Verdier site.
\end{Lemma}

\begin{proof}
Suppose that $X \to y(M)$ is a $j_{\text{open}}$-covering basal map of presheaves, where $M$ is a finite dimensional smooth manifold. Then $X \cong \sum_i y(U_i)$, and the corresponding induced family $\mathbf{r} = \{r_i : U_i \to M \}$ is a $j_{\text{open}}$-covering family, so it is a collection of open embeddings. Now we consider the induced diagonal map $\Delta_X: X \to X \times_{y(M)} X$. This is isomorphic to the map $\Delta_X: \sum_i y(U_i) \to \sum_{i,j} y(U_i \times_M U_j)$, which factors through $\sum_i y(U_i) \to \sum_i y(U_i \times_M U_i)$. This map, when restricted to each component $U_i$, is isomorphic to the identity $U_i \to U_i \times_M U_i \cong U_i$, since each map $r_i : U_i \to M$ is a monomorphism in $\ncat{Man}$. None of the $U_i \times_M U_j$ with $i \neq j$ are ``hit'' by the diagonal map, hence the diagonal map is covering basal since each isomorphism $\{ U_i \to U_i \times_M U_i \}$ is a $j_{\text{open}}$-covering family.
\end{proof}

\begin{Lemma}[{\cite[Proposition 9.3]{Dugger2004}}] \label{lem building covering basal maps from maps of simplicial sets}
Given a map $f : K \to L$ of finite simplicial sets whose dimensions are at most $k$, let $p : H \to y(U)$ be a $k$-truncated Verdier hypercover on a Verdier site. Then the map $H^f_+ : H^L_+ \to H^K_+$ is covering basal.
\end{Lemma}

\begin{proof}
Given a $k$-truncated Verdier hypercover $p : H \to y(U)$ over a Verdier site $(\cat{C}, j)$, let $C$ be the class of maps $f : K \to L$ of finite simplicial sets of dimension at most $k$ such that $H^f_+ : H^L_+ \to H^K_+$ is covering basal. We want to show that $C$ consists of all maps between at most dimension $k$ simplicial sets.

We know that all the maps $\partial \Delta^n \to \Delta^n$ with $n \leq k$ belong to $C$, since $H^{\Delta^n}_+ \to H^{\partial \Delta^n}_+$ is precisely the matching map $H_n \to M_n^+ H$, which is assumed to be covering basal since $p : H \to y(U)$ is a Verdier hypercover.

Now suppose that $f : K \to L$ is in $C$, and let $ g : K \to M$ be an arbitrary morphism of simplicial sets. We want to show that the pushout map 
\begin{equation*}
    \begin{tikzcd}[ampersand replacement=\&]
	K \& L \\
	M \& N
	\arrow["f", from=1-1, to=1-2]
	\arrow["g"', from=1-1, to=2-1]
	\arrow[from=1-2, to=2-2]
	\arrow["{g_*(f)}"', from=2-1, to=2-2]
	\arrow["\lrcorner"{anchor=center, pos=0.125, rotate=180}, draw=none, from=2-2, to=1-1]
\end{tikzcd}
\end{equation*}
also belongs to $C$. As we saw in the proof of Lemma \ref{lem matching object is representable for basal hypercovers}, from this we obtain a pullback diagram
\begin{equation*}
    \begin{tikzcd}[ampersand replacement=\&]
	{H^N_+} \& {H^L_+} \\
	{H^M_+} \& {H^K_+}
	\arrow[from=1-1, to=1-2]
	\arrow["{H^{g_*(f)}_+}"', from=1-1, to=2-1]
	\arrow["\lrcorner"{anchor=center, pos=0.125}, draw=none, from=1-1, to=2-2]
	\arrow["{H^f_+}", from=1-2, to=2-2]
	\arrow["{H^g_+}"', from=2-1, to=2-2]
\end{tikzcd}
\end{equation*}
Now $H^f_+$ is covering basal by assumption, so by Lemma \ref{lem pulling back basal maps of presheaves}, so is $H^{g_*(f)}_+$, therefore $g_*(f) \in C$.

Now since $(\cat{C}, j)$ is a Verdier site, that means it is composition closed, which implies that $C$ is closed under composition as well.

Now every finite simplicial set $K$ can be obtained by finite composites of pushouts of boundary inclusions as in the proof of Lemma \ref{lem matching object is representable for basal hypercovers}. Hence every inclusion of a dimension at most $k$ subcomplex $L \subseteq K$ belongs to $C$. 

Thus the map $\varnothing \hookrightarrow \Delta^n$ belongs to $C$, and so the map $H_+^{\Delta^n} \to H_+^{\varnothing}$, isomorphic to the map $H_n \to y(U)$, is covering basal.

Now consider the codiagonal map $\Delta^n + \Delta^n \to \Delta^n$. By the axiom of Verdier sites (Definition \ref{def verdier site}), the induced diagonal map
\begin{equation*}
    H_n \cong H_+^{\Delta^n} \to H_+^{\Delta^n + \Delta^n} \cong H_n \times_{y(U)} H_n
\end{equation*}
is covering basal.

Hence the codiagonal map belongs to $C$. But every epimorphism $p : K \to L$ of finite simplicial sets can be obtained by a finite composition of codiagonal maps and pushouts. Indeed, we can build up any epimorphism $p : K \to L$ of finite simplicial sets by taking finitely many composites of pushouts of the following form. For every pair of $n$-simplices $\sigma, \sigma' \in K_n$ such that $p(\sigma) = p(\sigma')$, we take the pushout
\begin{equation*}
    \begin{tikzcd}[ampersand replacement=\&]
	{\Delta^n + \Delta^n} \& K \\
	{\Delta^n} \& {L_0}
	\arrow["{(\sigma, \sigma')}", from=1-1, to=1-2]
	\arrow[from=1-1, to=2-1]
	\arrow[from=1-2, to=2-2]
	\arrow["{p(\sigma)}"', from=2-1, to=2-2]
	\arrow["\lrcorner"{anchor=center, pos=0.125, rotate=180}, draw=none, from=2-2, to=1-1]
\end{tikzcd}
\end{equation*}
Thus every epimorphism belongs to $C$ and every monomorphism belongs to $C$. But every map of simplicial sets can be factored into an epimorphism followed by a monomorphism, hence $C$ consists of all maps between finite simplicial sets.
\end{proof}

Given an augmented simplicial presheaf $X$, let $L_n^+ X$ denote the presheaf of sets where $L_n^+ X(U)$ is the set of degenerate $n$-simplices in $X(U)$. Since maps of simplicial sets are completely determined by their action on nondegenerate simplices, if $ f: U \to V$ is a map in $\cat{C}$, we obtain a map $L_n^+ X(V) \to L_n^+ X(U)$ as a subobject of $X_n(V) \to X_n(U)$. We call $L_n^+ X$ the \textbf{$n$th augmented latching object} of $X$.

There is a canonical map $L_n^+ X \to X_n \to M_n^+ X$ of presheaves for any augmented simplicial presheaf $X$, given by including the degenerate $n$-simplices, and then considering their boundaries, see \cite[Section IV.3.2]{Goerss2009}.

\begin{Prop}[{\cite[Theorem 9.4]{Dugger2004}}] 
Let $p : H \to y(U)$ be a DHI-hypercover on a Verdier site $(\cat{C}, j)$. Then it can be refined by a split, Verdier hypercover.
\end{Prop}

\begin{proof}
Let us construct the refined split, Verdier hypercover $q : \widetilde{H} \to y(U)$ by induction.

In degree $0$, let us define $\widetilde{H}_0$ as follows. We know that $p_0 : H_0 \to M_0^+H$ is a local epimorphism. But $M_0^+H \cong y(U)$, and the matching map is precisely $p_0 : H_0 \to y(U)$. Since $p_0$ is a local epimorphism, there exists a $j$-covering family $r = \{r_i : U_i \to U \}_{i \in I}$ and for each $i \in I$ maps $s_i : y(U_i) \to H_0$ making the following diagram commute
\begin{equation*}
\begin{tikzcd}[ampersand replacement=\&]
	{y(U_i)} \& {H_0} \\
	{y(U)} \& {y(U)}
	\arrow["{s_i}",from=1-1, to=1-2]
	\arrow["{r_i}"', from=1-1, to=2-1]
	\arrow["{p_0}", from=1-2, to=2-2]
	\arrow[equals, from=2-1, to=2-2]
\end{tikzcd}
\end{equation*}
We set $q_0 : \widetilde{H}_0 \to y(U)$ to be $\sum_i y(U_i) \xrightarrow{\sum_i r_i} y(U)$, and the refinement $\widetilde{H}_0 \to H_0$ to be $\sum_i s_i : \sum_i y(U_i) \to H_0$.

Now suppose that we have an $n$-truncated, augmented, split, simplicial presheaf $q : \widetilde{H} \to y(U)$ that refines $p$ up to level $n$, and such that $\widetilde{H}_k \to M_k^+\widetilde{H}$ is covering basal for $k \leq n$. We want to now construct $\widetilde{H}_{n+1}$ and show that the resulting $(n+1)$-truncated hypercover satisfies the desired properties.

Consider the cospan
\begin{equation*}
\begin{tikzcd}[ampersand replacement=\&]
	\& {H_{n+1}} \\
	{M_{n+1}^+\widetilde{H}} \& {M_{n+1}^+H}
	\arrow["{m_{n+1}}", from=1-2, to=2-2]
	\arrow["{M_{n+1}^+(s)}"', from=2-1, to=2-2]
\end{tikzcd}    
\end{equation*}
The matching map $m_{n+1}$ is covering basal by assumption, and by Lemma \ref{lem matching object is representable for basal hypercovers}, $M_{n+1}^+ \widetilde{H}$ is semi-representable. Thus by Lemma \ref{lem pulling back basal maps of presheaves}, the pullback $\hat{H}_{n+1} = M^+_{n+1} \widetilde{H} \times_{M^+_{n+1}H} H_{n+1} \to M^+_{n+1}\widetilde{H}$ is covering basal.

Let $\widetilde{H}_{n+1} = \hat{H}_{n+1} + L_{n+1}^+ \widetilde{H}$, where $L_{n+1}^+ \widetilde{H}$ is the $(n+1)$-augmented latching object of $\widetilde{H}$. Thus $\widetilde{H}$ is a split, $(n+1)$-truncated simplicial presheaf.

We just need to show that $m: L_{n+1}^+ \widetilde{H} \to M^+_{n+1}\widetilde{H}$ is covering basal. First note that for every codegeneracy map $s^i: [n+1] \to [n]$ there is a corresponding map $s_i : \widetilde{H}_n \to L_{n+1}^+ \widetilde{H}$, which is just an inclusion map. Composing this with $m : L_{n+1}^+ \widetilde{H} \to M_{n+1}^+ \widetilde{H}$, we obtain $\widetilde{H}_n \to M_{n+1}^+\widetilde{H}$ which is precisely the map $\widetilde{H}_+^{s^i \iota} : \widetilde{H}_+^{\Delta^n} \to \widetilde{H}_+^{\partial \Delta^{n+1}}$ where $\iota :  \partial \Delta^{n+1} \hookrightarrow \Delta^{n+1}$ is the inclusion and $s^i : \Delta^{n+1} \to \Delta^n$ is the codegeneracy map. By Lemma \ref{lem building covering basal maps from maps of simplicial sets}, $\widetilde{H}^{s^i \iota}_+$ is covering basal. 

Thus, fixing some codegeneracy $s^i$, we have a commutative triangle
\begin{equation*}
\begin{tikzcd}[ampersand replacement=\&]
	{\widetilde{H}_n} \& {L_{n+1}^+\widetilde{H}} \\
	\& {M_{n+1}^+\widetilde{H}}
	\arrow["{s_i}", from=1-1, to=1-2]
	\arrow["{\widetilde{H}^{s^i\iota}_+}"', from=1-1, to=2-2]
	\arrow["m", from=1-2, to=2-2]
\end{tikzcd} 
\end{equation*}
where the left hand map is covering basal. Now we note that $L_{n+1}^+ \widetilde{H}$ is semi-representable, since $\widetilde{H}_n$ is, and if we let $\widetilde{H}_n \cong \sum_{k \in I_n} y(U_k)$, then
\begin{equation*}
 L_{n+1}^+ \widetilde{H} \cong \sum_{s^i : [n+1] \to [n]} \widetilde{H}_n \cong \sum_{(s^i, k)} y(U_k),
\end{equation*}
where the first sum is indexed over all codegeneracy maps. Thus for each fixed pair $(s^i, k \in I_n)$, the induced family $\mathbf{s_i}(s^i,k) = \{U_k \to U_k \}$ is just the identity map. Hence for each fixed index $\ell$ of the index set of $M_{n+1}^+\widetilde{H}$, the induced family $\mathbf{\widetilde{H}_+^{s^i \iota}}(\ell) = \{U_k \to U_\ell \}$ is just given by the corresponding induced family of $m$. In other words, since $\widetilde{H}_+^{s^i \iota}$ is covering basal, and $m$ shares the same induced families of morphisms, $m$ is also covering basal.

Thus both the maps $\hat{H}_{n+1} \to M_{n+1}^+ \widetilde{H}$ and $L_{n+1}^+ \to M_{n+1}^+ \widetilde{H}$ are covering basal, so the induced map
\begin{equation*}
    \widetilde{H}_{n+1} = \hat{H}_{n+1} + L_{n+1}^+ \widetilde{H} \rightarrow M_{n+1}^+ \widetilde{H}
\end{equation*}
is also covering basal. Thus we have constructed an $(n+1)$-truncated, augmented, split, Verdier hypercover $q_{\leq}: \widetilde{H}_{\leq n+1} \to y(U)$ refining $p_{\leq n+1} : H_{\leq n+1} \to y(U)$.

By induction, we construct an augmented, split, Verdier hypercover $q : \widetilde{H} \to y(U)$ refining $p : H \to y(U)$.
\end{proof}

\section{Matching Verdier Sites} \label{section matching verdier sites}

The condition in Definition \ref{def verdier site} that the site $(\cat{C}, j)$ be pullback stable is very strong. In the next sections we obtain weaker but more complicated conditions that a site can satisfy and still guarantee the existence of refinement by Verdier hypercovers, without pullback-stability. We call sites satisfying these conditions matching Verdier sites.

We should acknowledge that while matching Verdier sites appear to loosen the stricter requirements of Verdier sites, we have been unable to produce examples of sites that are matching Verdier but not Verdier. Hence we leave this section as a collection of partial results that may be of independent interest but do not tell a finished story.

Given an ($n$-truncated) augmented simplicial presheaf $X$, recall from Lemma \ref{lem characterization of matching object} the characterization of the matching object
\begin{equation*}
    M_k^+X \cong \sum_{a : \partial \Delta^k \to \text{ind}(X)} X(a).
\end{equation*}
We are going to investigate the structure of the matching object more closely.

Let $p : H \to y(U)$ denote an augmented, $n$-truncated, split simplicial presheaf, and consider the diagram $\tau : \nd(\Delta^k)^\op \to (\ncat{Pre}(\cat{C}) \downarrow y(U))$. Since the projection functor $(\ncat{Pre}(\cat{C}) \downarrow y(U)) \to \ncat{Pre}(\cat{C})$ reflects connected limits, we can ignore the maps down to $y(U)$ when $k > 1$. The resulting diagram will look like the barycentric subdivision of $\Delta^k$, with representables labeling every vertex, and basal maps labeling the face maps.

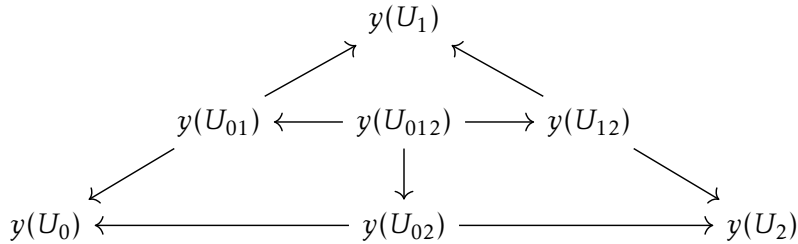
\begin{equation*}
    \begin{tikzcd}
	&& {y(U_1)} \\
	& {y(U_{01})} & {y(U_{012})} & {y(U_{12})} \\
	{y(U_0)} && {y(U_{02})} && {y(U_2)}
	\arrow[from=2-2, to=1-3]
	\arrow[from=2-2, to=3-1]
	\arrow[from=2-3, to=2-2]
	\arrow[from=2-3, to=2-4]
	\arrow[from=2-3, to=3-3]
	\arrow[from=2-4, to=1-3]
	\arrow[from=2-4, to=3-5]
	\arrow[from=3-3, to=3-1]
	\arrow[from=3-3, to=3-5]
\end{tikzcd}
\end{equation*}

Now in general, the face maps from the barycenter $y(U_{01 \dots k})$ form a cone over the boundary diagram $\partial \tau$. We call this the canonical cone of $\partial \tau$. In general, there is no reason for the canonical cone to be a limit cone. However, we wish to argue that if the diagram $\tau$ is degenerate and all the maps in the diagram are monomorphisms, then it is a limit cone. 

So suppose that $\tau: \Delta^k \to \text{ind}(\widetilde{H})$ is degenerate, then $\tau = s_i \theta$ for some simplex $\theta$, and we know that $d_i \tau = d_{i+1} \tau = \theta$, and every other face $d_\ell \tau$, $\ell \neq i, i+1$ will be degenerate on a face of $\theta$.

Below is an example of such a diagram for a degenerate $2$-simplex along with the canonical cone
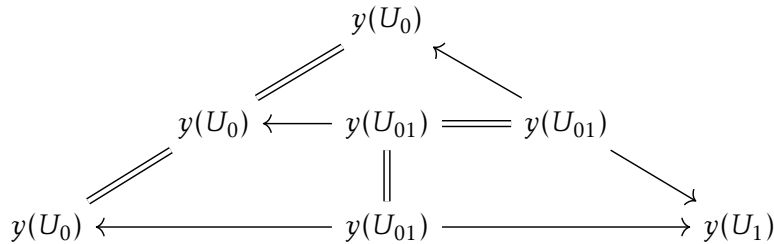
\begin{equation*}
\begin{tikzcd}
	&& {y(U_0)} \\
	& {y(U_0)} & {y(U_{01})} & {y(U_{01})} \\
	{y(U_0)} && {y(U_{01})} && {y(U_1)}
	\arrow[equals, from=2-2, to=1-3]
	\arrow[equals, from=2-2, to=3-1]
	\arrow[from=2-3, to=2-2]
	\arrow[equals, from=2-3, to=2-4]
	\arrow[equals, from=2-3, to=3-3]
	\arrow[from=2-4, to=1-3]
	\arrow[from=2-4, to=3-5]
	\arrow[from=3-3, to=3-1]
	\arrow[from=3-3, to=3-5]
\end{tikzcd}
\end{equation*}

Let $y(U_{d_\ell \tau})$ denote the representable at the barycenter of the face $d_\ell \tau$, so $y(U_{d_i \tau}) = y(U_{d_{i+1} \tau}) = y(U_\theta)$. Suppose that $Q$ is a cone over the boundary $\partial \tau$, with cone maps $\lambda_\ell : Q \to y(U_{d_\ell \tau})$. If $ \lambda_i : Q \to y(U_\theta)$ and $\lambda_{i+1} : Q \to y(U_\theta)$ denote the cone maps to those faces, then since $d_i \tau$ and $d_{i+1} \tau$ must share a common vertex $z$, labeled with representable $y(U_z)$, there is a monomorphism $f : y(U_\theta) \to y(U_z)$ such that $f \lambda_i = f \lambda_{i+1}$. Hence $\lambda_i = \lambda_{i+1} = \lambda$. Since all of the other faces $d_\ell \tau$ are degenerate on faces of $\theta$, this means that the cone map $\lambda_\ell : Q \to y(U_{d_\ell \tau})$ factors through $\lambda$, as at least one of the face maps $y(U_{d_\ell \tau}) \to y(U_{d_m d_\ell \tau})$ will be the identity map. 
\begin{equation*}
\begin{tikzcd}
	&& {y(U_{\ell})} & Q \\
	& {y(U_\ell)} & {y(U_{\theta})} & {y(U_\theta)} \\
	{y(U_{\ell})} && {y(U_\theta)} && {y(U_z)}
	\arrow["{\lambda_\ell}"{description}, from=1-4, to=2-2]
	\arrow["{\lambda_{i+1}}", from=1-4, to=2-4]
	\arrow["{\lambda_i}"{pos=0.7}, from=1-4, to=3-3]
	\arrow[equals, from=2-2, to=1-3]
	\arrow[equals, from=2-2, to=3-1]
	\arrow[from=2-3, to=2-2]
	\arrow[equals, from=2-3, to=2-4]
	\arrow[equals, from=2-3, to=3-3]
	\arrow[from=2-4, to=1-3]
	\arrow[from=2-4, to=3-5]
	\arrow[from=3-3, to=3-1]
	\arrow[from=3-3, to=3-5]
\end{tikzcd}
\end{equation*}
This means that the cone with apex $Q$ factors uniquely through the cone with apex $y(U_\theta)$. Hence $y(U_\theta)$ forms a limit cone over $\partial \tau$.

When $k = 1$, without loss of generality, assume that $\tau = s_0(x)$. Then the limit of $\tau$ is the pullback
\begin{equation*}
   \begin{tikzcd}
	{H(\tau)} & {y(U_0)} \\
	{y(U_0)} & {y(U)}
	\arrow[from=1-1, to=1-2]
	\arrow[from=1-1, to=2-1]
	\arrow["r", from=1-2, to=2-2]
	\arrow["r"', from=2-1, to=2-2]
\end{tikzcd} 
\end{equation*}
but $r$ is a basal map and therefore a monomorphism, hence $H(\tau) \to y(U_0)$ is isomorphic to the identity.

Hence we have proven the following result.

\begin{Lemma} \label{lem boundary of degen in matching is nondegen limit}
Let $(\cat{C}, j)$ be a site where all basal maps are monomorphisms, and let $q : H \to y(U)$ be a split, degreewise semi-representable, augmented, $n$-truncated simplicial presheaf. If $\tau : \Delta^k \to \text{ind}(H)$ is a degenerate simplex, with $\tau = \sigma(x)$, where $\sigma$ and $x$ are the unique corresponding surjection and non-degenerate simplex respectively, then
\begin{equation*}
   H(\partial \tau) \cong H(x).
\end{equation*}
\end{Lemma}

\begin{Def} \label{def matching verdier}
We say a site $(\cat{C}, j)$ is \textbf{matching Verdier} if
\begin{enumerate}
    \item for every $U \in \cat{C}$, every $n$-truncated, split Verdier hypercover $p : H \to y(U)$ and every $\tau : \Delta^{n+1} \to \text{ind}(H)$, the limit $X(\partial \tau)$ exists in $\cat{C}$, and
    \item every $j$-basal map $r_i : U_i \to U$ in $\cat{C}$ is a monomorphism.
\end{enumerate}
\end{Def}

\begin{Rem}
The first condition of Definition \ref{def matching verdier} guarantees that when we build up a Verdier hypercover by induction using Lemma \ref{lem refine local epi by covering basal map}, that the resulting matching objects are semi-representable and hence we can continue the induction. The second condition allows us to use Lemma \ref{lem boundary of degen in matching is nondegen limit}, with which we can use to attach latching objects.
\end{Rem}

\begin{Prop} \label{prop can refine hypercovers by Verdier hypercovers on matching Verdier site}
If $(\cat{C}, j)$ is a matching Verdier site and $p : H \to y(U)$ is a hypercover, then $p$ can be refined by a split Verdier hypercover.
\end{Prop}

\begin{proof}
We will construct a split Verdier hypercover $\widehat{p}: \widehat{H} \to y(U)$ by induction. First, by \cite{Low2014Notes}, we can assume that $p$ is a split DHI-hypercover. In degree $0$, consider the cospan
\begin{equation*}
    \begin{tikzcd}
	& {H_0} \\
	{y(U)} & {y(U)}
	\arrow["{p_0}", from=1-2, to=2-2]
	\arrow[equals, from=2-1, to=2-2]
\end{tikzcd}
\end{equation*}
Since $p_0 = m_0^+$ is a local epimorphism and $y(U)$ is representable, by Lemma \ref{lem refine local epi by covering basal map} there exists a covering basal map $\widehat{p}_0 : \widehat{H}_0 \to y(U)$ and a map $g_0 : \widehat{H}_0 \to H_0$ making the following diagram commute
\begin{equation*}
    \begin{tikzcd}
	{\widehat{H}_0} & {H_0} \\
	{y(U)} & {y(U)}
	\arrow["{g_0}", from=1-1, to=1-2]
	\arrow["{\widehat{p}_0}"', from=1-1, to=2-1]
	\arrow["{p_0}", from=1-2, to=2-2]
	\arrow[equals, from=2-1, to=2-2]
\end{tikzcd}
\end{equation*}
This is equivalent to a map $g : \widehat{H} \to \text{tr}_0 H$ of $0$-truncated simplicial presheaves over $y(U)$.

Now for $n > 0$, suppose that $\widehat{p} : \widehat{H} \to y(U)$ is a semi-representable, split, $n$-truncated, augmented simplicial presheaf equipped with a map $g : \widehat{H} \to \text{tr}_n H$ over $y(U)$. Consider the cospan
\begin{equation*}
\begin{tikzcd}
	& {H_{n+1}} \\
	{M_{n+1}^+ \widehat{H}} & {M_{n+1}^+ H}
	\arrow["{m_{n+1}^+}", from=1-2, to=2-2]
	\arrow["{M_{n+1}^+(g)}"', from=2-1, to=2-2]
\end{tikzcd}    
\end{equation*}
since $m_{n+1}^+$ is a local epimorphism and $M_{n+1}^+ \widehat{H}$ is semi-representable since $(\cat{C}, j)$ is a matching Verdier site, then again by Lemma \ref{lem refine local epi by covering basal map}, we obtain a covering basal map $\widehat{p}_{n+1} : W \to M_{n+1}^+ \widehat{H}$ and a map $g_{n+1} : W \to H_{n+1}$ making the following diagram commute
\begin{equation*}
    \begin{tikzcd}
	{W} & {H_{n+1}} \\
	{M_{n+1}^+ \widehat{H}} & {M_{n+1}^+ H}
	\arrow["{g_{n+1}}", from=1-1, to=1-2]
	\arrow["{\widehat{p}_{n+1}}"', from=1-1, to=2-1]
	\arrow["{m_{n+1}^+}", from=1-2, to=2-2]
	\arrow["{M_{n+1}^+(g)}"', from=2-1, to=2-2]
\end{tikzcd}
\end{equation*}
So we now wish to extend $\widehat{H}$ from an $n$-truncated simplicial presheaf to an $(n+1)$-truncated simplicial presheaf. So let $\widehat{H}_{n+1} = W + L_{n+1}^+ \widehat{H}$. Attaching the latching object allows us to define degeneracy maps $\sigma : \widehat{H}_k \to \widehat{H}_{n+1}$ for any surjection $\sigma : [n+1] \to [k]$ using the description from (\ref{eq structure of latching object}).

Let us define face maps $d_i : \widehat{H}_{n+1} \to \widehat{H}_n$ as follows. Let $\delta_i : M^+_{n+1} \widehat{H} \to \widehat{H}_n$ denote the map that precomposes a map $\partial \Delta^i \to \text{ind}(\widehat{H})$ with the map $d^i : \Delta^{i-1} \to \partial \Delta^i$. Now define $d_i$ on the component $W$ to be $d_i = \delta_i \widehat{p}_{n+1}$. On $L_{n+1}^+\widehat{H}$, $d_i$ is defined via the structure of faces of degenerate simplices. Hence $\widehat{H}$ is a well-defined $(n+1)$-truncated simplicial presheaf. We must now show that the matching map $\widehat{m}_{n+1}^+ : \widehat{H}_{n+1} \to M_{n+1}^+ \widehat{H}$ is covering basal. 

By construction, $\widehat{m}_{n+1}^+$ restricted to $W$ is $\widehat{p}_{n+1}$, which is covering basal. Hence we need only to show that $\widehat{m}_{n+1}^+$ restricted to $L_{n+1}^+ \widehat{H}$ is covering basal. So let $\tau : \Delta^{n+1} \to \text{ind}(\widehat{H})$ be degenerate, with $\tau = \sigma(x)$ for $\sigma$ and $x$ are the unique corresponding surjection and non-degenerate simplex, with $\sigma : [n+1] \to [k]$ and $x \in \text{ind}(\widehat{H})_k$, $k < n+1$.

Consider the representable $y(U_\tau)$ in component $\tau$ of $L_{n+1}^+ \widehat{H}$. By definition this is $y(U_x)$. Now the face maps of $\widehat{H}$ define maps $U_x = U_\tau \to U_{d_i \tau}$, forming a cone over the matching basal diagram $\tau : \nd(\Delta^{n+1})^\op \to \text{ind}(\widehat{H})$, hence defining a unique map to the limit $U_{\tau} \to \widehat{H}(\partial \tau)$. By Lemma \ref{lem boundary of degen in matching is nondegen limit}, $\widehat{H}(\partial \tau) \cong \widehat{H}(x)$. Hence the matching map $m_{n+1}^+$ restricted to $L_{n+1}^+ \widehat{H}$ is componentwise isomorphic to the identity map, and hence is covering basal.

Thus we have defined an $(n+1)$-truncated, split, Verdier hypercover $\widehat{H} \to y(U)$ which refines $\text{tr}_{n+1} H \to y(U)$. Continuing by induction we obtain a split Verdier hypercover refining $p : H \to y(U)$.
\end{proof}

\begin{Cor}
If $(\cat{C}, j)$ is a matching Verdier site, then $\hat{\H}(\cat{C}, j)$ is equal to the left Bousfield localization of $\H(\cat{C})$ at the Verdier hypercovers.
\end{Cor}

\begin{proof}
This follows from Proposition \ref{prop can refine hypercovers by Verdier hypercovers on matching Verdier site} and \cite[Theorem 6.2]{Dugger2004}.
\end{proof}

\section{Lurie's Lemma} \label{section luries lemma}
In this section we give some background on Lurie's lemma, a point-set topological result that is the core lemma we need for Theorem \ref{th many sites are strictly hypercomplete}.

\begin{Def} \label{def locally finite}
A collection $\closed = \{ C_i \}_{i \in I}$ of subsets of a topological space $X$ is \textbf{locally finite} if for each $x \in X$ there exists an open neighborhood $U_x$ of $x$ which intersects only finitely many elements of $\closed$, i.e. there exist $i_0, \dots, i_n \in I$ such that
\begin{equation*}
    U_x \cap \left( \bigcup_{i \in I} C_i \right) = U_x \cap \left( C_{i_0} \cup \dots \cup C_{i_n} \right).
\end{equation*}
\end{Def}

\begin{Lemma} \label{lem union of loc finite closed is closed}
If $\closed = \{ C_i \}_{i \in I}$ is a locally finite collection of closed subsets of $X$ then $\bigcup_{i \in I} \, C_i$ is closed in $X$.
\end{Lemma}
\begin{proof}
It suffices to show that the complement of the union is open. Suppose that $x \in \left( X - \cup_i C_{i} \right)$. Since $\closed$ is locally finite then there exists an open neighborhood $V_{x}$ of $x$ only intersecting finitely many elements of $\closed$, label them $C_{i_0}, \ldots, C_{i_n}$. For each $k \in \{0, \ldots, n\}$, since $C_{i_k}$ is closed we can find an open neighborhood $U_{i_k}$ of $x$ such that $U_{i_k} \subset (X - C_{i_k})$. Now $( \cap_{i=0}^n U_{i_k} ) \cap V_x$ is an open neighborhood of $x$ which is disjoint from each $C_i$. Thus $\cup_{i \in I} \, C_i$ is closed.
\end{proof}

Given a topological space $X$, a collection $\mathscr{S} = \{ S_i \}_{i \in I}$ of subsets of $X$, and a sub-indexing set $I_0 \subseteq I$, let
\begin{equation*}
    S(I_0) = \bigcap_{i \in I_0} S_i.
\end{equation*}
Given $k \geq 0$, and a subset $I_0 \subseteq I$ with cardinality $|I_0| = k+1$, we say that $I_0$ is a \textbf{$(k+1)$-subindex}. Let $I(k+1)$ denote the set of $(k+1)$-subindices of $I$.

\begin{Lemma}\label{lem collection of intersections is loc finite}
Given a topological space $X$, and a locally finite collection $\closed = \{ C_{i}\}_{i \in I}$ of closed subsets of $X$, then for every $k \geq 1$, the collection $\left\{ C(I_0) \right\}_{I_0 \in I(k + 1)}$ is a locally finite collection of closed subsets of $X$.
\end{Lemma}

\begin{proof}
Let $x \in X$, since $\closed$ is locally finite, there exists an open neighborhood $U_x$ of $x$ intersecting only finitely many $C_i$, label them $C_{i_0}, \ldots, C_{i_n}$. Fix $k \geq 1$. Then for $I_0 \in I(k + 1)$, the open set $U_x$ intersects $C(I_0)$ if and only if $I_0$ is a subset of $\{ i_0, \dots, i_n\}$; in other words $U_x$ only intersects $C(I_0)$ for finitely many $I_0$, as there are only finitely many subsets of $\{i_0, \dots, i_n \}$. Thus $\left\{ C(I_0) \right\}_{I_0 \in I(k + 1)}$ is a locally finite collection of closed subsets of $X$.
\end{proof}

\begin{Lemma}[{\cite[Lemma 39.1]{Munkres2000}}] \label{lem closure of locally finite collection}
Let $\mathscr{S} = \{ S_i \}_{i \in I}$ be a locally finite collection of subsets of $X$. Then 
\begin{enumerate}
\item Any subcollection $\mathscr{S}' \subseteq \mathscr{S}$ is locally finite,
\item The collection $\overline{\mathscr{S}} = \{ \overline{S_i} \}_{i \in I}$ is locally finite, where $\overline{S_i}$ is the closure of the subset $S_i \subseteq X$, and
\item $\overline{\bigcup\limits_{i \in I} S_i} = \bigcup\limits_{i \in I} \overline{S_i}$.
\end{enumerate}
\end{Lemma}

\begin{Def} \label{def regular, normal top space}
We say that a topological space $X$ is \textbf{regular} if for every $x \in X$ and every closed subset $C \subset X$, there exist disjoint open neighborhoods $U_x, U_C$ of $x$ and $C$ respectively. We say that $X$ is \textbf{normal} if for every pair of closed subsets $C_1, C_2 \subset X$, there exist disjoint open neighborhoods $U_1$, $U_2$ of $C_1$, $C_2$ respectively.
\end{Def}

\begin{Lemma}[{\cite[Theorem 41.1]{Munkres2000}}]
If $X$ is a paracompact, Hausdorff topological space, then $X$ is regular and, moreover, normal. 
\end{Lemma}

\begin{comment}
\begin{proof}
First regularity is proven. Let $a \in X$ and $B \subset X$ be closed and disjoint from $a$. For each $b \in B$, since $X$ is Hausdorff, choose an open neighborhood $U_b$ of $b$ whose closure is disjoint from $a$. Cover $X$ with $\{U_b\}_{b \in B} \cup \{ X - B \}$ and then by paracompactness take a locally finite open refinement $\mathscr C$ covering $X$. Next form the subcollection $\mathscr D \subset \mathscr C$ of elements which have nonempty intersection with $B$; which now covers $B$. The closure of each open subset in $\mathscr D$ also is disjoint from $a$. Now define $W = \bigcup\limits_{V \in \mathscr D} V$ which is an open set containing $B$ but by the previous lemma we have $\overline{W} =  \bigcup\limits_{V \in \mathscr D} \overline{V}$ and so $a \notin \overline{W}$ which means that $a$ and $B$ are separated by $U_a = X - \overline{W}$ and $U_B = W$. To further prove normality, repeat the proof for $a$ replaced by a closed set $A$ and instead of using Hausdorff use regularity.
\end{proof}
\end{comment}

Recall that if $X$ is a topological space, and $\mathscr{U} = \{U_i \}_{i \in I}$, $\mathscr{V} = \{ V_j \}_{j \in J}$ are open covers of $X$, then we say that $\mathscr{V}$ is a \textbf{refinement} of $\mathscr{U}$, and write $\mathscr{V} \leq \mathscr{U}$ if there exists a function $\pi : J \to I$ such that $V_j \subseteq U_{\pi(j)}$. We say that a refinement $\mathscr{V} \leq \mathscr{U}$ is \textbf{index-constant} if $\mathscr{V} = \{V_i \}_{i \in I}$ and $\mathscr{U} = \{U_i \}_{i \in I}$ have the same index set, and $\pi : I \to I$ is the identity function. 

\begin{Lemma}[{\cite[Lemma 41.6]{Munkres2000}}] \label{lem shrinking lemma}
Let $X$ be a paracompact Hausdorff topological space and let $\mathscr{U} = \{ U_i \}_{i \in I}$ be an open cover of $X$. Then there exists a locally finite, index-constant refinement $\mathscr{V} = \{V_i \}_{i \in I}$ of $\mathscr{U}$ such that $\overline{V_i} \subset U_{i}$ for each $i \in I$.
\end{Lemma}

We call the next result \textbf{Lurie's Lemma}. It will be pivotal in understanding hypercovers on paracompact Hausdorff spaces.

\begin{Lemma}[{\cite[Lemma 7.2.3.5]{Lurie2009}}] \label{lurie's lemma}
Let $X$ be a paracompact Hausdorff space with an open cover $\mathscr{U} = \{U_i \}_{i \in I}$. Fix $k \geq 1$, and suppose that for every $(k+1)$-subindex $I_0 \subseteq I$ there is an open cover $\mathscr{V}_{I_0} = \{ V_b \}_{b \in B(I_0)}$ of the intersection $U(I_0)$. Then there exists an open cover $\mathscr{W} = \{W_{\overline{i}}\}_{\overline{i} \in \overline{I}}$ of $X$, which is a refinement of $\mathscr{U}$ under $\pi: \overline{I} \to I$ satisfying the following property: If $\overline{I}_0$ is a $(k+1)$-subindex of $\overline{I}$ with $\pi(\overline{I}_0) = I_0$, then there exists a $b \in B(I_0)$ such that $W(\overline{I}_0) \subseteq V_b$.
\end{Lemma}

\begin{proof}
By Lemma \ref{lem shrinking lemma}, there exists a locally finite, index-constant refinement $\mathscr{U}' = \{U'_{i} \}_{i \in I}$ of $\mathscr{U}$ such that for each $i \in I$, $\overline{U'_i} \subset U_i$. The key property we will use here is that for every $(k+1)$-subindex $I_0 \subseteq I$ we have $\overline{U'}(I_0) \subseteq U(I_0)$, and therefore $\overline{U'}(I_0)$ is covered by $\{ V_b \}_{b \in B(I_0)}$.

Now note that the collection $\{ \overline{U'_i} \}_{i \in I}$ is locally finite by Lemma \ref{lem closure of locally finite collection}. Let $I(k+1)$ denote the set of $(k+1)$-subindexes of $I$. For each $I_0 \in I(k+1)$, consider $\overline{U'}(I_0) = \cap_{i \in I_0} \overline{U'_i}$. Note that the collection $\{\overline{U'}(I_0)\}_{I_0 \in I(k+1)}$ is a locally finite collection of closed subsets by Lemma \ref{lem collection of intersections is loc finite}. Let us now define a new index set $\overline{I}$ as
$$\overline{I} = \{ (i, I_0, b) \mid I_0 \in I(k+1), \,  i \in I_0 \text{ and } b \in B(I_0) \} + I.$$
Let $\pi: \overline{I} \to I$ be the function defined by $\pi(i, I_0, b) = i$ and $\pi(i) = i$. The reasoning for this indexing set will hopefully become clear by the end of the proof. Given $i \in I$, let $I(k+1, i)$ denote the subset of $I(k+1)$ consisting of those $(k+1)$-subindices $I_1$ of $I$ that contain $i$. Define $W_{\overline{i}}$ for each $\overline{i} \in \overline{I}$ as follows:
\begin{equation*}
    W_{\overline{i}} = \left( U'_i - \bigcup_{I_1 \in I(k+1, i)} \overline{U'}(I_1) \right) \cup (\widetilde{V}_b \cap U'_i),
\end{equation*}
where if $\overline{i} = (i, I_0, b)$, then $\widetilde{V}_b = V_b$, and if $\overline{i} = i \in I$, then $\widetilde{V}_b = U'_i$.

Now by Lemma \ref{lem closure of locally finite collection}, the collection $\{ \overline{U'}(I_1) \}_{I_1 \in I(k+1, i)}$ is locally finite, since it is a subcollection of $\{ \overline{U'}(I_1) \}_{I_1 \in I(k+1)}$. Therefore, by Lemma \ref{lem union of loc finite closed is closed}, $\cup_{I_1 \in I(k+1, i)} \overline{U'}(I_1)$ is closed. 

Now if $\pi(\overline{i}) = i$, then $W_{\overline{i}} \subseteq U'_i \subseteq U_i$, so to check that $\mathscr{W} = \{ W_{\overline{i}}\}$ is a refinement of $\mathscr{U}$, it is sufficient to show that it is a cover of $X$. Suppose that $x \in X$, then since $\mathscr{U}$ is an open cover, there exists a $U'_{i}$ which contains $x$. 

Now if $x \notin \bigcup_{I_1 \in I(k+1, i)} \overline{U'}(I_1)$, then $x \in W_{i}$. On the other hand, if there exists an $I_1 \in I(k+1)$ such that $x \in \overline{U'}(I_1)$, then since $\overline{U'}(I_1) \subseteq U(I_1)$, and $\{V_b \}_{b \in B(I_1)}$ is a cover of $U(I_1)$, $x \in (V_b \cap U'_i)$ for some $b \in B(I_1)$ and therefore $x \in W_{(i, I_1, b)}$. Therefore $\mathscr{W}$ is an open cover refining $\mathscr{U}$.

Now we need only show that $\mathscr{W}$ satisfies the stated property in the lemma. Suppose that $\overline{I}_0$ is a $(k+1)$-subindex of $\overline{I}$ with $\pi(\overline{I}_0) = I_0$, then we want to show that there exists a $b \in B(I_0)$ such that $W(\overline{I}_0) \subseteq V_b$.

We have
\begin{equation*}
 W(\overline{I}_0) = \bigcap_{\ell = 0}^k W_{\overline{i}_\ell} = \bigcap_{\ell = 0}^k \left( U'_{i_\ell} - \bigcup_{I_1^\ell \in I(k+1, i_\ell)} \overline{U'}(I^\ell_1) \right) \cup \bigcap_{\ell = 0}^k (\widetilde{V}_{b_\ell} \cap U'_{i_\ell} ).
\end{equation*}

Now note that 
\begin{equation*}
    \bigcap_{\ell = 0}^k \left( U'_{i_\ell} - \bigcup_{I_1^\ell \in I(k+1, i_\ell)} \overline{U'}(I^\ell_1) \right) = \bigcap_{\ell = 0}^k U'_{i_\ell} - \bigcap_{\ell = 0}^k \, \bigcup_{I^\ell_1 \in I(k+1, i_\ell)} \overline{U'}(I^\ell_1) = \varnothing,
\end{equation*}
because $\cap_{\ell = 0}^k U'_{i_\ell} = U'(I_0)$, and $I_0 \in I(k+1, i_\ell)$ for each $0 \leq \ell \leq k$.

Therefore
\begin{equation*}
   W(\overline{I}_0) = \bigcap\limits_{\ell = 0}^k W_{\overline{i}_\ell} = \bigcap_{\ell = 0}^k (\widetilde{V}_{b_\ell} \cap U'_{i_\ell})
\end{equation*}
Now if each $\overline{i}_\ell \in I$, then $W(\overline{I}_0) = \varnothing$, and hence $W(\overline{I}_0) \subseteq V_b$ for every $b \in B(I_0)$. So assume that at least one $\overline{i}_\ell = (i_\ell, I_0^\ell, b_\ell)$. Then
\begin{equation*}
    W(\overline{I}_0) = \bigcap_{\ell = 0}^k (\widetilde{V}_{b_\ell} \cap U'_{i_\ell}) \subseteq V_{b_{\ell_0}}
\end{equation*}
for some $0 \leq \ell_0 \leq k$. Thus $\mathscr{W} = \{W_{\overline{i}} \}_{\overline{i} \in \overline{I}}$ is a refinement of $\mathscr{U}$ with the desired property.
\end{proof}

\begin{Ex} \label{ex for luries proof}
The previous proof should become more clear with the following example. Consider the open cover $\mathscr{U} = \{U_0, U_1 \}$ of $\R$ given by $U_0 = (-\infty, 1)$ and $U_1 = (-1, \infty)$. Thus $I = \{0, 1\}$, and for $k = 1$, there is only $1$ $(k+1)$-subindex of $I$, given by $I$ itself. Now let us choose a refinement $\mathscr{U'} = \{U'_0, U'_1 \}$ with $U'_0 = (-\infty, 1/2)$ and $U'_1 = (-1/2, \infty)$, and clearly $\overline{U}_0 \subseteq U_0$ and $\overline{U}_1 \subseteq U_1$. Let $\mathscr{V} = \{ V_0, V_1 \}$ denote the open cover of $U_0 \cap U_1 = (-1, 1)$ given by $V_0 = (-1, 1/2)$ and $V_1 = (-1/2, 1)$. Thus $B(I) = \{0, 1 \}$. So we have $\overline{I} = \{ (0, I, 0), (1, I, 0), (0,I,1), (1,I,1)\}$. Now $\overline{U'_0} \cap \overline{U'_1} = [-1/2,1/2]$, so we get
\begin{itemize}
    \item $W_{(0,I,0)} = (U'_0 - [-1/2,1/2]) \cup (V_0 \cap U'_0) = (-\infty, 1/2)$,
    \item $W_{(0,I,1)} = (U'_0 - [-1/2,1/2]) \cup (V_1 \cap U'_0) = (-\infty, -1/2) \cup (-1/2, 1/2)$,
    \item $W_{(1,I,0)} = (U_1 - [-1,1]) \cup (V_0 \cap U_1) = (-1/2, 1/2) \cup (1/2, \infty)$,
    \item $W_{(1,I,1)} = (U_1 - [-1,1]) \cup (V_1 \cap U_1) = (-1/2, \infty)$,
    \item $W_0 = (U_0 - [-1,1]) = (-\infty, -1/2)$,
    \item $W_1 = (U_1 - [-1,1]) = (1/2, \infty)$.
\end{itemize}
Now let us check all the $(k+1)$-subindices $\overline{I}_0 \subseteq \overline{I}$ such that $\pi(\overline{I}_0) = I$. These are
\begin{itemize}
\item $\{ (0, I, 0), (1, I, 0) \}$, $\{ (0, I, 0), (1, I, 1) \}$, $\{ (0, I, 1), (1, I, 0) \}$, $\{(0, I, 1), (1, I, 1) \}$, with intersection $W(\overline{I}_0) = (-1/2, 1/2)$,
\item  $\{ (0, I, 0), 1 \}$, $\{ (0, I, 1), 1 \}$, $\{ (1, I, 0), 0 \}$, $\{ (1, I, 1), 0 \}$, $\{ 0, 1 \}$, with intersection $W(\overline{I}_0) = \varnothing$.
\end{itemize}
Thus for each such $\overline{I}_0$, $W(\overline{I}_0)$ is a subset of $V_0$ or $V_1$.
\end{Ex}

\section{Homotopical Categories} \label{section homotopical categories}

In this section we briefly review the different homotopical structures that play a role in Section \ref{section plus construction}.

\begin{Def} \label{def homotopical category}
Given a category $\cat{C}$ and a class $W \subseteq \text{Mor}(\cat{C})$ of morphisms we call \textbf{weak equivalences}, we say that the pair $(\cat{C}, W)$ is a \textbf{homotopical category} if $W$ satisfies the \textbf{$2$-of-$6$ property}: given morphisms $f : U \to V$, $g: V \to W$, $h : W \to X$, if $hg$ and $gf$ are weak equivalences, then so are $f$, $g$, $h$, and $hgf$.

If $(\cat{C}, W)$ is a homotopical category with $U,V \in \cat{C}$, we say that $U$ and $V$ are \textbf{weak equivalent} if there exists a finite zig-zag of weak equivalences 
\begin{equation*}
    U \xleftarrow{w_0} U_0 \xrightarrow{w_1} U_1 \leftarrow \dots \xrightarrow{w_n} V 
\end{equation*}
where the $w_i \in W$. We will write $U \simeq V$ to mean they are weak equivalent.
\end{Def}

Given a homotopical category $(\cat{C}, W)$, let $L_W \cat{C}$ denote its Hammock localization. This is a simplicial category (a category enriched in $\ncat{sSet}$), where the simplicial sets $\u{L_W \cat{C}}(X,Y)$ are constructed by creating ``hammocks'' from the morphisms in $\cat{C}$, which are a generalization of the zig-zags of the $1$-localization $\cat{C}[W^{-1}]$, see \cite[Section 2]{Dwyer1980} for more details on its construction. Dwyer and Kan prove that Hammock localization satisfy the following properties.

\begin{Lemma}[{\cite[Proposition 3.1]{Dwyer1980}}]
Given a homotopical category $(\cat{C}, W)$, with $X,Y \in \cat{C}$, then if we let $\pi_0 L_W \cat{C}$ denote the category with the same objects as $\cat{C}$ and whose sets of morphisms are given by $[\pi_0 L_W \cat{C}](X,Y) \coloneqq \pi_0 [ \u{L_W \cat{C}}(X,Y)]$, then there is an equivalence of categories
\begin{equation*}
    \pi_0 L_W \cat{C} \simeq \cat{C}[W^{-1}].
\end{equation*}
\end{Lemma}

\begin{Lemma}[{\cite[Prop 3.3]{Dwyer1980}}]
Given a homotopical category $(\cat{C}, W)$, with objects $X,Y,Z \in \cat{C}$ and $f: X \to Y$ a weak equivalence, then the induced maps
$$\u{L_W\cat{C}}(Z,f) : \u{L_W \cat{C}}(Z,X) \to \u{L_W \cat{C}}(Z,Y), \qquad \u{L_W \cat{C}}(f,Z) : \u{L_W \cat{C}}(Y,Z) \to \u{L_W \cat{C}}(X,Z)$$
are weak equivalences of simplicial sets.
\end{Lemma}

However the Hammock localization is useless in practice, as its construction is extremely complicated. To obtain a smaller model for the derived mapping space, one can use a category of fibrant objects structure.

\begin{Def}[{\cite{Brown1973}}] \label{def cat of fibrant objects}
A category of fibrant objects consists of a category $\cat{C}$ with finite products equipped with two classes of maps $\ncat{Weak}$ and $\ncat{Fib}$, called weak equivalences and fibrations respectively, (we call the intersection $\ncat{Fib} \cap \ncat{Weak}$ trivial fibrations) such that:
\begin{enumerate}
	\item $\ncat{Weak}$ and $\ncat{Fib}$ contain all isomorphisms of $\cat{C}$,
	\item $\ncat{Weak}$ satisfies the $2$-of-$3$ property, and $\ncat{Fib}$ is closed under composition,
	\item fibrations and trivial fibrations are stable under pullback,
	\item there exist \textbf{path objects} in $\cat{C}$, i.e. for every object $X \in \cat{C}$, there exists an object $X^I$ such that the diagonal $X \xrightarrow{\Delta} X \times X$ factors as a weak equivalence followed by a fibration:
\begin{equation*}
	\begin{tikzcd}
	& {X^I} \\
	X && {X \times X}
	\arrow["w", from=2-1, to=1-2]
	\arrow["f", from=1-2, to=2-3]
	\arrow["{\Delta = (1_X, 1_X)}"', from=2-1, to=2-3]
\end{tikzcd}
\end{equation*}
Note that path objects are not necessarily unique, nor is the factorization.
\item every object is fibrant, i.e. the unique map $X \to *$ is a fibration.
\end{enumerate}
\end{Def}

\begin{Rem}
The full subcategory on the fibrant objects of a model category form a category of fibrant objects.
\end{Rem}

Low augments this definition by adding simplicial enrichment to it. Namely a \textbf{simplicial category of fibrant objects} is a simplicially enriched category, whose underlying category is a category of fibrant objects, with some compatibility requirements \cite[Definition 5.1]{Low2015}.

\begin{Th}[{\cite[Prop 6.2]{Low2015}, \cite[Proposition 4.33]{Jardine2015}}] \label{Th loc fib simplicial presheaves form cat of fib objects}
Given a site $(\cat{C}, j)$, the category\footnote{There are set-theoretic issues here, see Remark \ref{rem size considerations for low's theorem}.} of locally fibrant simplicial presheaves forms a simplicial category of fibrant objects, where the the fibrations are the local fibrations and the weak equivalences are the local weak equivalences.
\end{Th}

\begin{Rem}
Notice that the above category of fibrant objects structure is not induced from a model structure on the category of simplicial presheaves.
\end{Rem}

The improved mapping space for two objects $X$ and $Y$ in a category of fibrant objects $C$ is given as follows. 

\begin{Def} \label{def cocycle category}
Given two objects $X,Y$ in a category of fibrant objects, a cocycle between them is a span:
\begin{equation*}
    \begin{tikzcd}
	X & A & Y
	\arrow["{\sim}"', two heads, from=1-2, to=1-1]
	\arrow[from=1-2, to=1-3]
\end{tikzcd}
\end{equation*}
where the left hand map is a trivial fibration. A map of two cocycles is a diagram:
\begin{equation*}
    \begin{tikzcd}
	& A \\
	X && Y \\
	& B
	\arrow["{\sim}"',two heads, from=1-2, to=2-1]
	\arrow[from=1-2, to=2-3]
	\arrow["{\sim}", two heads, from=3-2, to=2-1]
	\arrow[from=3-2, to=2-3]
	\arrow["{\sim}", from=1-2, to=3-2]
\end{tikzcd}
\end{equation*}
where the vertical map is a weak equivalence. Let $\ncat{Cocycle}(X,Y)$ denote the category of cocycles between $X$ and $Y$.
\end{Def}

The category of cocycles is homotopical in the following sense.

\begin{Prop}[{\cite[Corollary 2.11]{Low2015}}]\label{prop cocycle homs are homotopical}
If $X,Y,Z$ are objects in a simplicial category of fibrant objects, and $f: X \to Y$ is a weak equivalence, then the induced maps
$$N \ncat{Cocycle}(Z,X) \xrightarrow{f_*} N\ncat{Cocycle}(Z,Y), \qquad N \ncat{Cocycle}(Y,Z) \xrightarrow{f^*} N\ncat{Cocycle}(X,Z)$$
are weak equivalences of simplicial sets.
\end{Prop}

\begin{Prop}[{\cite[Theorem 3.12]{Low2015}}]
Given a a category of fibrant objects $\cat{C}$, we can consider the hammock localization of its underlying homotopical category, and there exists a weak equivalence:
\begin{equation}
    N \ncat{Cocycle}(X,Y) \simeq \u{L_W \cat{C}}(X,Y).
\end{equation}
\end{Prop}

Thus for a category of fibrant objects, its derived mapping space can be computed as a nerve of a category, which is easier to work with than the hammock localization.

To obtain the most useful model for the derived mapping space of a homotopical category, one must use a simplicial model category. We refer to \cite[Chapter 9]{Hirschhorn2009} for the definition of a simplicial model category and for the necessary background in model category theory.

\begin{Def}
Given a simplicial model category $\cat{M}$, we let
\begin{equation*}
    \R \cat{M}(X,Y) \coloneqq \u{\cat{M}}(Q(X), R(Y)),
\end{equation*}
denote the \textbf{derived mapping space} of $X$ and $Y$, where $Q$ and $R$ are fixed cofibrant and fibrant replacement functors for $\cat{M}$, respectively. We let $\R \cat{M}_{Q,R} : \cat{M}^\op \times \cat{M} \to \ncat{sSet}$ denote the corresponding functor, when we wish to make the (co)fibrant replacement functors explicit.
\end{Def}

\begin{Prop}
If $f: X \to X'$ and $g : Y \to Y'$ are weak equivalences in a simplicial model category $\cat{M}$, then
\begin{equation*}
    \R\cat{M}(f,Y) : \R\cat{M}(X',Y) \to \R\cat{M}(X,Y), \qquad \text{and} \qquad \R\cat{M}(X,g): \R\cat{M}(X,Y) \to \R\cat{M}(X,Y')
\end{equation*}
are homotopy equivalences of Kan complexes.
\end{Prop}

\section{Some History} \label{section history}

In this section, we provide our understanding of the history of ideas that led to the study of hypercovers in differential geometry\footnote{Warning, the authors are not historians.}. Let us start with the notion of a sheaf of abelian groups on a topological space. Sheaves were first conceived of by Jean Leray while a prisoner of war in Austria during World War II \cite{Miller2000}. It did not take long for the French algebraic geometers including Serre, Grothendieck and Verdier to master the theory of sheaves and push it to new heights, culminating in the massively influential works \'{E}l\'{e}ments de G\'{e}om\'{e}trie Alg\'{e}brique (EGA) \cite{Grothendieck1960EGA}, Fondements de la G\'{e}ometrie Alg\'{e}brique (FGA) \cite{Grothendieck1957FGA} and S\'{e}minaire de G\'{e}om\'{e}trie Alg\'{e}brique du Bois-Marie (SGA) \cite{Grothendieck1960SGA}\footnote{See \href{https://thosgood.net/translations/}{Tim Hosgood's Translation Webpage} where there are partial translations of EGA, SGA and a full translation of FGA into English.}. It is in the last of these, SGA, that Verdier proves his now famous theorem, the Verdier Hypecovering Theorem \cite[Expose V, Section 7.4]{Grothendieck1960SGA}. Typically one computes the cohomology of a sheaf $A$ of abelian groups on a topological space $X$, using an injective resolution of $A$. Verdier, however, computes sheaf cohomology by instead ``resolving'' $X$ in terms of hypercovers. One can then connect sheaf cohomology to \v{C}ech cohomology via a spectral sequence using hypercovers. Hypercovers thus became a useful tool for algebraic geometers.

Given an arbitrary topological space $X$, and a sheaf $A$ of abelian groups, there is a canonical map
\begin{equation*}
    \check{H}^n(X,A) \to H^n(X,A),
\end{equation*}
from the $n$-th \v{C}ech cohomology of $X$ to the $n$th sheaf cohomology of $X$ with values in $A$. In degrees $0$ and $1$ this map is an isomorphism, and in degree $2$ it is injective. However, if $X$ is a Hausdorff, paracompact topological space, then the above map is an isomorphism for all $n \geq 0$ and all abelian sheaves $A$. This was proven by Godemont as Theorem 5.10.1 in \cite{Godement1958}\footnote{But see \cite[Section 13.3]{Gallier2022} for a modern, readable discussion and proof.}.

In 1973, Kenneth Brown published ``Abstract Homotopy Theory and Generalized Sheaf Cohomology'' \cite{Brown1973} where he constructed a category of fibrant objects structure (Definition \ref{def cat of fibrant objects}) on the category of locally fibrant simplicial presheaves, with weak equivalences the local weak equivalences (Definition \ref{def j local weak equiv}), and fibrations the local fibrations (Definition \ref{def j local fibration}). Using this structure, he showed how one could compute abelian sheaf cohomology of a topological space using homotopical machinery. In modern parlance, we would say that Brown computes abelian sheaf cohomology using the connected components of the derived mapping space $\R \Hom(X,Y)$
between locally fibrant simplicial presheaves
\begin{equation*}
    H^n(X,A) \cong \pi_0 \R \Hom(X, K(A,n)).
\end{equation*}
Brown then proved a slightly more general\footnote{Brown proves Verdier's theorem using local trivial fibrations of locally fibrant simplicial presheaves, whereas Verdier asks for a stricter notion of hypercover, what we call a Verdier hypercover in this paper, see Definition \ref{def DHI-hypercover, basal and covering basal} for different distinctions.} version of Verdier's theorem with this machinery, which can be stated as an isomorphism
\begin{equation*}
\pi_0 \R \Hom(X, K(A,n)) \cong \ncolim{p : H \to X \in [\ncat{Triv}_{X}]} [H, K(A,n)],
\end{equation*}
see \cite[Section 3]{Brown1973} for more details. Also see the article \cite{Jardine2012} by Jardine for an elegant, modern proof of Verdier's theorem using the category of fibrant objects structure. Also in 1973, Brown-Gerstein \cite{Brown1973A} constructed a model structure on simplicial sheaves over a topological space $X$, now called the local projective model structure, with weak equivalences the local weak equivalences and cofibrations defined as the projective cofibrations (Proposition \ref{prop projective model structure}). In other words, they define the left Bousfield localization of the projective model structure by the local weak equivalences, of course before the term was invented.

In 1984, Joyal \cite{Joyal1984} constructed a model structure on simplicial sheaves over an arbitrary site, whose weak equivalences were local weak equivalences and whose cofibrations are monomorphisms. Soon after, Jardine \cite{Jardine1987} constructed a very similar model structure on the category of simplicial presheaves over a site, whose weak equivalences are also the local weak equivalences. These model structures are nowadays referred to as the local injective model structure on simplicial sheaves and simplicial presheaves, respectively, over a site. 

In 1998, Hirschowitz and Simpson \cite{Hirschowitz1998} wrote a more than 200 page preprint in French detailing a theory of $n$-stacks using model structures on presheaves of Segal spaces, where a version of \v{C}ech descent is used.

In 1998, Dugger \cite{Dugger1998} wrote a set of unfinished notes ``Sheaves and Homotopy Theory," where he introduced the \v{C}ech projective model structure and the $\R$-local model structure\footnote{Which would later be re-understood and clarified in the papers \cite{Bunke2013, Bunk2022}.} and left open the problem of showing that the \v{C}ech weak equivalences are precisely the local weak equivalences on the site of smooth manifolds. 

In 2001, Blander \cite{Blander2001} constructed the local projective model structure on simplicial presheaves. Since the identity map is a Quillen equivalence between the injective and projective versions of all of the model structures we will discuss here, we will consider only the projective model structure.

In 2004, Toen-Vezzosi \cite{Toen2004} used a simplicially enriched version of the local projective model structures on categories of simplicial presheaves, but now the underlying site is allowed to be a simplicially enriched site. They based their paper on the unpublished work of Rezk on homotopy toposes \cite{Rezk2010}. Their paper on model toposes served as a foundation for their later book \cite{Toen2008} on homotopical algebraic geometry. In the same year, Dugger-Hollander-Isaksen \cite{Dugger2004} gave an alternate construction of the local projective model structure. They showed that it is a left Bousfield localization of the corresponding projective model structure on simplicial presheaves at the class of maps $ p : H \to y(U)$ of simplicial presheaves that are local trivial fibrations in the sense of Brown and where $H$ is degreewise representable (Definition \ref{def DHI-hypercover, basal and covering basal}). They also constructed the left Bousfield localization of the projective model structure at the maps $p : \check{C}(r) \to y(U)$ of projections of \v{C}ech nerves. This is precisely the \v{C}ech projective model structure that Dugger introduced in \cite{Dugger1998}. Also in 2004, Tibor Beke \cite{Beke2004} proved an analogue of Verdier's theorem for truncated hypercovers.

In 2009, Lurie published his landmark book \cite{Lurie2009}, in which he developed the theory of $\infty$-categories. One of the great achievements of this work was to connect the homotopy theory of simplicial (pre)sheaves as discussed above with his new notion of $\infty$-sheaf. He proved that the \v{C}ech projective model structure presented an $\infty$-category equivalent to his own notion of the corresponding $\infty$-category of $\infty$-sheaves. He also proved that the local projective model structure presented the \textit{hypercompletion} of the $\infty$-category of $\infty$-sheaves. In other words, Lurie was able to argue that the \v{C}ech model structure was the more natural structure, but in some situations the $\infty$-category of $\infty$-sheaves was equivalent to its hypercompletion, i.e. the \v{C}ech and local model structures were Quillen equivalent (by a zig-zag of Quillen equivalences). 

At this point, higher sheaf theory had taken off as a subject, and we can hardly hope to cover its development. However, let us mention the early work of Fiorenza-Schreiber-Stasheff \cite{Fiorenza2011} and Freed-Hopkins \cite{Freed2013} where the above mentioned simplicial sheaf technology is used in the context of differential geometry to develop differential cohomology. A key issue, most clearly emphasized by Schreiber \cite{Schreiber2013}, is that the derived mapping Hom for simplicial presheaves should be thought of as an enhanced, generalized form of cohomology. In the context of differential geometry, $\infty$-sheaves would be the coefficient objects of generalized notions of sheaf cohomology, and they could be manipulated like homotopy types. One could then go back and forth between cocycle descriptions of geometric structures and the global geometric structures represented by $\infty$-sheafification, i.e. \v{C}ech fibrant replacement. Hence, from this formalism, to develop various exotic cohomology theories, it was necessary to be able to work with the \v{C}ech model structure, and understand fibrant replacement. 

Little work went into this direction, with many preferring to work entirely with $\infty$-categories, until in 2014, when Zhen Lin Low asked on MathOverflow \cite{Low2014MO} if there was a generalization of Verdier's Hypercovering theorem, in the sense that one could find a formula for the fibrant replacement of a locally fibrant simplicial presheaf in the local projective model structure. After a discussion with Rezk on the website, Low was able to solve his problem and recorded his solution in his paper \cite{Low2015}, the subject of which is Section \ref{section plus construction}.
\printbibliography

\end{document}